\RequirePackage[l2tabu,orthodox]{nag}
\documentclass[reqno,12pt,a4paper,oneside]{amsart}
\pdfoutput=1
\usepackage
           {geometry}
\usepackage{ifxetex,ifluatex}
\newif\ifxetexorluatex
\ifxetex
  \xetexorluatextrue
\else
  \ifluatex
    \xetexorluatextrue
  \else
    \xetexorluatexfalse
  \fi
\fi

\ifxetexorluatex
  \usepackage{fontspec}           
  \usepackage{unicode-math}       
  \usepackage{polyglossia}        
  \setmainlanguage{english}
\else
  \usepackage[utf8]{inputenc}     
  \usepackage[english]{babel}     

  \usepackage{amsfonts}
  \usepackage{amssymb}
  
  \usepackage{lmodern}
  \usepackage[T1]{fontenc}
\fi
\usepackage{booktabs}   
\usepackage[inline]{enumitem} 
\usepackage{xspace}     
\usepackage{xcolor}     
\usepackage{amsmath}    
\usepackage{mathtools}  
\usepackage{stmaryrd} 

\usepackage{graphics}

\definecolor{linkblue}{RGB}{1,1,190}
\definecolor{citegreen}{RGB}{1,190,1}
\usepackage[linkcolor=linkblue
           ,citecolor=citegreen
           ,ocgcolorlinks        
           ,bookmarksopen=True,  
           ]{hyperref}
\makeatletter
  \AtBeginDocument{
    \hypersetup{
      pdftitle  = {\@title},
      pdfauthor = {\authors}     
    }
  }
\makeatother

\usepackage{amsthm}

\usepackage[capitalize    
           ]{cleveref}

\theoremstyle{definition}
\newtheorem {defi}          {Definition}[section]

\theoremstyle{plain}
\newtheorem {thm}     [defi]{Theorem}
\crefname   {thm}{Theorem}{Theorems}
\newtheorem*{thm*}          {Theorem}
\newtheorem {lemma}   [defi]{Lemma}
\newtheorem {prop}    [defi]{Proposition}

\newtheorem {cor}     [defi]{Corollary}

\theoremstyle{remark}
\newtheorem {remark}  [defi]{Remark}
\newtheorem {exm}     [defi]{Example}
\newtheorem*{exm*}          {Example}
\crefname   {exm}{Example}{Examples}

\crefname{equation}{Equation}{Equations}
\Crefname{equation}{Equation}{Equations}
\crefname{figure}  {Figure}{Figures}
\Crefname{figure}  {Figure}{Figures}
\makeatletter

\newcommand{\sc@lettershortcut}[3]{%
  \expandafter\providecommand\csname #2#3\endcsname{#1{#3}}%
}

\newcommand{\sc@shortcuts}[3]{%
  \count@=0
  \loop
  \advance\count@ 1
  \edef\tmp@{%
    \noexpand\sc@lettershortcut\unexpanded{{#1}}{#2}{#3\count@}
  }
  \tmp@
  \ifnum\count@<26
  \repeat
}

\newcommand{\defshortcuts}[2]{\sc@shortcuts{#1}{#2}{\@alph}}

\newcommand{\defShortcuts}[2]{\sc@shortcuts{#1}{#2}{\@Alph}}

\makeatother
\let\sc\undefined
\defShortcuts{\mathbb}{b}
\defShortcuts{\mathcal}{c}
\defShortcuts{\mathfrak}{f}
\defShortcuts{\mathsf}{s}
\defshortcuts{\mathfrak}{f}
\defshortcuts{\mathsf}{s}

\def\rfop{*}
\makeatletter
\newcommand\rigidfactorization[2][]{%
  \def\rf@delim{\rfop}
  \newif\ifrf@notfirst
  #1
  \@for\next:=#2\do{%
    \ifrf@notfirst
      \rf@delim
    \fi
    \rf@notfirsttrue
    \next
  }%
}
\makeatother
\newcommand\rf\rigidfactorization
\newcommand{\quo}{\mathsf{q}}                     

\ifxetexorluatex
  
\else
  
\fi
\providecommand{\val}{\mathsf{v}}                 

\newcommand{\isomto}{\overset{\sim}{\rightarrow}} 

\DeclarePairedDelimiter{\length}{\lvert}{\rvert}
\DeclarePairedDelimiter{\abs}{\lvert}{\rvert}
\DeclarePairedDelimiter{\card}{\lvert}{\rvert}

\DeclareMathOperator{\End}{End}

\DeclareMathOperator{\Hom}{Hom}
\DeclareMathOperator{\Ext}{Ext}

\DeclareMathOperator{\ann}{ann}

\DeclareMathOperator{\id}{id}

\DeclareMathOperator{\nr}{nr}

\DeclareMathOperator{\modspec}{modspec}

\DeclareMathOperator{\udim}{udim}

\DeclareMathOperator{\im}{im}

\setlist[enumerate,1]{label=\textup{(\arabic*)}, ref=\textup{(}\arabic*\textup{)}, leftmargin=0.75cm}
\setlist[enumerate,2]{label=\textup{(}\roman*\textup{)}, ref=\textup{(}\roman*\textup{)}}

\newlist{equivenumerate}{enumerate}{1}
\setlist[equivenumerate,1]{%
  label=\textup{(\alph*)},
  ref=\textup{(}\alph*\textup{)},
  leftmargin=0.75cm
}

\newlist{equivenumerate*}{enumerate*}{1}
\setlist*[equivenumerate*,1]{%
  label=\textup{(\alph*)},
  ref=\textup{(}\alph*\textup{)},
  leftmargin=0.75cm
}

\newlist{propenumerate}{enumerate}{1}
\setlist[propenumerate,1]{%
  label=\textup{(\roman*)},
  ref=\textup{(}\roman*\textup{)},
  leftmargin=0.75cm
}

\newenvironment{claim}[1][]{%
\begin{description}[leftmargin=0pt]%
\ifthenelse{\isempty{#1}}{\item[Claim]}{\item[Claim {#1}]}
}%
{\end{description}}

\makeatletter
\def\sr@stripleadingcol::#1{#1}
\def\sr@dosubref#1#2:#3 #4{\if\relax#3\relax%
  \def\first{\sr@stripleadingcol #4}%
  #1{\first}\ref{\first:#2}%
\else%
  \sr@dosubref#1#3 {#4:#2}%
\fi}%

\newcommand{\subref}[1]{\sr@dosubref\cref#1: :\relax}
\newcommand{\Subref}[1]{\sr@dosubref\Cref#1: :\relax}
\makeatother
\AtBeginDocument{\renewcommand{\setminus}{\smallsetminus}}

\usepackage{tikz-cd}
\usepackage{wrapfig}
\usepackage{xifthen}
\usepackage{rotating}

\usepackage{microtype}

\title[Factorizations in bounded HNP rings]{Factorizations in bounded hereditary Noetherian prime rings}
\author[D.~Smertnig]{Daniel Smertnig}

\address{University of Graz\\
         NAWI Graz\\
         Institute for Mathematics and Scientific Computing\\
         Heinrichstra\ss e 36\\
         8010 Graz, Austria}
\email{daniel.smertnig@uni-graz.at}
\thanks{The author was supported by the Austrian Science Fund (FWF) project P26036-N26.}

\ifxetexorluatex
  \newcommand{\delspc}{\!\!}
\else
  \newcommand{\delspc}{\!}
\fi

\newcommand{\fmon}[1]{\mathcal{F^*}({#1})}            
\newcommand{\famon}[1]{\mathcal{F}_{\delspc M}({#1})} 
\newcommand{\fagrp}[1]{\mathcal{F}_{\delspc G}({#1})} 
\newcommand{\res}[1]{\overline{#1}}                   

\DeclareMathOperator{\ttop}{top}
\DeclareMathOperator{\tbottom}{bottom}
\newcommand{\simple}[1]{\cS({#1})}                
\newcommand{\towers}[1]{\cT({#1})}                
\newcommand{\ctowers}[1]{\cT_{\delspc\cC}({#1})}  
\newcommand{\ftowers}[1]{\cT_{\delspc\cF}({#1})}  
\newcommand{\cp}[1]{\cP({#1})}                    
\newcommand{\qcp}[1]{\langle\cp{#1}\rangle}       
\newcommand{\pcg}[1]{G({#1})}                     
\newcommand{\pcgm}[1]{G_{\textup{max}}({#1})}     
\newcommand{\stzcls}[1]{\mathsf S(#1)}            
\newcommand{\gcg}[1]{\operatorname{gcg}({#1})}    
\newcommand{\gcgm}[1]{\operatorname{gcg}_{\textup{max}}({#1})}
\newcommand{\cgrp}[1]{\mathcal C({#1})}           
\newcommand{\cgrpm}[1]{\mathcal C_{\textup{max}}({#1})}
\newcommand{\modfl}{\operatorname{\textup{\textbf{mod}}_{fl}}}

\makeatletter
 \def\@textbottom{\vskip \z@ \@plus 6pt}
 \let\@texttop\relax
\makeatother


\begin{document}

\begin{abstract}
  If $H$ is a monoid and $a=u_1 \cdots u_k \in H$ with atoms (irreducible elements) $u_1, \ldots, u_k$, then $k$ is a length of $a$, the set of lengths of $a$ is denoted by $\mathsf L(a)$, and $\mathcal L(H)=\{\,\mathsf L (a) \mid a \in H \,\}$ is the system of sets of lengths of $H$.
  Let $R$ be a hereditary Noetherian prime (HNP) ring.
  Then every element of the monoid of non-zero-divisors $R^\bullet$ can be written as a product of atoms.
  We show that, if $R$ is bounded and every stably free right $R$-ideal is free, then there exists a transfer homomorphism from $R^{\bullet}$ to the monoid $B$ of zero-sum sequences over a subset $G_{\textup{max}}(R)$ of the ideal class group $G(R)$.
  This implies that the systems of sets of lengths, together with further arithmetical invariants, of the monoids $R^{\bullet}$ and $B$ coincide.
  It is well-known that commutative Dedekind domains allow transfer homomorphisms to monoids of zero-sum sequences, and the arithmetic of the latter has been the object of much research.
  Our approach is based on the structure theory of finitely generated projective modules over HNP rings, as established in the recent monograph by Levy and Robson.
  We complement our results by giving an example of a non-bounded HNP ring in which every stably free right $R$-ideal is free but which does not allow a transfer homomorphism to a monoid of zero-sum sequences over any subset of its ideal class group.
\end{abstract}

\maketitle

\section{Introduction}
\label{sec:introduction}

In a Noetherian ring, every non-zero-divisor that is not a unit can be written as a finite product of atoms (irreducible elements).
However, usually such a factorization is not unique.
Arithmetical invariants, such as sets of lengths, describe this non-uniqueness: If $a \in R^\bullet$ and $a=u_1\cdots u_k$ with atoms $u_1$,~$\ldots\,$,~$u_k$, then $k$ is a length of $a$, the \emph{set of lengths} of $a$ is denoted by $\sL(a)$, and $\cL(R^\bullet) = \{\,\sL(a) \mid a \in R^\bullet \,\}$ is the \emph{system of sets of lengths}.
If there exists an element $a$ with $\card{\sL(a)} > 1$, then $\card{\sL(a^n)} > n$ for every $n \in \mathbb N$.
Therefore the system of sets of lengths either consists of singletons, so that $\cL(R^\bullet) = \{\, \{n\} \mid n \in \bN_0 \,\}$ (except in the trivial case where $R^\bullet$ is a group), or the cardinalities of its elements are unbounded.

The principal aim of factorization theory is to describe the various phenomena of non-uniqueness of factorizations by suitable arithmetical invariants (such as sets of lengths), and to study the interaction of these arithmetical invariants with classical algebraic invariants of the underlying ring (such as class groups).
Factorization theory has its origins in algebraic number theory and in commutative algebra.
The focus has been on commutative Noetherian domains, commutative Krull domains and monoids, their monoids of ideals, and their monoids of modules.
We refer to \cite{andersondd97,narkiewicz04,chapman05,ghk06,geroldinger09,baeth-wiegand13,facchini12,leuschke-wiegand12,baeth-wiegand13,fontana-houston-lucas13,chapman-fontana-geroldinger-olberding16} for recent monographs, conference proceedings, and surveys on the topic.

A key strategy, lying at the very heart of factorization theory, is the construction of a transfer homomorphism from a ring or monoid of interest to a simpler one.
This allows one to study the arithmetic of the simpler object and then pull back the information to the original object of interest.
An exposition of transfer principles in the commutative setting can be found in \cite[Section 3.2]{ghk06}.
The best investigated class of simpler objects occurring in this context is that of monoids of zero-sum sequences over subsets of abelian groups.
These monoids are commutative Krull monoids having a combinatorial flavor; questions about factorizations are reduced to questions about zero-sum sequences, which are studied with methods from additive combinatorics.
We refer the reader to the survey \cite{geroldinger09} for the interplay between additive combinatorics and the arithmetic of monoids.

Until relatively recently, the study of factorizations of elements in noncommutative rings had its emphasis on a variety of concepts of factoriality, some of which are finding continuing applications in ring theory (see for example \cite{launois-lenagan-rigal06,goodearl-yakimov14}).
In the last couple of years, first steps have been made to describe the non-uniqueness of factorizations in noncommutative rings by extending the machinery successfully used in the commutative setting.
We refer to \cite[Section 5]{smertnig16a} for a summary of results that were obtained in \cite{estes-matijevic79a,estes-matijevic79b, estes91b,baeth-ponomarenko-etal11,geroldinger13,smertnig13,bachman-baeth-gossell14,baeth-smertnig15,bell-heinle-levandovskyy17}, and to \cite{geroldinger16} for a recent survey on sets of lengths that presents many rings and monoids having transfer homomorphisms to monoids of zero-sum sequences and also incorporates the noncommutative point of view.

A central result of the commutative theory implies that every commutative Dedekind domain possesses a transfer homomorphism to a monoid of zero-sum sequences over a subset $G_P$ of its ideal class group $G$.
Here, $G_P$ consists of all classes containing nonzero prime ideals.
Many arithmetical invariants, in particular sets of lengths but also more refined arithmetical invariants such as catenary degrees, are (essentially) preserved by this transfer homomorphism.
By means of multiplicative ideal theory in monoids, a similar transfer homomorphism can be constructed more generally for commutative Krull monoids.

Since multiplicative ideal theory has played a key role in the commutative setting, it is natural that the first extensions of the method to the noncommutative setting have followed the same path:
In \cite{geroldinger13}, it has been shown that a transfer homomorphism exists for normalizing Krull monoids.
A further generalization to \emph{arithmetical maximal orders} was undertaken in \cite{smertnig13,baeth-smertnig15}; the existence of a transfer homomorphism was shown under a (sufficient) condition that becomes trivial in the cases considered before.
Again, the transfer homomorphism is to a monoid of zero-sum sequences over a subset $C_M$ of an abelian group $C$.

Natural noncommutative generalizations of commutative Dedekind domains are \emph{hereditary Noetherian prime \textup{(}HNP\textup{)} rings} and \emph{Dedekind prime rings} (rings such that every nonzero submodule of a left or right progenerator is a progenerator).
Every HNP ring is an order in its simple Artinian ring of quotients; a ring is a Dedekind prime ring if and only if it is an HNP ring and a maximal order.
A ring is \emph{bounded} if every essential left or right ideal contains a nonzero two-sided ideal.
In this paper we extend the transfer result for commutative Dedekind domains to bounded HNP rings (under a sufficient condition).

For Dedekind prime rings there exists a well-developed multiplicative ideal theory for two-sided as well as one-sided ideals, originating with pioneering work of Asano.
In particular, bounded Dedekind prime rings are arithmetical maximal orders, and hence the results of \cite{smertnig13,baeth-smertnig15} are applicable in principle.
However, to make effective use of these results it is necessary to understand the sufficient condition appearing in the transfer result for arithmetical maximal orders, as well as $C_M$ and $C$, in terms of more natural invariants of Dedekind prime rings.

In \cite{smertnig13}, this was done in the more special setting where $R$ is a classical maximal order over a holomorphy ring $\cO$ in a central simple algebra over a global field.
In this case the sufficient condition for the existence of a transfer homomorphism is for every stably free right $R$-ideal to be free.
The group $C$ is isomorphic to a ray class group of $\cO$, and $C=C_M$.
Moreover, when $\cO$ is a ring of algebraic integers, the condition for the existence of the transfer homomorphism is not only sufficient but also necessary by \cite[Theorem 1.2]{smertnig13}.
If the condition fails, several arithmetical invariants, which are finite otherwise, are infinite.

For the more general class of HNP rings, the multiplicative ideal theory of two-sided ideals has been investigated.
We refer the reader to \cite[\S22]{levy-robson11}, and also mention \cite{rump01,akalan-marubayashi16,rump-yang16} for samples of recent progress in noncommutative multiplicative ideal theory.
However, a one-sided ideal theory seems not to have been developed.

In this paper, in lieu of an ideal-theoretic approach, our method for constructing a transfer homomorphism is module-theoretic in nature.
We make extensive use of the structure theory of finitely generated projective modules over HNP rings, which can be viewed as a far-reaching generalization of Steinitz's theorem.
This theory was developed over the last decades, chiefly by Eisenbud, Levy, and Robson, and is presented in the monograph \cite{levy-robson11}.

We obtain a transfer homomorphism for bounded HNP rings in which every stably free right $R$-ideal is free.
(In fact, the method works a little bit beyond the bounded case.)
This is the main result of the present paper, and it is given in \cref{thm-transfer,thm-catenary,cor-cat}.
It implies that results on the system of sets of lengths and catenary degrees in monoids of zero-sum sequences carry over to bounded HNP rings, as long as every stably free right ideal is free.

In \cref{sec-background} we recall the main results on finitely generated projective modules over HNP rings, as they are used in the present paper.
The necessary basic notions from factorization theory are also recalled.

In \cref{sec-class-groups}, for an HNP ring $R$, we define a class group $\cgrp{R}$ as a subquotient of $K_0\modfl(R)$, where $\modfl(R)$ denotes the category of (right) $R$-modules of finite length.
If $R$ is commutative, it is easily seen that $\cgrp{R}$ is isomorphic to the ideal class group as it is traditionally defined.
The main result in this section is \cref{t-k0-iso-pcg}, which shows that $\cgrp{R}$ is isomorphic to the ideal class group $\pcg{R}$ as defined in \cite{levy-robson11} as a direct summand of $K_0(R)$.
Moreover, we show that the distinguished subset $\pcgm{R} \subset \pcg{R}$, which appears in the construction of a transfer homomorphism, is preserved under Morita equivalence and passage to a Dedekind right closure.
The results of \cref{sec-class-groups} hold for all HNP rings; no additional restrictions (such as boundedness) are imposed.

\Cref{sec-transfer-hom} contains the main results of the paper.
We construct a transfer homomorphism for bounded HNP rings in which every stably free right $R$-ideal is free in \cref{thm-transfer}.
The natural class group to use in the construction is $\cgrp{R}$, but by the results from the previous section it is isomorphic to $\pcg{R}$.
The possible existence of non-trivial cycle towers complicates matters compared to the special case of bounded Dedekind prime rings; the combinatorial \cref{lemma-comb} is crucial.
In the case of classical hereditary orders, some of the results can also be derived from earlier work of Estes \cite{estes91b} together with results on congruence monoids; see \subref{r-main:est}.
Catenary degrees require additional work and are dealt with in \cref{ssec:catenary}.

Bounded Dedekind prime rings form the class of rings in the intersection between arithmetical maximal orders and HNP rings.
Thus, to such rings the results from \cref{sec-transfer-hom} as well as the results for arithmetical maximal orders in \cite{smertnig13,baeth-smertnig15} apply.
In \cref{sec-dedekind}, we show how the results of \cite{smertnig13,baeth-smertnig15} can be applied to deduce some of the conclusions in \cref{sec-transfer-hom} in the special case of bounded Dedekind prime rings.
This section depends on the results from \cref{sec-class-groups} but not on those from \cref{sec-transfer-hom}.

We remark again that the approach in the present paper (for HNP rings) and the one in \cite{smertnig13,baeth-smertnig15} (for arithmetical maximal orders) are somewhat different.
The former is ring- and module-theoretic in nature and restricted to dimension $1$.
(A similar approach, studying factorizations of $a \in R^\bullet$ by means of the module $R/aR$, has also been used in \cite{facchini-smertnig-nguyen14}.)
The latter pursues a monoid- and ideal-theoretic viewpoint, through the Brandt groupoid, and, while not being limited to dimension $1$, is limited to maximal orders.
However, for bounded Dedekind prime rings, in the end both approaches yield essentially the same transfer homomorphism.

Finally, in \cref{sec-examples}, we give some examples showing different behavior in non-bounded HNP rings.
For instance, if $R$ is a bounded HNP ring in which every stably free right $R$-ideal is free and $\pcg{R}$ is trivial, then $R$ is half-factorial, that is, the length of a factorization of an element is uniquely determined.
(This is a trivial consequence of the results in \cref{sec-transfer-hom}.)
In \cref{p-nonhf}, we show that, if the condition `bounded' is dropped, there exist counterexamples.
In particular, for these examples, there does not exist a transfer homomorphism to a monoid of zero-sum sequences over a subset of the class group.
In \cref{e-no-ft}, we construct an explicit such example in the $2 \times 2$-matrix ring over a basic idealizer of the first Weyl algebra.

\emph{Throughout the paper, let $R$ be a hereditary noetherian prime \textup{(}HNP\textup{)} ring.
To avoid trivial cases, we assume that $R$ is non-Artinian.}

\section{Background and Notation}
\label{sec-background}

By $\bN_0$ we denote the set of nonnegative integers, and by $\bN$ the set of positive integers.
The symbol $\subset$ denotes an inclusion of sets that is not necessarily proper.
We set $[a,b]=\{\, x \in \bZ \mid a \le x \le b \,\}$ for $a$,~$b \in \bZ$.
By a \emph{monoid} we mean a cancellative semigroup with identity.
As a general reference for noncommutative Noetherian rings we use \cite{mcconnell-robson01}, for HNP rings \cite{levy-robson11}.

An HNP ring $R$ has a simple Artinian quotient ring $\quo(R)$.
To avoid trivial cases, we assume that $R$ is not Artinian, that is, $R \ne \quo(R)$.
By $\quo(R)^\times$ we denote the unit group of $\quo(R)$.
For a subset $X \subset \quo(R)$ we write $X^\bullet$ for the subset of non-zero-divisors of $\quo(R)$ contained in $X$.
We recall that $\quo(R)^\bullet=\quo(R)^\times$ and hence $X^\bullet = X \cap \quo(R)^\times$.
The monoid $R^\bullet=R\cap \quo(R)^\times$ is the multiplicative monoid of all non-zero-divisors of $R$.
In particular, a non-zero-divisor of $R$ remains a non-zero-divisor in $\quo(R)$.

In noncommutative rings, the behavior of zero-divisors can be quite pathological.
However, for HNP rings this is not the case.
For $a \in R$ the following are equivalent:
\begin{equivenumerate*}
  \item $a$ is a zero-divisor
  \item $a$ is a left zero-divisor
  \item $a$ is a right zero-divisor.
\end{equivenumerate*}
(This is a consequence of $R$ being a prime Goldie ring.)
It follows that every multiple of a zero-divisor is a zero-divisor.
Consequently, every left or right divisor of a non-zero-divisor is a non-zero-divisor.
Thus, $R^\bullet$, as a submonoid of $R$, is closed under taking left or right divisors.
For $a$,~$b \in R^\bullet$, we have $aR \subset bR$ if and only if $aR^\bullet \subset bR^\bullet$.
Moreover, if $a \in R^\bullet$ and $b \in R$ with $aR=bR$, then also $b \in R^\bullet$.

A right $R$-submodule $I \subset \quo(R)$ is a \emph{fractional right $R$-ideal} if there exist $x$,~$y \in \quo(R)^\times$ such that $x \in I$ and $yI \subset R$.
It is a \emph{right $R$-ideal} if moreover $I \subset R$.
In other words, a right $R$-ideal is a right ideal of $R$ that contains a non-zero-divisor.
A right ideal $I$ of $R$ contains a non-zero-divisor if and only if it is an essential submodule of $R$, which in turn is equivalent to $\udim I=\udim R$.
The ring $R$ is \emph{right bounded} if every right $R$-ideal contains a nonzero (two-sided) ideal of $R$.
It is \emph{bounded} if it is left and right bounded.

By an ($R$-)module, without further qualification, we mean a right $R$-module.
We will make use of the theory of finitely generated projective modules over an HNP ring, as given by Levy and Robson in \cite{levy-robson11}.
Recall that two projective modules $P$ and $Q$ are \emph{stably isomorphic} if there exists an $n \in \bN_0$ such that $P \oplus R^n \cong Q \oplus R^n$.
Levy and Robson give a description of the stable isomorphism classes of finitely generated projective modules by means of two independent invariants: the \emph{Steinitz class}, which is an element of an abelian group, and the \emph{genus}.
The genus is a vector of nonnegative integers, defined in terms of the isomorphism classes of simple modules of $R$.

We briefly recall the definition of the genus.
Let $V$ and $W$ be two simple modules, and let $(V)$ and $(W)$ denote their isomorphism classes.
If $\Ext_R^1(V,W) \ne \mathbf 0$, then $W$ is a \emph{successor} of $V$, and $V$ is a \emph{predecessor} of $W$.
Every unfaithful simple module has a unique predecessor (up to isomorphism).
If a simple module $V$ has an unfaithful successor $W$, then $W$ is the unique successor of $V$ up to isomorphism.
We tacitly apply the terminology of predecessors, successors, and so on to simple modules as well as to isomorphism classes of simple modules.
Let $(V)^+$ denote the unique unfaithful successor of $(V)$, if it exists.

\begin{defi}
An ($R$-)\emph{tower} of length $n \in \bN$ is a finite sequence $(W_1)$, $\ldots\,$,~$(W_n)$ of isomorphism classes of simple modules such that
\begin{propenumerate}
  \item\label{tower:distinct} the $(W_i)$ are pairwise distinct,
  \item\label{tower:succ} $(W_i)$ is the unique unfaithful successor of $(W_{i-1})$ for all $i \in [2,n]$,
  \item\label{tower:max} the sequence $(W_1)$, $\ldots\,$,~$(W_n)$ is maximal with respect to \ref*{tower:distinct} and \ref*{tower:succ}.
\end{propenumerate}
\end{defi}

If $T$ is a tower, we write $\length{T}$ for its length.
A tower is \emph{trivial} if its length is $1$.
Towers take one of the following two forms:
\begin{propenumerate}
  \item In a \emph{faithful tower}, $W_1$ is faithful, while $W_2$, $\ldots\,$,~$W_n$ are unfaithful.
    The last module, $W_n$, has no unfaithful successor.
    Thus, the tower is linearly ordered by the successor relationship.
  \item In a \emph{cycle tower}, $W_1$ is the unfaithful successor of $W_n$.
    Thus, all the modules in the tower are unfaithful, and the tower is cyclically ordered by the successor relationship.
    Note that a cyclic permutation of the tower again gives a cycle tower.
    We will consider two cycle towers to be the same if they are cyclic permutations of each other.
\end{propenumerate}

The (isomorphism class of the) module $W_1$ is the \emph{top} of the tower, while $W_n$ is the \emph{base} of the tower.
For a cycle tower this depends on the arbitrary choice of starting point of the enumeration of the cyclically ordered set.

With the convention that cyclic permutations of cycle towers are considered to be the same tower, every isomorphism class of a simple module is contained in a unique tower.

Let $P$ be a finitely generated projective module.
If $W$ is an unfaithful simple module, then $M=\ann_R(W)$ is a maximal ideal of $R$ and $R/M$ is a simple Artinian ring.
Thus, the $R/M$-module $P/PM$ has finite length.
We define the \emph{rank of $P$ at $W$}, denoted by $\rho(P,W)$, to be the length of $P/PM$.
For the zero module, we define $\rho(P,\mathbf 0) = \udim P$.
Let $\modspec(R)$ denote a set of representatives of the isomorphism classes of unfaithful simple modules, together with the zero module.
We use the notation $e_{(W)}$ for the vector in $\bN_0^{\modspec{R}}$ which has a $1$ in the coordinate corresponding to the isomorphism class $(W)$ and zeroes everywhere else.
Then
\[
\Psi(P) = \sum_{W \in \modspec(R)} \rho(P,W) e_{(W)} \in \bN_0^{\modspec(R)}
\]
is the \emph{genus of $P$}.
If $T$ is a tower, we define
\[
\rho(P,T) = \sum_{\substack{(W) \in T \\ \text{$W$ unfaithful}}} \rho(P,W).
\]

The genus of a nonzero finitely generated projective module $P$ has two properties:
\begin{propenumerate}
  \item It has \emph{almost standard rank}, that is, for all but finitely many $W \in \modspec(R)$, it holds that
    \[
    \rho(P,W) = \frac{\udim P}{\udim R} \rho(R,W).
    \]
  \item It has \emph{cycle standard rank}, that is, for all cycle towers $T$, it holds that
    \[
    \rho(P,T) = \frac{\udim P}{\udim R} \rho(R,T).
    \]
\end{propenumerate}
The genus can take arbitrary values subject to these two conditions and $\udim P > 0$.

The value of $\Psi(P)$ only depends on the stable isomorphism class of $P$.
Moreover, $\Psi(P \oplus Q) = \Psi(P) + \Psi(Q)$ for finitely generated projective modules $P$ and $Q$.
Thus, the genus extends to a group homomorphism $\Psi^+$ from the Grothendieck group $K_0(R)$ to a direct product of copies of $\bZ$.
We define the \emph{ideal class group} $\pcg{R}$ of $R$ as $\pcg{R}=\ker(\Psi^+)$.
If $P$ is a finitely generated projective module, we write $[P]$ for its class in $K_0(R)$.

A \emph{base point set} $\cB$ for $R$ is a set consisting of exactly one finitely generated projective module in each genus of nonzero finitely generated projective modules that is closed under direct sums (up to isomorphism).
Given a base point set $\cB$, to each nonzero finitely generated projective module $P$ we can associate a \emph{Steinitz class} $\stzcls{P} = [P] - [\cB(P)] \in \pcg{R}$.
Here, $\cB(P)$ denotes the unique module in $\cB$ for which $\Psi(\cB(P))=\Psi(P)$.
We set $\stzcls{\mathbf 0}=\mathbf 0$.
The definition of the Steinitz class depends on the choice of the base point set, which is not canonical in general.
We will always assume a fixed but unspecified choice of base point set.
\begin{thm}[{\cite[Theorem 35.12 (Main Structure Theorem)]{levy-robson11}}] \label{t-main-structure}
  \mbox{}
  \begin{enumerate}
    \item `Steinitz class' and `genus' are independent invariants of nonzero finitely generated projective $R$-modules.
    \item These invariants are additive on direct sums.
    \item\label{tms:iso} Together, they determine the stable isomorphism class of the module and, if it has uniform dimension 2 or more, its isomorphism class.
  \end{enumerate}
\end{thm}
Moreover, $K_0(R) \cong \pcg{R} \times \im(\Psi^+)$ by \cite[Corollary 35.17]{levy-robson11}.

\begin{remark}
$R$ is a Dedekind prime ring if and only if all towers are trivial.
Then cycle standard rank forces $\Psi(P)$ to be completely determined by $\udim P$, and $\pcg{R}$ is just the usual ideal class group, defined as $\ker(\udim\colon K_0(R) \to \mathbb Z)$.
(It is also called \emph{locally free class group} in the setting of orders in central simple algebras.)
We recover the well-known result that the stable isomorphism class of a nonzero finitely generated projective module over a Dedekind prime ring is uniquely determined by its ideal class and its uniform dimension (rank).
\end{remark}

If $P$ is a finitely generated projective module and $M$ is a maximal submodule of $P$, we can use results from \cite[\S32]{levy-robson11} to express the genus of $M$ in terms of the genus of $P$ and the simple module $P/M$ as follows.
This will be particularly useful when $P=R$ and $M$ is a maximal right ideal.

\begin{lemma} \label{l-rank}
  Let $V$ be a simple module, let $P$ be a nonzero finitely generated projective module, and let $M \subsetneq P$ be a maximal submodule such that $P/M \cong V$.
  Let $X$ be an unfaithful simple module.
  \begin{enumerate}
    \item If $V$ is contained in a trivial tower, then $\rho(M,X) = \rho(P,X)$.
    \item If $V$ is contained in a non-trivial tower, then
      \[
      \rho(M,X) = \begin{cases}
        \rho(P,X) - 1 &\text{ if $X\cong V$ \textup{(}then $V$ is unfaithful\textup{)},} \\
        \rho(P,X) + 1 &\text{ if $X$ is an unfaithful successor of $V$,} \\
        \rho(P,X)     &\text{ otherwise}.
      \end{cases}
      \]
  \end{enumerate}
\end{lemma}

\begin{proof}
  We note that the proof of \cite[Lemma 32.15(iii)]{levy-robson11} only requires that $\Ext^1_R(V,X) = \mathbf 0$ and that $\ann_R(X) \ne \ann_R(V)$.
  This implies $\rho(M,X)=\rho(P,X)$ if $X$ is neither the unfaithful successor of $V$, nor $X \cong V$.
  Thus, we still have to show the claim if $X$ is an unfaithful successor of $V$ or $X \cong V$.

  Suppose first that $V$ is contained in a trivial tower.
  If $V$ is faithful, then it has no unfaithful successor (by triviality of the tower) and $X \ncong V$ due to $X$ being unfaithful.
  Hence in this case there is nothing left to show, and we may suppose that $V$ is unfaithful.
  Because the tower is trivial, $X\cong V$ is the only remaining possibility.
  Then $\rho(M,V) = \rho(P,V)$ due to cycle standard rank and $\udim M = \udim P$.

  Suppose now that $V$ is contained in a non-trivial tower.
  If $V \cong X$, then the proof of \cite[Lemma 32.15(i)]{levy-robson11} goes through and implies $\rho(M,V) = \rho(P,V)-1$.
  If $X$ is an unfaithful successor of $V$, then, as in \cite[Lemma 32.14(i)]{levy-robson11}, there exists $N \subset M$ such that $M/N\cong X$ and $P/N$ is uniserial.
  Now \cite[Lemma 32.15(ii)]{levy-robson11} implies $\rho(M,X) = \rho(P,X) + 1$.
\end{proof}

If $M \subset P$ are finitely generated projective modules, then $P/M$ has finite length if and only if $\udim P = \udim M$ by \cite[Corollary 12.17]{levy-robson11}.
As an immediate consequence of this and the previous lemma, we have the following.

\begin{cor} \label{cor-genus-of-simples}
  Let $V$ be a simple module, let $P$ be a nonzero finitely generated projective module, and let $M \subsetneq P$ be a maximal submodule such that $P/M \cong V$.
  \begin{enumerate}
    \item
      If $V$ is contained in a trivial tower, then $\Psi(M) = \Psi(P)$.
    \item
      If $V$ is contained in a non-trivial tower, then
      \[
      \Psi(M) = \begin{cases} \Psi(P) + e_{(V)^+} - e_{(V)} &\text{if $V$ is contained in a cycle tower,} \\
        &\text{or $V$ is any but the top or bottom}\\
        &\text{of a faithful tower,}\\
        \Psi(P) + e_{(V)^+} & \text{if $V$ is the top of a faithful tower,} \\
        \Psi(P) - e_{(V)}  & \text{if $V$ is the bottom of a faithful tower.}
      \end{cases}
      \]
  \end{enumerate}
  With the convention $e_{(V)}=\mathbf 0$ if $V$ is faithful and $e_{(V)^+}=\mathbf 0$ if $V$ does not have an unfaithful successor,
  \[
  \Psi(M) = \Psi(P) + e_{(V)^+} - e_{(V)}
  \]
  in any case.
\end{cor}

Occasionally, we will also need the following.
Recall that for a cycle tower the choice of top is arbitrary, but for a faithful tower the faithful simple module is the top.
\begin{lemma} \label{l-ex-uniserial}
  Let $P$ be a finitely generated projective module.
  Let $T$ be a tower with top $V$.
  If $\rho(P,V) \ne 0$, then there exists a submodule $P' \subset P$ such that $P/P'$ is a uniserial module with composition series, from top to bottom, precisely the modules of $T$.
\end{lemma}

\begin{proof}
  By \cite[Lemma 32.18]{levy-robson11}.
\end{proof}

A module $M$ is stably free if $M$ is stably isomorphic to $R^m$ for some $m \in \bN_0$.
That is, there exists $n \in \bN_0$ such that $M \oplus R^n \cong R^m \oplus R^n$.
By \cite[Corollary 35.6]{levy-robson11}, this is equivalent to $M \oplus R \cong R^m \oplus R$.
Since $m \udim(R) =\udim(M)$, any nonzero stably free right ideal $I$ of $R$ satisfies $\udim I=\udim R$, and is therefore a right $R$-ideal.
Note that a finitely generated stably free module $M$ is necessarily finitely generated projective and $\udim(M)$ is a multiple of $\udim(R)$.
If $\udim(M) \ge 2$ and $M$ is stably free, then $M$ is free by \cite[Corollary 35.6]{levy-robson11}.
Thus $M$ can only be non-free but stably free if $\udim(M)=\udim(R)=1$, that is, $R$ is a domain and $M$ is isomorphic to a right $R$-ideal.

In \cref{sec-transfer-hom}, we will need to impose the condition that every stably free right $R$-ideal is free.
By the previous paragraph, this is equivalent to every finitely generated stably free $R$-module being free.
The notion is left/right symmetric by dualization of finitely generated projective modules.
A ring having this property is sometimes called a \emph{Hermite ring} (see \cite[Chapter I.4]{lam06}).

Each of the following conditions is sufficient for every stably free right $R$-ideal to be free:
\begin{propenumerate}
  \item $R$ is commutative,
  \item $\udim R \ge 2$,
  \item $\cO$ is a Dedekind domain with quotient field a global field $K$, and $R$ is a classical hereditary $\cO$-order in a central simple algebra $A$ over $K$ such that $A$ satisfies the Eichler condition relative to $\cO$ (see \cite[Theorem 38.2]{reiner75}).
\end{propenumerate}
In the last case, if $K$ is a number field and $\cO$ is its ring of algebraic integers, then $A$ satisfies the Eichler condition with respect to $\cO$ unless $A$ is a totally definite quaternion algebra.
If $\cO$ is a ring of algebraic integers and $R$ is a classical hereditary $\cO$ order in a totally definite quaternion algebra, there exists a full classification of when every stably free right $R$-ideal is free (see \cite{vigneras76,hallouin-maire06,smertnig15}).

\subsection{Factorizations and transfer homomorphisms}

Let $(H,\cdot)$ be a monoid.
By $H^\times$ we denote the group of units of $H$.
The monoid $H$ is \emph{reduced} if $H^\times=\{1\}$.
An element $u \in H \setminus H^\times$ is an \emph{atom} if $u=ab$ with $a$,~$b \in H$ implies $a \in H^\times$ or $b \in H^\times$.
The set of atoms of $H$ is denoted by $\cA(H)$.
If $u \in \cA(H)$ and $\varepsilon$,~$\eta \in H^\times$ then also $\varepsilon u \eta \in \cA(H)$.
If every non-unit of $H$ can be written as a product of atoms, then $H$ is \emph{atomic}.

To be able to define arithmetical invariants, intended to measure the extent of non-uniqueness of factorizations in an atomic monoid, it is first necessary to give precise definitions of factorizations of an element and of distances between factorizations.
The following definitions were introduced in \cite{smertnig13,baeth-smertnig15} in the setting of cancellative small categories.
We recall them for monoids, as this will be sufficiently general for the present paper.

A first attempt may be to call an element of the free monoid on atoms of $H$, denoted by $\fmon{\cA(H)}$, a factorization.
Then the factorizations of an element $a \in H$ are all those formal products in $\fmon{\cA(H)}$ which, when multiplied out in $H$, give $a$ as a product.
This works well if $H$ is reduced.
However, in the presence of non-trivial units, this approach has two drawbacks.
First, if $u$,~$v \in \cA(H)$ and $\varepsilon \in H^\times$, then trivially $uv=(u\varepsilon)(\varepsilon^{-1}v)$.
It is more natural to consider these to be the same factorization (e.g., for the number of factorizations of a given element to be a more meaningful measure).
Second, to avoid having to treat units as a special case, it is preferable for units to also have (trivial) factorizations.
If $H$ is reduced, the empty product is the unique factorization of $1_H$.
In the presence of other units, we need one such trivial factorization for each unit.
The following definition takes care of both of these issues by `tagging' an empty factorization with a unit, and by factoring out a suitable congruence relation to deal with the trivial insertion of units.

We endow the cartesian product $H^\times \times \fmon{\cA(H)}$ with the following operation:
If $(\varepsilon,y)$,~$(\varepsilon',y') \in H^\times \times \fmon{\cA(H)}$ where $y=u_1\cdots u_k$ and $y'=v_1\cdots v_l$ with $u_1$, $\ldots\,$,~$u_k$, $v_1$, $\ldots\,$,~$v_l \in \cA(H)$, then
\[
(\varepsilon,y)(\varepsilon',y') =
\begin{cases} (\varepsilon,u_1\cdots (u_k\varepsilon')v_1\cdots v_l) &\text{if $k >0$,} \\
              (\varepsilon\varepsilon', v_1\cdots v_l)               &\text{if $k=0$}.
\end{cases}
\]
With this product $H^\times \times \fmon{\cA(H)}$ is a monoid.
On $H^\times \times \fmon{\cA(H)}$, we define a congruence $\sim$ by $(\varepsilon,y)\sim(\varepsilon',y')$ if all of the following hold:
\begin{propenumerate}
  \item $k=l$,
  \item $\varepsilon u_1\cdots u_k=\varepsilon' v_1\cdots v_k$ as product in $H$,
  \item either $k=0$, or there exist $\delta_2$, $\ldots\,$,~$\delta_k \in H^\times$ and $\delta_{k+1}=1$ such that
    \[
    \varepsilon' v_1 = \varepsilon u_1 \delta_2^{-1} \quad\text{and}\quad v_i = \delta_i u_i \delta_{i+1}^{-1} \quad\text{for all $i \in [2,k]$.}
    \]
\end{propenumerate}
\begin{defi}
  The quotient $\sZ^*(H) = H^\times \times \fmon{\cA(H)}/\sim$ is the \emph{monoid of \textup{(}rigid\textup{)} factorizations of $H$}.
  The class of $(\varepsilon,u_1\cdots u_k)$ in $\sZ^*(H)$ is denoted by $\rf[\varepsilon]{u_1,\cdots,u_k}$.
  The symbol $\rfop$ also denotes the operation on $\sZ^*(H)$.
  There is a natural homomorphism
  \[
  \pi=\pi_H \colon \sZ^*(H) \to H, \quad \rf[\varepsilon]{u_1,\cdots,u_k} \mapsto \varepsilon u_1\cdots u_k.
  \]
  For $a \in H$, the set $\sZ^*(a)=\sZ_H^*(a) = \pi^{-1}(a)$ is the set of \emph{\textup{(}rigid\textup{)} factorizations of $a$}.
  If $z=\rf[\varepsilon]{u_1,\cdots,u_k}$, then $\length{z}=k$ is the \emph{length} of $z$.
\end{defi}

By construction, $\rf[\varepsilon]{u_1,\cdots,u_{i},u_{i+1},\cdots,u_k}=\rf[\varepsilon]{u_1,\cdots,u_i\delta^{-1},\delta u_{i+1},\cdots,u_k}$ for all $\delta \in H^\times$ and $i \in [1,k-1]$.
Similarly, $\rf[\varepsilon]{u_1,\cdots,u_k} = \rf{(\varepsilon u_1),\cdots,u_k}$ if $k \ge 1$.
In particular, we may represent $z$ as $z=\rf{u_1',\cdots,u_k'}$ with atoms $u_1'$, $\ldots\,$,~$u_k'$, omitting the unit at the beginning, as long as $k \ge 1$ (equivalently, $\pi(z) \not\in H^\times$).

If $H$ is reduced, then $\sZ^*(H) \cong \fmon{\cA(H)}$.

\begin{remark}
For $a \in H$, denote by $[aH,H]$ the set of all principal right ideals $bH$ of $H$ with $aH \subset bH \subset H$.
The set $[aH,H]$ is partially ordered by set inclusion.
There is a natural bijection between $\sZ^*(a)$ and maximal chains of finite length in $[aH,H]$, given by
\[
\rf[\varepsilon]{u_1,\cdots,u_k} \quad\leftrightarrow\quad aH=\varepsilon u_1\cdots u_kH \subsetneq \varepsilon u_1\cdots u_{k-1}H \subsetneq \cdots \subsetneq \varepsilon u_1H \subsetneq H.
\]
The restriction to chains of finite lengths corresponds to the fact that we consider representations of $a$ as finite products.
The restriction to maximal such chains corresponds to the factors being atoms.
\end{remark}

We now recall the concept of a \emph{\textup{(}weak\textup{)} transfer homomorphism} in a setting sufficiently general for the present paper.
See \cite{baeth-smertnig15} for a more general definition.

\begin{defi}\label{th}
  Let $H$ be a monoid and let $T$ be a reduced commutative monoid.
  \begin{enumerate}
    \item\label{th:th}
    A homomorphism $\theta\colon H \to T$ is called a \emph{transfer homomorphism} if it has the following properties:
    \begin{enumerate}[label=\textup{(\textbf{T\arabic*})},ref=\textup{(T\arabic*)},leftmargin=*]
      \item\label{th:units} $\theta$ is surjective and $\theta^{-1}(\{1\})=H^{\times}$.
      \item\label{th:lift} If $a \in H$, $s$, $t \in T$ and $\theta(a)=st$, then there exist $b$,~$c \in H$ such that $a = bc$, that $\theta(b) = s$, and that $\theta(c) = t$.
    \end{enumerate}

  \item\label{th:wth}
    Suppose $T$ is atomic.
    A homomorphism $\theta\colon H \rightarrow T$ is called a \emph{weak transfer homomorphism} if it has the following properties:
    \begin{enumerate}[label=\textup{(\textbf{WT\arabic*})},ref=\textup{(WT\arabic*)},leftmargin=*]
      \item[\textbf{(T1)}] $\theta$ is surjective and $\theta^{-1}(\{1\})=H^{\times}$.
        \setcounter{enumii}{1}
      \item\label{wth:lift} If $a \in H$, $k \in \bN$, $v_1$, $\ldots\,$,~$v_k \in \cA(T)$, and $\theta(a)=v_1\cdots v_k$, then there exist $u_1$, $\ldots\,$,~$u_k \in \cA(H)$ and a permutation $\sigma \in \fS_k$ such that $a=u_1\cdots u_k$ and $\theta(u_i) = v_{\sigma(i)}$ for all $i \in [1,k]$.
    \end{enumerate}
  \end{enumerate}
\end{defi}

If $T$ is atomic, every transfer homomorphism $\theta\colon H \to T$ is also a weak transfer homomorphism.
If $\theta$ is a (weak) transfer homomorphism and $u \in H$, then $u \in \cA(H)$ if and only if $\theta(u) \in \cA(T)$.
If $\theta\colon H \to T$ is a transfer homomorphism and $\theta(a) = v_1\cdots v_k$ with $v_1$, $\ldots\,$,~$v_k \in \cA(T)$, then there exists a rigid factorization $z=\rf[\varepsilon]{u_1,\cdots,u_k}$ of $a$ with $\theta(u_i)=v_i$ for $i \in [1,k]$.
Thus, a transfer homomorphism $\theta$ induces a surjective homomorphism $\sZ^*(H) \to \sZ^*(T)$.

To define more fine grained arithmetical invariants than those based on sets of lengths, it is necessary to be able to compare two factorizations of an element.
For this, we use \emph{distances} between factorizations.

\begin{defi} \label{d-distance}
  Let $D = \{\, (z,z') \in \sZ^*(H) \times \sZ^*(H) \mid \pi(z)=\pi(z') \,\}$.
  A \emph{distance on $H$} is a map $\sd\colon D \to \bN_0$ having the following properties for all $z$,~$z'$,~$z'' \in \sZ^*(H)$ with $\pi(z)=\pi(z')=\pi(z'')$ and all $x \in \sZ^*(H)$:
  \begin{enumerate}[label=\textup{(\textbf{D\arabic*})},ref=\textup{(D\arabic*)},leftmargin=*]
    \item\label{d:ref} $\sd(z,z) = 0$.
    \item\label{d:sym} $\sd(z,z') = \sd(z',z)$.
    \item\label{d:tri} $\sd(z,z') \le \sd(z,z'') + \sd(z'',z')$.
    \item\label{d:mul} $\sd(x\rfop z, x \rfop z') = \sd(z,z') = \sd(z \rfop x, z' \rfop x)$.
    \item\label{d:len} $\abs[\big]{\length{z} - \length{z'}} \le \sd(z,z') \le \max\left\{ \length{z}, \length{z'}, 1 \right\}$.
  \end{enumerate}
\end{defi}

If $F=\famon{X}$ is a free abelian monoid with basis $X$ and $x$,~$y \in F$, then $x=x_0z$ and $y=y_0z$ with $z=\gcd(x,y)$ and suitable cofactors $x_0$,~$y_0 \in F$.
We set $\mathsf d_F(x,y) = \max\{\length{x_0},\length{y_0}\}$.
Suppose $\sim$ is an equivalence relation on $\cA(H)$ such that $u = \varepsilon v \eta$ with $\varepsilon$,~$\eta \in H^\times$ implies $u \sim v$.
Then there is a natural homomorphism $\varphi\colon \sZ^*(H) \to F=\famon{\cA(H)/\sim}$, and $\sd_\sim = \sd_F \circ (\varphi \times \varphi)$, restricted to $D$, is a distance on $H$ (see \cite[Construction 3.3(2)]{baeth-smertnig15}).
In other words, to evaluate $\sd_\sim(z,z')$, we first remove as many pairs of atoms from $z$ and $z'$ that are equivalent under $\sim$ as possible, and then take the maximum length of the remaining factorizations as the distance.
For the \emph{permutable distance}, denoted by $\sd_p$, we use $u \sim v$ if and only if $u = \varepsilon v \eta$ for some $\varepsilon$,~$\eta \in H^\times$.

The permutable distance is the distance typically used in the commutative setting.
If $H=R^\bullet$ is the monoid of non-zero-divisors of the HNP ring $R$, there are two other important distances constructed in this fashion.
With the choice of $u \sim v$ if $R/uR \cong R/vR$, we obtain the \emph{similarity distance}, denoted by $\sd_{\textup{sim}}$.

If $a \in R^\bullet$, then $aR$ is an essential right ideal of $R$ and $R/aR$ has finite length.
Thus, we can define another distance by defining $u\sim v$ if $R/uR$ and $R/vR$ have equivalent composition series.
We call this the \emph{composition distance} and denote it by $\sd_{\textup{cs}}$.

\begin{remark}
  \begin{enumerate}
    \item
      Utilizing the transpose (see \cite[\S17]{levy-robson11}), we see that the definitions of similarity and composition distance are in fact symmetric with respect to taking left or right modules: $R/uR\cong R/vR$ if and only if $R/Ru \cong R/Rv$ and moreover $R/uR$ and $R/vR$ have equivalent composition series if and only if the same is true for $R/Ru$ and $R/Rv$.

    \item If $R$ is commutative, then $\sd_{\textup{cs}}=\sd_{\textup{sim}}=\sd_{p}$.
  \end{enumerate}
\end{remark}

Finally, for any monoid $H$, the \emph{rigid distance}, denoted by $\sd^*$, is defined as in \cite[Construction 3.3(1)]{baeth-smertnig15}.
Essentially, one counts how many insertions, deletions, or replacements of atoms are necessary to pass from $z$ to $z'$.
(With some added technicalities to account for the possible presence of units.)

If $\sd$ is a distance on $H$, we may define a congruence relation $\sim_\sd$ on $\sZ^*(H)$ where $z \sim_\sd z'$ if and only if $\pi(z)=\pi(z')$ and $\sd(z,z')=0$.
Then $\sZ_\sd(H)=\sZ^*(H)/\sim_\sd$ is called the \emph{monoid of $\sd$-factorizations}.
We say that $H$ is \emph{$\sd$-factorial} if the induced homomorphism $\sZ_\sd(H) \to H$ is bijective, that is, every element of $H$ has precisely one $\sd$-factorization.
For $\sd=\sd^*$ [$\sd_{\text{sim}}$, $\sd_{\text{cs}}$, $\sd_p$], we say that $H$ is \emph{rigidly \textup{[}similarity, composition series, permutably\textup{]} factorial}.

For any distance we can now define a catenary degree.
\begin{defi}[Catenary degree]
  Let $H$ be atomic and let $\sd$ be a distance on $H$.
  \begin{enumerate}
    \item
      Let $a \in H$ and $z$,~$z' \in \sZ^*(H)$.
      A finite sequence of rigid factorizations $z_0$, $\ldots\,$,~$z_n$ of $a$ is called an $N$-chain (in distance $\sd$) between $z$ and $z'$ if
      \[
      z=z_0, \quad \sd(z_{i-1},z_i) \le N\text{ for all }i\in [1,n], \quad \text{and }z_n=z'.
      \]

    \item The \emph{catenary degree \textup{(}in distance $\sd$\textup{)} of $a$}, denoted by $\sc_\sd(a)$, is the minimal $N \in \bN_0 \cup \{\infty\}$ such that for any two factorizations $z$,~$z' \in \sZ^*(a)$ there exists an $N$-chain between $z$ and $z'$.

    \item The \emph{catenary degree \textup{(}in distance $\sd$\textup{)} of $H$} is
      \[
      \sc_\sd(H) = \sup\{\, \sc_\sd(a) \mid a \in H \,\} \in \bN_0 \cup \{\infty\}.
      \]
  \end{enumerate}
\end{defi}
For specific distances, we write $\sc_{\textup{sim}}(H)$ instead of $\sc_{\sd_{\textup{sim}}}(H)$, etc.
An atomic monoid $H$ is $\sd$-factorial if and only if $\sc_\sd(H)=0$.

Let $\theta\colon H \to T$ be a transfer homomorphism.
We say that two factorization $z=\rf[\varepsilon]{u_1,\cdots,u_k}$ and $z'=\rf[\varepsilon']{u_1',\cdots,u_k'}$ in $\sZ^*(H)$ with $\pi(z)=\pi(z')$ \emph{lie in the same \textup{(}permutable\textup{)} fiber} (of $\theta$), if there exists a permutation $\sigma \in \fS_k$ such that $\theta(u_i)=\theta(u_{\sigma(i)}')$ for all $i \in[1,k]$.
This is equivalent to $z$ and $z'$ lying in the same fiber of the natural extension of $\theta$ to $\sZ^*(H) \to \sZ^*(T) \to \sZ_p(T)=\famon{\cA(T)}$.

\begin{defi}[Catenary degree in the fibers]
  Let $\sd$ be a distance on $H$ and suppose that $H$ is atomic.
  Let $T$ be a reduced commutative monoid, and let $\theta\colon H \to T$ be a transfer homomorphism.
  \begin{enumerate}
    \item
      An \emph{$N$-chain in the \textup{(}permutable\textup{)} fiber} of $z$ is an $N$-chain all of whose factorizations lie in the fiber of $z$.

    \item
      We define $\sc_\sd(a,\theta)$ to be the smallest $N \in \bN_0 \cup \{\infty\}$ such that, for any two rigid factorizations $z$,~$z'$ of $a$ lying in the same fiber, there exists an $N$-chain in the fiber between $z$ and $z'$.
      The \emph{catenary degree in the \textup{(}permutable\textup{)} fibers} is
      \[
      \sc_\sd(H,\theta) = \sup\{\, \sc_\sd(a,\theta) \mid a \in H\,\} \in \bN_0 \cup \{\infty\}.
      \]
  \end{enumerate}
\end{defi}

For $a \in H$, we call $\sL(a)=\sL_H(a)=\{\, \length{z} \mid z \in \sZ_H^*(a) \,\} \subset \bN_0$ the \emph{set of lengths} of $a$.
The monoid $H$ is \emph{half-factorial} if $\card{\sL(a)}=1$ for all $a \in H$.
By $\cL(H) = \{\, \sL(a) \mid a \in H \,\} \subset \bP(\bN_0)$ we denote the \emph{system of sets of lengths}.
Several arithmetical invariants are defined in terms of the system of sets of lengths.

For instance, if $L \subset \bN_0$ and $d \in \bN$, then $d$ is a \emph{distance} of $L$ if there exist $k$,~$l \in L$ with $l-k=d$ and the interval $[k,l]$ contains no further elements of $L$.
Then $\Delta(L)$ denotes the set of distances of $L$, and $\Delta(H) = \smash\bigcup_{L \in \cL(H)} \Delta(L)$ is the \emph{set of distances of $H$}.
(The terminology `distance' here is unrelated to the one in \cref{d-distance}.)
See the survey \cite{geroldinger16} for further arithmetical invariants.

The following theorem shows that transfer homomorphisms preserve many arithmetical invariants.

\begin{thm} \label{t-transfer}
  Let $H$ be a monoid, $T$ an atomic reduced commutative monoid.
  \begin{enumerate}
  \item\label{t-transfer:len}
    Let $\theta\colon H \to T$ be a weak transfer homomorphism.
    For all $a \in H$, we have $\sL_H(a) = \sL_T(\theta(a))$.
    In particular, $\cL(H) = \cL(T)$ and all arithmetical invariants defined in terms of lengths coincide for $H$ and $T$.
  \item\label{t-transfer:cat} If $\theta\colon H \to T$ is a transfer homomorphism, then
      \begin{align*}
        \sc_\sd(H) & \le \max\big\{ \sc_p(T), \sc_\sd(H, \theta) \big\}.
      \end{align*}
  \end{enumerate}
\end{thm}

\begin{proof}
  \ref*{t-transfer:len} is straightforward.
  For \ref*{t-transfer:cat}, see \cite[Proposition 4.6(2)]{baeth-smertnig15}.
\end{proof}

For not too coarse distances $\sd$, it is usually easy to obtain $\sc_\sd(H) \ge \sc_p(T)$.
The strength of \ref*{t-transfer:cat} comes from the fact that, for interesting classes of monoids and rings, one is able to prove $\sc_\sd(H,\theta) \le N$ for some small constant $N$ (say, $N=2$).
Then the catenary degree of $H$ is equal to that of $T$, unless $\sc_\sd(H)$ is very small.

We briefly recall monoids of zero-sum sequences.
As codomains of transfer homomorphisms, these monoids play a pivotal role in the study of non-unique factorizations.
Let $(G,+)$ be an additive abelian group and let $G_0 \subset G$ be a subset.
In the tradition of combinatorial number theory, an element of the, multiplicatively written, free abelian monoid $(\famon{G_0},\cdot)$ is called a \emph{sequence}.
If $S=g_1\cdots g_k \in \famon{G_0}$ is a sequence, then $\sigma(S)= g_1 + \cdots + g_k \in G$ is its sum.
The map $\sigma\colon (\famon{G_0},\cdot) \to (G,+)$ is a monoid homomorphism.
The \emph{monoid of zero-sum sequences} over $G_0$,
\[
\cB(G_0) = \{\, S \in \famon{G_0} \mid \sigma(S) = 0 \,\},
\]
is a commutative Krull monoid.

\begin{remark}
  If $H$ is a commutative Krull monoid, $G$ is its divisor class group, and $G_0\subset G$ is the set of classes containing prime divisors, then there exists a transfer homomorphism $\theta\colon H \to \cB(G_0)$ with $\sc_\sd(H,\theta) \le 2$ for any distance $\sd$ on $H$.
  If $a \in H$ and $aH = \fp_1 \cdot_v \cdots \cdot_v \fp_k$ with divisorial prime ideals $\fp_1$, $\ldots\,$,~$\fp_k$, then $\theta(a)=[\fp_1]\cdots [\fp_k] \in \cB(G_0)$, where $[\fp_i] \in G_0$ is the class of $\fp_i$ and the product is the formal product in $\famon{G_0}$.
  See \cite[Theorem 3.2.8]{ghk06} or \cite[Theorem 1.3.4]{geroldinger09} for this; a generalization to normalizing Krull monoids appears in \cite[Theorem 6.5]{geroldinger13} and one to arithmetical maximal orders in \cite[Theorem 5.23]{smertnig13} and \cite[Section 7]{baeth-smertnig15}.
\end{remark}

Monoids of zero-sum sequences are studied using methods from combinatorial and additive number theory.
We refer to \cite{ghk06,geroldinger09,grynkiewicz13} as starting points in this direction.
By \cref{t-transfer}, results about their factorization theory carry over to monoids which possess a transfer homomorphism to monoids of zero-sum sequences.
Such monoids are called \emph{transfer Krull monoids} in \cite{geroldinger16}.

If $G_0$ is finite, the main results about $\cB(G_0)$ are the finiteness of several arithmetical invariants \cite[Theorem 4.6]{geroldinger16} and the \emph{structure theorem for sets of lengths} \cite[Theorem 5.3]{geroldinger16}, which implies that sets of lengths are almost arithmetical multiprogressions with differences $d \in \Delta(\cB(G_0))$.
Motivated by rings of algebraic integers, the main attention however has been on the case where $G=G_0$ is a finite abelian group.
In this case, $\Delta(\cB(G))$ is a finite interval starting at $1$ (if it is non-empty) and, under weak assumptions on $G$, it has been shown that $\Delta(\cB(G)) = [1,\sc_p(\cB(G))-2]$. Many arithmetical invariants of $\cB(G)$ can be expressed in terms of the Davenport constant of $G$.

\section{Class groups and modules of finite length}
\label{sec-class-groups}

In this section we introduce a class group $\cgrp{R}$, together with a distinguished subset $\cgrpm{R}$, in terms of modules of finite length.
In the next section, $\cgrp{R}$ will be used to construct a transfer homomorphism.
Here, we show that $\cgrp{R}$ is isomorphic to the ideal class group $\pcg{R}$ as defined in \cite{levy-robson11}.
The subset $\cgrpm{R}$ corresponds to a subset $\pcgm{R}$, and we show that $\pcgm{R}$ and $\pcg{R}$ are preserved under Morita equivalence and passage to a Dedekind right closure.

The modules of finite length, as a full subcategory of the category of $R$-modules, form an abelian category, denoted by $\modfl(R)$.
This category has an associated Grothendieck group $K_0\modfl(R)$.
The Jordan-Hölder Theorem implies that $K_0\modfl(R)$ is a free abelian group with basis the isomorphism classes of simple $R$-modules.
We use additive notation for $K_0\modfl(R)$ and write classes in $K_0\modfl(R)$ using round parentheses; thus, if $M$ is a module of finite length, we write
\[
(M) \in K_0\modfl(R).
\]
If $M$ has a composition series with composition factors $V_1$, $\ldots\,$,~$V_n$, then $(M) = (V_1) + \cdots + (V_n)$ in $K_0\modfl(R)$.

We write $\simple{R}$ for the set of isomorphism classes of simple $R$-modules.
For a set $\Omega$, we write $\famon{\Omega}$ for the free abelian monoid with basis $\Omega$, and $\fagrp{\Omega}$ for the free abelian group with basis $\Omega$.
With this notation $K_0\modfl(R) = \fagrp{\simple{R}}$.
For a tower $T$, we set $(T) = \smash\sum_{(V) \in T} (V) \in K_0\modfl(R)$.
We write $\towers{R}$ for the set of all towers (as usual, identifying cyclic permutations of a cycle tower).
Note that $\{\, (T) \mid T \in \towers{R} \,\}$ forms a $\bZ$-linearly independent subset of $\famon{\simple{R}}$.
Thus, the submonoid of $\famon{\simple{R}}$ generated by the towers is isomorphic to the free abelian monoid on towers.
In this way we can embed $\famon{\towers{R}}$ into $\famon{\simple{R}}$, and $\fagrp{\towers{R}}$ into $\fagrp{\simple{R}}=K_0\modfl(R)$.
We shall identify them by means of these embeddings.

\begin{lemma} \label{l-ex-phi0}
  There exists a homomorphism of abelian groups
  \[
  \Phi_0\colon K_0\modfl(R) \to K_0(R)
  \]
  such that whenever $M$ is a module of finite length and $M \cong P/Q$ with finitely generated projective modules $P$ and $Q$, then $\Phi_0((M)) = [Q] - [P]$.
\end{lemma}

\begin{proof}
  Let $M$ be a module of finite length, and suppose that $M \cong P/Q \cong P'/Q'$ with $P$, $Q$, $P'$, and $Q'$ finitely generated projective modules.
  Then Schanuel's Lemma implies $P \oplus Q' \cong P' \oplus Q$, and hence $[Q]-[P]=[Q']-[P']$.
  Thus we can define a map $\widetilde\Phi_0\colon \modfl(R) \to K_0(R)$ such that $M \mapsto [Q] - [P]$.

  If $\mathbf 0 \to L \to M \to N \to \mathbf 0$ is a short exact sequence in $\modfl(R)$ and $M \cong P/Q$, then there exists $Q \subset P' \subset P$ such that $N \cong P/P'$ and $L \cong P' / Q$.
  Since $[Q]-[P] = ([Q]-[P']) + ([P'] - [P])$, the map $\widetilde\Phi_0$ is additive on short exact sequences.
  By the universal property of the Grothendieck group, there exists $\Phi_0\colon K_0\modfl(R) \to K_0(R)$ as claimed.
\end{proof}

Recall that $\Psi^+\colon K_0(R) \to \bZ^{\modspec{R}}$ is the extension of the genus $\Psi$ to $K_0(R)$.
\begin{prop} \label{p-f}
  \begin{enumerate}
    \item \label{p-f:ker} We have $\ker(\Psi^+ \circ \Phi_0) = \fagrp{\towers{R}}$.
    \item\label{p-f:gen} Let $Q \subset P$ be finitely generated projective modules.
      Then the following statements are equivalent:
      \begin{equivenumerate}
        \item\label{p-f:gen:rank} $\Psi(Q)=\Psi(P)$.
        \item\label{p-f:gen:tower} $P/Q \in \modfl(R)$ and $(P/Q) \in \famon{\towers{R}}$.
      \end{equivenumerate}
  \end{enumerate}
\end{prop}

\begin{proof}
  \ref*{p-f:ker}
  Let $(M) - (N) \in K_0\modfl(R)$ with $M$, $N$ modules of finite length.
  Since $K_0\modfl(R)$ is the free abelian group on isomorphism classes of simple modules, we have
  \[
  (M) - (N) = \sum_{T \in \towers{R}} \sum_{(V) \in T} n_{(V)} (V) \in K_0\modfl(R)
  \]
  with $n_{(V)} \in \bZ$ almost all zero.
  Thus, by Corollary~\ref{cor-genus-of-simples},
  \[
    \Psi^+ \circ \Phi_0((M)-(N)) = \sum_{T \in \towers{R}} \sum_{(V) \in T} n_{(V)} e_{(V)^+} - n_{(V)} e_{(V)},
  \]
  where we use the notational convention $e_{(V)}=\mathbf 0$ if $(V)$ is the top of a faithful tower, and $e_{(V)^+}=\mathbf 0$ if $(V)$ is the bottom of a faithful tower.
  Hence, $\Psi^+\circ \Phi_0((M)-(N)) = \mathbf 0$ if and only if $n_{(V)} = n_{(W)}$ whenever $(V)$ and $(W)$ are contained in the same tower.
  This is the case if and only if $(M) - (N) \in \fagrp{\towers{R}}$.

  \ref*{p-f:gen} Recall that $P/Q$ has finite length if and only if $\udim P = \udim Q$ by \cite[Corollary 12.17]{levy-robson11}.
  Then the claim follows from \ref*{p-f:ker} and $\fagrp{\towers{R}} \cap \famon{\simple{R}} = \famon{\towers{R}}$.
\end{proof}

\begin{defi}
A right $R$-ideal $I$ is in the \emph{principal genus} if $\Psi(I)=\Psi(R)$.
\end{defi}
By \cref{p-f}, a right $R$-ideal $I$ is in the principal genus if and only if $(R/I) \in \famon{\towers{R}}$.
If $I$ is in the principal genus, then $I$ is a progenerator.
The converse is not true.

A cyclic module $R/I$ with $I$ a right ideal of $R$ has finite length if and only if $I$ is a right $R$-ideal.
In particular, a cyclically presented module, that is, a module of the form $R/aR$ with $a\in R$, has finite length if and only if $a \in R^\bullet$.
We write $\cp{R}$ for the submonoid of $K_0\modfl(R)$ consisting of all $(R/aR)$ with $a \in R^\bullet$.
(This is indeed a submonoid, because $(R/aR) + (R/bR) = (R/aR) + (aR/abR) = (R/abR)$ for $a$,~$b \in R^\bullet$.)
We write $\qcp{R}$ for the quotient group of $\cp{R}$, which is the subgroup of $K_0\modfl(R)$ generated by $\cp{R}$.
Classes in the factor group $K_0\modfl(R)/\qcp{R}$ will be denoted using angle brackets, that is, for a module $M$ of finite length, we write
\[
\langle M \rangle \in K_0\modfl(R)/\qcp{R}.
\]
(We also use angle brackets to denote a subgroup generated by a set, as on the right hand side of the previous equation, but there should be no confusion.)
For a tower $T$, we abbreviate $\langle (T) \rangle$ as $\langle T \rangle$.

We make an easy observation about the kernel of $\Phi_0$.
\begin{lemma} \label{l-prinker}
  We have $\qcp{R} \subset \ker(\Phi_0) \subset \fagrp{\towers{R}}$.
\end{lemma}

\begin{proof}
  If $a \in R^\bullet$, then $[aR]=[R]$ and hence $\Phi_0((R/aR))=\mathbf 0$.
  Thus $\qcp{R} \subset \ker(\Phi_0) \subset \ker(\Psi^+ \circ \Phi_0) = \fagrp{\towers{R}}$ by \subref{p-f:ker}.
\end{proof}
With this, we are able to define a class group in terms of modules of finite length.
\begin{defi} \label{d-cgrp}
  Let $\cgrp{R} = \fagrp{\towers{R}}/ \qcp{R}$.
\end{defi}
The group $K_0\modfl(R)/\qcp{R}$ may be viewed as a sort of class group in the sense `right $R$-ideals modulo principal right $R$-ideals', except that instead of a right $R$-ideal $I$ one considers (the composition series of) $R/I$.
The group $\cgrp{R}$ may be viewed as `right $R$-ideals in the principal genus modulo principal right $R$-ideals'.
For the construction of a transfer homomorphism in the setting of HNP rings, it turns out that $\cgrp{R}$ is the correct class group to use.

\begin{remark}
  If $R$ is a Dedekind prime ring, then every tower is trivial, every right $R$-ideal is in the principal genus, $K_0\modfl(R) = \fagrp{\simple{R}} = \fagrp{\towers{R}}$, and $K_0\modfl(R)/\qcp{R} = \cgrp{R}$.
  If $R$ is a commutative Dedekind domain, then $I=\ann_R(R/I)$, and it is easy to see that $K_0\modfl(R)/\qcp{R}=\cgrp{R}$ is isomorphic to the group of fractional $R$-ideals modulo the principal fractional $R$-ideals.
  (This is a special case of \cref{l-dedekind-classiso} below.)
\end{remark}

The main result in this section, \cref{t-k0-iso-pcg} below, implies $\cgrp{R} \cong \pcg{R}$.
For the proof of this, several intermediate steps are needed.
Let us first consider the image of $\Phi_0\colon K_0\modfl(R) \to K_0(R)$.
Observe that the uniform dimension gives a homomorphism of abelian groups $\udim\colon K_0(R) \to \bZ$.

\begin{lemma} \label{l-image}
  We have $\im \Phi_0 = \ker(\udim) \cong \pcg{R} \times \bigoplus_{T \in \towers{R}} \bZ^{\length{T}-1}$.
\end{lemma}

\begin{proof}
  The inclusion $\im \Phi_0 \subset \ker(\udim)$ is immediate from the definition of $\Phi_0$.
  Every element of $K_0(R)$ has the form $[Q]-[P]$ with $P$ and $Q$ finitely generated projective modules.
  Suppose now that $[Q] - [P] \in \ker(\udim)$.
  Then $\udim(P)=\udim(Q)$.
  Using \cite[Lemma 12.8]{levy-robson11}, we may without restriction assume $Q \subset P$.
  Then $P/Q$ has finite length by \cite[Corollary 12.17]{levy-robson11}, and $\Psi_0((P/Q)) = [Q]-[P]$.
  Thus we have shown $\im\Phi_0 = \ker(\udim)$.

  We now show the isomorphism on the right hand side.
  Let $\ftowers{R}$ denote the set of all faithful towers, and $\ctowers{R}$ the set of all cycle towers.
  By almost standard rank, and additivity of rank and Steinitz class, there exists a homomorphism
  \[
  K_0(R) \to \pcg{R} \times \bZ \times \bigoplus_{T \in \ctowers{R}} \tfrac{1}{\udim R} \bZ^{\length{T}} \times \bigoplus_{T \in \ftowers{R}} \tfrac{1}{\udim R} \bZ^{\length{T}-1}
  \]
  satisfying $[P] \mapsto \left(\stzcls{P}, \Psi(P) - \frac{\udim P}{\udim R} \Psi(R)\right)$.
  This homomorphism restricts to $\ker(\udim)$ to give a homomorphism
  \[
  \alpha\colon \ker(\udim) \to \pcg{R} \times \bZ \times \bigoplus_{T \in \ctowers{R}} \bZ^{\length{T}} \times \bigoplus_{T \in \ftowers{R}} \bZ^{\length{T}-1}
  \]
  with $\alpha([Q] - [P])=\big(\stzcls{Q} - \stzcls{P}, \Psi(Q) - \Psi(P) \big)$.
  Since Steinitz class and genus determine the stable isomorphism class of a finitely generated projective module, $\alpha$ is injective.
  For a cycle tower $T \in \ctowers{R}$, let $V_T \subset \bZ^{\length{T}}$ denote the subset of all $(v_1,\ldots,v_{\length{T}}) \in \bZ^{\length{T}}$ with $v_1 + \cdots + v_{\length{T}} = 0$ and note that $V_T \cong \bZ^{\length{T}-1}$.
  By cycle standard rank, the image of $\alpha$ is contained in
  \begin{equation} \label{eq-alpha-image}
  \pcg{R} \times \mathbf 0 \times \bigoplus_{T \in \ctowers{R}} V_T \times \bigoplus_{T \in \ftowers{R}} \bZ^{\length{T}-1}
  \end{equation}
  and we claim equality.

  It suffices to show that each of the summands in \cref{eq-alpha-image}, with the remaining components set to zero, is contained in the image of $\alpha$.
  For $g \in \pcg{R}$, note that since Steinitz class and rank are independent invariants of finitely generated projective modules, there exists a finitely generated projective module $P$ with $\Psi(P)=\Psi(R)$ and $\stzcls{P}=g + \stzcls{R}$.
  Hence $\alpha([P]-[R]) = (g,\mathbf 0)$.

  Let $T$ be a non-trivial cycle tower, represented by pairwise non-isomorphic simple modules $W_1$, $\ldots\,$,~$W_n$ with $W_i$ an unfaithful successor of $W_{i-1}$ for all $i \in [2,n]$.
  If $M_i$ is a maximal right $R$-ideal with $R/M_i \cong W_i$, then \cref{cor-genus-of-simples} implies $\alpha([M_i]-[R]) = (g, e_{(W_{i+1})} - e_{(W_i)})$ for some $g \in \pcg{R}$.
  Since $(g,\mathbf 0) \in \im\alpha$, also $(0,e_{(W_{i+1})} - e_{(W_i)})$ is in the image of $\alpha$.
  It is easy to see that the vectors $(0,-e_{(W_1)} + e_{(W_2)})$, $\ldots\,$, $(0,-e_{(W_{n-1})}+e_{(W_{n})})$ span the $V_T$-summand.

  If $T$ is a non-trivial faithful tower, represented by pairwise non-isomorphic simple modules $W_1$, $\ldots\,$,~$W_n$ with $W_i$ a successor of $W_{i-1}$ for all $i \in [2,n]$, we see, again using \cref{cor-genus-of-simples}, that $(0,e_{(W_2)})$, $(0, e_{(W_{i+1})}-e_{(W_i)})$ for $i \in [2,n-1]$, and $(0,-e_{(W_n)})$ are contained in the image of $\alpha$.
  These vectors span the $\bZ^{\card{T}-1}$-factor corresponding to $T$.
\end{proof}

\begin{remark}
  Let $\mathbf 0 \ne P$ be a finitely generated projective module and let $n=\udim(P)$.
  Let $K_0^{(n)}(R)$ denote the subgroup generated by all stable isomorphism classes of finitely generated projective modules of uniform dimension $n$.
  Then $\ker(\udim) \cong K_0^{(n)}(R) / \langle [P] \rangle$.
  In particular, if $R$ is a domain, then $\ker(\udim) \cong K_0(R)/\langle [R] \rangle$.
\end{remark}

We now study some properties of $\langle I /J \rangle$ when $J \subset I$ are fractional right $R$-ideals.
Several basic properties follow immediately by considering composition series.
For instance, if $K \subset J \subset I$ are fractional right $R$-ideals, then $\langle I / K \rangle = \langle I / J \rangle + \langle J / K \rangle$ because $(I/K) = (I/J) + (J/K)$.
Clearly, if $J \subset I$ are fractional right $R$-ideals, and $\alpha\colon I \to \alpha(I)$ is an isomorphism, then $\alpha(I)/\alpha(J) \cong I/J$.
In particular, $\langle R / I \rangle = \langle xR / xI \rangle$ for all $x \in \quo(R)^\times$.
We show that something stronger holds.

\begin{lemma} \label{l-symboliso}
  Let $I$, $I'$, $J$, and $J'$ be fractional right $R$-ideals with $J \subset I$ and $J' \subset I'$.
  If $I \cong I'$ and $J \cong J'$, then $\langle I / J \rangle = \langle I' / J' \rangle$.
\end{lemma}

\begin{proof}
  If $\alpha\colon I' \to I$ is an isomorphism, then $J' \cong \alpha(J')$ and thus  $I'/J' \cong \alpha(I')/\alpha(J')$.
  So we may assume $I=I'$.
  Since $J' \cong J$, there exists $x \in \quo(R)^\times$ such that $J' = xJ$.

  We first consider the case where $I=R$.
  Let $a$,~$b \in R^\bullet$ with $x = b^{-1}a$.
  Now $\langle R / J' \rangle = \langle R / b^{-1}a J \rangle = \langle bR / aJ \rangle$.
  Since $\langle R / bR \rangle = \mathbf 0$, we have $\langle bR / aJ \rangle = \langle R / bR \rangle + \langle bR /aJ \rangle = \langle R/aJ \rangle$.
  Next, since $J \subset R$ and $a \in R^\bullet$, we have $aJ \subset aR \subset R$.
  Hence $\langle R /aJ \rangle = \langle R/aR \rangle + \langle aR/aJ \rangle = \langle aR/aJ \rangle$.
  Finally, $\langle aR/aJ \rangle = \langle R / J \rangle$, and so, altogether $\langle R /J' \rangle = \langle R / J \rangle$.

  Now let $I$ be an arbitrary fractional right $R$-ideal.
  Since $R$ satisfies the left Ore condition with respect to $R^\bullet$, there exists $s \in R^\bullet$ such that $I \cup xI \subset s^{-1}R$.
  But then $\langle s^{-1}R/I \rangle + \langle I / xJ \rangle = \langle s^{-1}R/xI \rangle + \langle xI/xJ \rangle$.
  From the already established case, it follows that $\langle s^{-1}R/I \rangle = \langle R/sI \rangle= \langle R/sxI \rangle= \langle s^{-1}R/xI \rangle$.
  Therefore $\langle I / xJ \rangle = \langle xI / xJ \rangle = \langle I /J \rangle$, as claimed.
\end{proof}

The previous lemma shows that there is no harm in symbolically extending the definition of $\langle I / J \rangle$ to arbitrary fractional right $R$-ideals $I$ and $J$ (without requiring $J \subset I$) in the following way.
Recall that, for $I$ and $J$ fractional right $R$-ideals, there always exists some $x \in \quo(R)^\times$ such that $xJ \subset I$.

\begin{defi}
  If $I$ and $J$ are fractional right $R$-ideals, and $x \in \quo(R)^\times$ is such that $xJ \subset I$, we define
  \[
  \langle I / J \rangle = \langle I / xJ \rangle \quad\in K_0\modfl(R)/\qcp{R}.
  \]
\end{defi}

Obviously, the conclusion of \cref{l-symboliso} still holds with the extended definition.
We also need the following properties.

\begin{lemma}\label{l-symprop}
  Let $I$, $J$, and $K$ be fractional right $R$-ideals.
  \begin{enumerate}
    \item\label{l-symprop:sub} $\langle I / J \rangle + \langle J / K \rangle = \langle I / K \rangle$.
    \item\label{l-symprop:inv} $\langle I / J \rangle = - \langle J/I \rangle$.
  \end{enumerate}
\end{lemma}

\begin{proof}
  \ref*{l-symprop:sub} Let $x$,~$y \in \quo(R)^\times$ be such that $xJ \subset I$ and $yK \subset J$.
  Then $\langle I / K \rangle = \langle I / xyK \rangle = \langle I / xJ \rangle + \langle xJ / xyK \rangle = \langle I/J \rangle + \langle J/K\rangle$.

  \ref*{l-symprop:inv} Apply \ref*{l-symprop:sub} with $K=I$ and observe $\langle I/I\rangle = \mathbf 0$.
\end{proof}

\begin{prop} \label{p-stable-iso-to-equi}
  If $I$ and $J$ are fractional right $R$-ideals such that $R \oplus I \cong R \oplus J$, then $\langle R/I \rangle = \langle R/J \rangle$.
\end{prop}

\begin{proof}
  If $\udim R \ge 2$, then also $\udim I=\udim J =\udim R \ge 2$.
  It follows that $I \cong J$ (cf. \cite[Corollary 35.6]{levy-robson11}), and then \cref{l-symboliso} implies the claim.
  We may therefore assume $\udim R=1$, and hence that $R$ is a domain.
  We may without restriction assume $I$, $J \subset R$.

  We show more generally: If $I$, $J$,~$K$ are right $R$-ideals such that $J \oplus K \cong R \oplus I$, then $\langle K/I \rangle = \langle R / J \rangle$; this allows us to treat two cases that would otherwise occur as one.
  Let $\alpha\colon J \oplus K \to R \oplus I$ be an isomorphism, and denote by $p\colon R \oplus I \to R$ the projection onto $R$ along $I$. Write $\beta = p \circ \alpha$ for the composition, and note that $\ker(\beta) \cong I$.

  We view $J$ and $K$ as submodules of $J \oplus K$ in the canonical way.
  Since $\beta(J \oplus K) = R$, at least one of $\beta(J)$ and $\beta(K)$ must be nonzero.
  Assume first that $\beta(J) \ne \mathbf 0$.
  Then $\beta(J) \subset R$ is a right $R$-ideal because $R$ is a domain.
  Since $R$ is hereditary, $\beta(J)$ is projective, and we have $J \cong \beta(J) \oplus \ker(\beta|_J)$.
  But since $\udim J=1$ and $\beta(J) \ne \mathbf 0$, it follows that $\ker(\beta|_J) = \mathbf 0$ and hence $J \cong \beta(J)$ via $\beta$.
  Now consider the induced epimorphism
  \[
  \overline\beta\colon K \cong J \oplus K/J \to R / \beta(J).
  \]
  Since $\ker(\beta) \cap J = \mathbf 0$, it follows that $\ker(\overline \beta) = (\ker(\beta) + J)/ J \cong \ker(\beta) \cong I$.
  Then $K / \ker(\overline \beta) \cong R /\beta(J)$ implies $\langle K / I \rangle = \langle R / J \rangle$.

  If $\beta(J)=\mathbf 0$, then $\beta(K) \ne \mathbf 0$.
  We can reduce to the previous case by swapping the roles of $K$ and $J$, and obtain $\langle J / I \rangle = \langle R / K \rangle$.
  But then $\langle I/J \rangle = \langle K / R \rangle$.
  Adding $\langle R / I \rangle$ to both sides, we again find $\langle R /J \rangle = \langle K / I \rangle$.
\end{proof}

\begin{lemma} \label{l-k0-cycgen}
  \begin{enumerate}
    \item\label{l-k0-cycgen:constr} Let $n \in \bN_0$ and let $T_1$, $\ldots\,$,~$T_n$ be towers.
      If $I$ is a right $R$-ideal in the principal genus, then there exists a right $R$-ideal $J$ with $J \subset I$ such that $(I/J) = (T_1) + \cdots + (T_n)$.
    \item\label{l-k0-cycgen:towers} $\fagrp{\towers{R}} = \{\, (R/I) - (R/J) \mid \text{$I$, $J$ right $R$-ideals in the principal genus}\,\}$.
    \item\label{l-k0-cycgen:cgrp} $\cgrp{R} = \{\, \langle R/I \rangle \mid \text{$I$ a right $R$-ideal in the principal genus}\,\}.$
    \end{enumerate}
\end{lemma}

\begin{proof}
  \ref*{l-k0-cycgen:constr}
  We proceed by induction on $n$.
  If $n=0$, let $J=I$.
  Suppose now $n \ge 1$ and that the claim holds for $n-1$.
  By induction hypothesis, there exists a right $R$-ideal $J'$ such that $J' \subset I$ and $(I/J') = (T_1) + \dots + (T_{n-1})$.
  Since $(I/J') \in \famon{\towers{R}}$, we have $\Psi(J')=\Psi(I)=\Psi(R)$ by \subref{p-f:gen}.
  Hence $\rho(J',X) = \rho(R,X) > 0$ for all unfaithful simple modules $X$.
  By \cref{l-ex-uniserial}, there exists $J \subset J'$ such that $(J'/J)=(T_n)$ and the first claim follows.

  \ref*{l-k0-cycgen:towers}
  Since $\fagrp{\towers{R}}$ is the quotient group of $\famon{\towers{R}}$, the claim follows immediately from \ref*{l-k0-cycgen:constr}.

  \ref*{l-k0-cycgen:cgrp}
  By \ref*{l-k0-cycgen:towers}, any element of $\cgrp{R}=\fagrp{\towers{R}} /\qcp{R}$ can be written in the form $\langle R/I\rangle - \langle R/J\rangle$ with $I$,~$J$ right $R$-ideals in the principal genus.
  Let $a \in J^\bullet$.
  Then $-\langle R/J \rangle = \langle J / R \rangle = \langle J / aR \rangle$.
  Since $(R/aR)$ and $(R/J) \in \famon{\towers{R}}$, also $(J/aR) \in \famon{\towers{R}}$.
  Thus, by \ref*{l-k0-cycgen:constr}, there exists a right $R$-ideal $K$ such that $K \subset I$ and $(I/K) = (J/aR)$.
  Then $\langle R / K \rangle = \langle R / I \rangle + \langle I / K \rangle = \langle R / I \rangle - \langle R / J \rangle.$
\end{proof}

\begin{remark}
  If $R$ is a Dedekind prime ring then $\famon{\simple{R}}=\famon{\towers{R}}$.
  Thus, if $M$ is a module of finite length, then $(M)=(R/I)$ for some right $R$-ideal $I$ by \ref*{l-k0-cycgen:constr}.
  If $R$ is not a Dedekind prime ring, this is no longer true in general.
  For instance, if $V$ is an unfaithful simple module that is contained in a non-trivial cycle tower and $\rho(R,V) = 1$, then $(V) + (V)$ cannot be represented by a cyclic module.
  (If $I \subset R$ is a right $R$-ideal with $R/I \cong V$, then $\rho(I,V)=0$ by \cref{cor-genus-of-simples}.)
\end{remark}

We are now ready to prove the main theorem of this section.
Recall that the uniform dimension gives a homomorphism of abelian groups $\udim\colon K_0(R) \to \bZ$.
\begin{thm}\label{t-k0-iso-pcg}
  Let $R$ be an HNP ring.
  There exists an isomorphism of abelian groups
  \[
  \Phi\colon K_0\modfl(R)/\qcp{R} \isomto \ker(\udim) \subset K_0(R)
  \]
  such that $\Phi( \langle R/I\rangle) = [I] - [R]$ for all right $R$-ideals $I$, and
  \[
  \ker(\udim) \cong \pcg{R} \times \bigoplus_{T \in \towers{R}} \bZ^{\card{T}-1}.
  \]
  The isomorphism $\Phi$ restricts to an isomorphism
  \[
  \phi\colon \cgrp{R} \to \pcg{R}
  \]
  such that $\phi(\langle R / I\rangle) = [I] - [R] = \stzcls{I} - \stzcls{R}$ for all right $R$-ideals $I$ in the principal genus.
\end{thm}

\begin{proof}
  Since $\qcp{R} \subset \ker(\Phi_0)$ is immediate from the definitions and $\im \Phi_0 = \ker(\udim)$ by \cref{l-image}, the map $\Phi_0$ induces an epimorphism
  \begin{equation*}
    \Phi\colon K_0\modfl(R)/ \qcp{R} \to \ker(\udim).
  \end{equation*}
  We show that $\Phi$ is a monomorphism.

  From \cref{l-prinker}, we have $\ker(\Phi) \subset \fagrp{\towers{R}}/ \qcp{R}=\cgrp{R}$.
  \cref{l-k0-cycgen} implies $\cgrp{R} = \{\, \langle R / I \rangle \mid \text{$I$ is a right $R$-ideal in the principal genus}\,\}$.
  Let $I$ be a right $R$-ideal in the principal genus such that $\Phi(\langle R / I \rangle) = [I] - [R] = \mathbf 0$.
  Then $I \oplus R \cong R \oplus R$.
  \Cref{p-stable-iso-to-equi} implies $\langle R / I \rangle = \langle R / R \rangle = \mathbf 0$.
  Hence $\Phi$ is an isomorphism.

  Now let $\phi$ denote the restriction of $\Phi$ to $\cgrp{R}$.
  By \subref{p-f:ker}, we have $\Psi^+ \circ \Phi_0((T)) = 0$ for any tower $T$.
  Thus, the image of $\phi$ is contained in $\pcg{R}=\ker(\Psi^+)$.
  If $I$ is a right $R$-ideal in the principal genus, then $\phi(\langle R / I\rangle) = [I] - [R] = \stzcls{I} - \stzcls{R}$ since $\stzcls{I} = [I] - [B]$ and $\stzcls{R} = [R] - [B]$ for the unique $B$ in the base-point set with $\Psi(B)=\Psi(R)$.
  If $g \in \pcg{R}$, there exists a right $R$-ideal $I$ of $R$ such that $\Psi(I) = \Psi(R)$ and $\stzcls{I}=g+\stzcls{R}$, by independence of rank and Steinitz class.
  Then $\phi(\langle R/I\rangle) = g$ and thus $\phi$ is surjective.
\end{proof}

\begin{remark}
  \begin{enumerate}
  \item
    While the Steinitz class depends on the choice of base-point set, $\stzcls{I} - \stzcls{R}$ does not.
    The isomorphism $\cgrp{R} \cong \pcg{R}$ is independent of the choice of base-point set.

  \item
    To obtain a transfer homomorphism, we will later have to impose the condition that every stably free right $R$-ideal is free.
    In this case, $[I]=0$ implies that $I$ is principal, and hence $\langle R / I \rangle = \mathbf 0$.
    Thus, under this additional condition, the proof of the injectivity of $\Phi$ simplifies significantly.
  \end{enumerate}
\end{remark}

The following observation is useful in connection with \subref{thm-transfer:cp} and \labelcref{thm-transfer:cp-weak}.
See \cite[Theorems 2.3.11 and 2.4.8]{ghk06} for various equivalent characterizations of commutative Krull monoids.

\begin{prop} \label{p-krull}
  Let $R$ be an HNP ring such that every stably free right $R$-ideal is free.
  Then the monoid $\cp{R}$ is a commutative Krull monoid, and the inclusion $\cp{R} \to \famon{\towers{R}}$ is a cofinal divisor homomorphism.
\end{prop}

\begin{proof}
  By \cref{l-prinker}, the monoid $\cp{R}$ is a submonoid of the free abelian monoid $\famon{\towers{R}}$.
  If we show that the inclusion is a divisor homomorphism, it follows that $\cp{R}$ is a Krull monoid.
  Thus we have to show: If $a$,~$b \in R^\bullet$ with $(R/aR) - (R/bR) \in \famon{\towers{R}}$, then $(R/aR)-(R/bR) \in \cp{R}$.
  Since $(R/aR) - (R/bR) \in \famon{\towers{R}}$, there exists a right $R$-ideal $I$ such that $(R/I) = (R/aR) - (R/bR)$ (by \cref{l-k0-cycgen}).
  But then $[I]-[R] = ([aR]-[R]) + ([bR] - [R])$ by \cref{l-ex-phi0}, and therefore $[I]=[R]$.
  Hence, $I$ is stably free.
  Thus, by our assumption, $I=cR$ with $c \in R^\bullet$ and $(R/aR) - (R/bR) = (R/cR) \in \cp{R}$.

  To show that the inclusion is cofinal, we need to show: If $T$ is a tower, then there exists an element $a \in R^\bullet$ such that $(R/aR)=(T) + (X)$ with $(X) \in \famon{\towers{R}}$.
  Let $T$ be a tower.
  By \cref{l-ex-uniserial}, there exists a right $R$-ideal $I$ such that $(R/I)=(T)$.
  Any $a \in I^\bullet$ has the required property.
\end{proof}

\subsection{A crucial lemma}
We now prove a combinatorial lemma that will yield \cref{l-int-mod} about modules of finite length below.
The latter lemma will be vital in the construction of a transfer homomorphism in \cref{sec-transfer-hom}.
The combinatorial lemma says that, if a finite cyclic group $\sC_n$ with generator $g$ is covered by $l$ arithmetic progressions with difference $g$, then there exist pairwise disjoint (possibly empty) starting segments of these arithmetic progressions that cover $\sC_n$.
\begin{lemma} \label{lemma-comb}
  Let $n \in \bN$, let $(\sC_n,+)$ be a cyclic group of order $n$, and let $g$ be a generator of $\sC_n$.
  Let $l \in \bN$, let $(k_1,\ldots,k_l) \in \bN_0^l$, and let $(a_1,\ldots,a_l) \in \sC_n^{l}$.
  Suppose that
  \[
  \bigcup_{i=1}^l \big\{\, a_i + j g \mid j \in [0,k_i] \,\big\} = \sC_n.
  \]
  Then there exist $m_1$, $\ldots\,$,~$m_l \in \bN_0$ with $m_i \in [0,k_i+1]$ such that $m_1+\cdots+m_l=n$ and
  \[
  \bigcup_{i=1}^l \big\{\, a_i + j g \mid j \in [0,m_i-1] \,\big\} = \sC_n.
  \]
\end{lemma}

\begin{proof}
  Set $A_i = \{\, a_i + jg \mid j \in [0,k_i] \,\}$ for $i \in [1,l]$.
  We prove the claim by induction on $l$.
  In case $l=1$, we necessarily have $k_1 \ge n-1$.
  The claim follows with the choice $m_1 = n$.
  Suppose now that $l > 1$ and that the claim holds for smaller $l$.
  If there exists a proper subset $\Omega \subsetneq [1,l]$ such that $\smash\bigcup_{i \in \Omega} A_i = \sC_n$, then we are done by induction hypothesis, setting $m_i = 0$ for $i \in [1,l]\setminus \Omega$.
  We may therefore from now on assume that each $A_i$ contains an element which is not contained in $\smash\bigcup_{j=1,j \ne i}^n A_j$.
  For each $i \in [1,l]$, let $m_i \in [1,k_i+1]$ be maximal such that $a_i+(m_i-1)g$ is not contained $\smash\bigcup_{j=1,j\ne i}^n A_j$.
  We claim that this choice works.

  Let $M=\smash\bigcup_{i=1}^l M_i$ with $M_i=\{\, a_i + jg \mid j \in [0,m_i-1]\,\}$.
  First we show that indeed $M=\sC_n$.
  Clearly $M \ne \emptyset$ since $a_i \in M_i$ for all $i \in [1,l]$.
  It therefore suffices to show: If $b \in M$, then also $b + g \in M$.
  Let $b \in M$.
  Then $b = a_i + t g$ for some $i \in [1,l]$ and $t \in [0,m_i-1]$.
  If $t < m_i-1$, then $b+g=a_i+(t+1)g \in M$.
  Suppose now that $t=m_i-1$.
  Since $\smash\bigcup_{j=1}^l A_j = \sC_n$, there exists at least one $j \in [1,l]$ such that $b+g \in A_j$.
  If $A_i$ is the only set containing $b+g$, then, by choice of $m_i$, necessarily $m_i=k_i+1$.
  But then $A_i \subset M$ and hence $b+g \in M$.
  Otherwise, we may suppose $j \ne i$.
  Thus $b+g = a_j + s g$ for some $s \in [0,k_j]$.
  If $s > 0$, then also $a_i + (m_i -1) g = b = a_j + (s-1)g \in A_j$, contradicting the fact that $a_i + (m_i -1) g$ is only contained in $A_i$.
  Thus $s=0$, and we have $b+g=a_j \in M$.

  We now show $m_1+\cdots+m_r=n$, that is, that the sets $M_1$,~$\ldots\,$,~$M_l$ are pairwise disjoint.
  Suppose that there exist distinct $i$,~$j \in [1,l]$ and $s \in [0,m_i-1]$, $t \in [0,m_j-1]$ such that $a_i + s g = a_j + t g$.
  We may without loss of generality assume $s \ge t$, and subsequently $t=0$.
  Thus $a_i + sg = a_j \in A_i \cap A_j$.
  Since $a_i + (m_i -1)g \notin A_j$ by our choice of $m_i$, we must have $m_j - 1 < m_i - 1 - s$.
  But then $A_j \subset A_i$, contradicting the hypothesis that $A_j$ contains an element not contained in $A_i$.
\end{proof}

Every module $M$ of finite length has a unique decomposition $M=F\oplus C$ where all composition factors of $F$ belong to faithful towers and all composition factors of $C$ belong to cycle towers (see \cite[Theorem 41.1]{levy-robson11}).
If $U$ is an indecomposable module of finite length all of whose composition factors belong to cycle towers, then $U$ is uniserial (see \cite[Theorem 41.2]{levy-robson11}).
Moreover, the composition factors of $U$ correspond to a segment of a repetition of a unique cycle tower $T$.
Explicitly, if $U=U_0 \supsetneq U_1 \supsetneq U_2 \supsetneq \dots \supsetneq U_n=\mathbf 0$ are the submodules of $U$, then $(U_{i}/U_{i+1})$ is the successor of $(U_{i-1}/U_i)$ for all $i \in [1,n-1]$.
Since the composition factors of $U$ may be labeled cyclically, the following lemma is an easy consequence of the previous one.
Due to its importance in the sequel, we state the proof anyway.

\begin{lemma} \label{l-int-mod}
  Let $M$ be a module of finite length.
  Suppose that there is a cycle tower $T$ such that $(M) = (T) + (X)$ with $(X) \in \famon{\simple{R}}$.
  \begin{enumerate}
    \item\label{l-int-mod:lower} There exists a submodule $K$ of $M$ such that $(K)=(T)$.
    \item\label{l-int-mod:upper} There exists a submodule $L$ of $M$ such that $(M/L)=(T)$.
  \end{enumerate}
\end{lemma}

\begin{proof}
  We show \ref*{l-int-mod:lower}; the proof of \ref*{l-int-mod:upper} is analogous.
  Grouping indecomposable summands by their towers, we see that $M=M_T \oplus M_0$ with all composition factors of $M_T$ belonging to $T$, and all composition factors of $M_0$ belonging to towers other than $T$.
  Without restriction, we may assume $M=M_T$.

  Let $n$ be the length of $T$, and label a set of representatives for the isomorphism classes of simple modules of $T$ by $W(\res 1)$, $\ldots\,$,~$W(\res n)$, where $\res a$ denotes the residue class of $a$ in $\bZ/n \bZ$, and $W(\res{a+1})$ is the successor of $W(\res a)$.
  Let $M = U_1 \oplus \cdots \oplus U_l$ with uniserial modules $U_1$, $\ldots\,$,~$U_l$.
  For $i \in [1,l]$, let $k_i+1$ be the length of $U_i$, and denote the submodules of $U_i$ by
  \[
  U_i = U_{i,0} \supsetneq U_{i,1} \supsetneq \dots \supsetneq U_{i,k_i} \supsetneq U_{i,k_i+1} = \mathbf 0.
  \]
  Let $\res{a_i} \in \bZ/n \bZ$ be such that $U_{i,k_i} \cong W(\res{a_i})$.
  Then $U_{i,j}/U_{i,j+1} \cong W(\res{a_i-(k_i-j)})$ for all $i \in [1,l]$ and $j \in [0,k_i]$.
  By our assumption, for every $\res a \in \bZ/n\bZ$, the isomorphism class of $W(\res a)$ appears among $(U_{i,j}/U_{i,j+1})$ for some $i \in [1,l]$ and $j \in [0,k_i]$.
  \Cref{lemma-comb} implies that there exist $m_i \in [0,k_i+1]$ with $m_1+ \cdots + m_l = n$ such that
  \[
  \bigcup_{i=1}^l \big\{\, \res{ a_i - j} \mid j \in [0,m_i-1] \,\big\} = \bZ/n\bZ.
  \]
  It follows that $\sum_{i=1}^l \sum_{j=0}^{m_i-1} (W(\res{a_i-j})) = (T)$ in $K_0\modfl(R)$.
  Choosing $K = U_{1,k_1-(m_1-1)} \oplus \cdots \oplus U_{l,k_l-(m_l-1)} \subset M$, the claim holds.
\end{proof}

\subsection{Tower-maximal right $R$-ideals.}

In this subsection, we introduce subsets $\pcgm{R}$ of $\pcg{R}$ and $\cgrpm{R}$ of $\cgrp{R}$, which will play an important role in the construction of a transfer homomorphism.
Moreover, we show that these sets are preserved under passage to a Morita equivalent ring or a Dedekind right closure.

\begin{lemma} \label{l-mgen}
  Let $Q \subset P$ be finitely generated projective modules.
  Consider the following statements.
  \begin{equivenumerate}
    \item \label{l-mgen:mod} $\Psi(Q)=\Psi(P)$ and $Q$ is a proper submodule of $P$ maximal with respect to this property.
    \item \label{l-mgen:tow} $(P/Q) = (T)$ for some tower $T$.
  \end{equivenumerate}
  Then \ref*{l-mgen:tow}${}\Rightarrow{}$\ref*{l-mgen:mod}.
  If every faithful tower is trivial, then also \ref*{l-mgen:mod}${}\Rightarrow{}$\ref*{l-mgen:tow}.
\end{lemma}

\begin{proof}
  \ref*{l-mgen:tow}${}\Rightarrow{}$\ref*{l-mgen:mod}:
  By \subref{p-f:gen}, we have $\Psi(Q)=\Psi(P)$.
  Since $(P/Q)=(T)$, the same proposition implies that the module $Q$ is maximal with respect to this property.

  Suppose every faithful tower is trivial.
  \ref*{l-mgen:mod}${}\Rightarrow{}$\ref*{l-mgen:tow}:
  By \subref{p-f:gen}, we have $(P/Q)=(T_1) + \cdots + (T_n)$ for some towers $T_1$, $\ldots\,$,~$T_n$ with $n \ge 1$.
  If one of the $T_i$ is a cycle tower, then \subref{l-int-mod:upper} implies that there exists a module $M$ such that $Q \subset M \subset P$ and $(P/M)=(T_i)$.
  If all of the $T_i$ are faithful towers, let $M$ be a maximal submodule of $P$ with $Q \subset M \subset P$ and let $i \in [1,n]$ be such that $(P/M)$ belongs to $(T_i)$.
  Since $T_i$ is trivial by assumption, again $(P/M)=(T_i)$.

  Since $(P/M)=(T_i)$, \subref{p-f:gen} implies $\Psi(M)=\Psi(P)$.
  Since $Q$ is maximal with this property, $Q=M$.
  Hence $n=1$ and $(P/Q)=(T_i)=(T_1)$.
\end{proof}

\begin{defi}
  Let $I$ be a right $R$-ideal.
  \begin{enumerate}
    \item $I$ is \emph{maximal in the principal genus}  if $I \subsetneq R$ is maximal among proper right $R$-ideals in the principal genus.
    \item $I$ is \emph{tower-maximal} if $(R/I) = (T)$ for some tower $T$.
  \end{enumerate}
\end{defi}

Note that, while $(R/I)=(T)$ for a tower-maximal right $R$-ideal $I$, it is not necessary for $R/I$ to be uniserial.

By \cref{l-mgen}, every tower-maximal right $R$-ideal is maximal in the principal genus.
If every faithful tower is trivial, the converse is also true.
However, in general, a right $R$-ideal which is maximal in the principal genus need not be tower-maximal.
(Counterexamples are given in the construction in \cref{p-nonhf} and by $x(x-y)R$ in \cref{e-no-ft}.)

\begin{defi} \label{d-pcgm}
  Let $\cgrpm{R} = \{\,\langle T \rangle \in \cgrp{R} \mid T \in \towers {R} \,\}$ and
  \begin{align*}
  \pcgm{R} &= \{\, \stzcls{I} - \stzcls{R} \in \pcg{R} \mid \text{$I$ is a tower-maximal right $R$-ideal}\,\} \\
           &= \phi(\cgrpm{R}).
  \end{align*}
\end{defi}

Before considering the behavior of $\pcgm{R}$ under Morita equivalence and passage to a Dedekind right closure, let us recall the notion of the genus class group (see \cite[\S25]{levy-robson11}).
Let $\mathbf 0 \ne P$ be a finitely generated projective module, and define
\[
\gcg{P} = \left\{\, [X] \in K_0(R) \mid \Psi(X)=\Psi(P)  \,\right\}
\]
Since any finitely generated projective module $X$ with $\udim(X)=\udim(P)$ is isomorphic to a submodule of $P$, in fact $\gcg{P} = \{\, [X] \in K_0(R) \mid X \subset P,\, \Psi(X)=\Psi(P) \,\}$.
On $\gcg P$ we define an operation $\boxplus_P$ by $[Z] = [X] \boxplus_P [Y]$ if and only if $Z \oplus P \cong X \oplus Y$.
Thus, $[X] \boxplus_P [Y] = [X] + [Y] - [P]$, where the operation on the right hand side is the usual one in $K_0(R)$.
With the operation $\boxplus_P$, the set $\gcg{P}$ is an abelian group with zero element $[P]$.
The group $\gcg P$ is the \emph{genus class group} of $P$.

If $\mathbf 0 \ne P$,~$Q$ are two finitely generated projective modules, then $(\gcg{P}, \boxplus_P) \cong (\gcg Q, \boxplus_Q)$.
The isomorphism is given as follows: If $[X] \in \gcg P$, then there exists a finitely generated projective module $Y$ such that $\Psi(Y) = \Psi(Q)$ and $\stzcls{Y} = \stzcls{X} + \stzcls{Q} - \stzcls{P}$.
The class $[X]$ in $\gcg P$ is mapped to $[Y]$ in $\gcg Q$.

Let $\gcgm P$ denote the subset of $\gcg P$ consisting of all classes which contain a submodule $M$ of $P$ such that $(P/M)=(T)$ for some tower $T$. (Note that $\Psi(M)=\Psi(P)$ by \subref{p-f:gen}.)
Recall that $\gcg R \cong \pcg R$ through $[I] \mapsto \stzcls{I}-\stzcls{R}=[I]-[R]$.
The set $\gcgm R$ is mapped to $\pcgm R$ under this isomorphism

\begin{lemma} \label{l-gcg-progen}
  Let $P$ and $Q$ be progenerators.
  Under the isomorphism $\gcg P \to \gcg Q$, the image of $\gcgm P$ is $\gcgm Q$.
\end{lemma}

\begin{proof}
  Let $\alpha \colon \gcg P \to \gcg Q$ denote the isomorphism as above.
  It suffices to show $\alpha(\gcgm P) \subset \gcgm Q$.
  The other inclusion follows by symmetry.
  Let $M$ be a submodule of $P$ such that $(P/M)=(T)$ for some tower $T$.
  Since $Q$ is a progenerator, \cref{l-ex-uniserial} implies that that there exists $N \subset Q$ such that $(Q/N)=(T)$.
  Thus, $[N] \in \gcgm Q$.
  Applying $\Phi_0$ to $(P/M)=(Q/N)$, we find in particular $\stzcls{M} - \stzcls{P} = \stzcls{N} - \stzcls{Q}$, and hence $\alpha([M])=[N]$.
\end{proof}

\begin{thm} \label{t-morita}
  Let $R$ and $S$ be Morita equivalent HNP rings.
  Then $\pcg R \cong \pcg S$, and under this isomorphism, the image of $\pcgm R$ is $\pcgm S$.
\end{thm}

\begin{proof}
  Since $R$ and $S$ are Morita equivalent, there exists a right $R$-module $P_R$ such that $P$ is a progenerator and $\End(P_R)=S$.
  Let ${}_R Q_S = \Hom(P_R,R_R)$.
  The category equivalence from right $R$-modules to right $S$-modules is given by the functor $- \otimes_R Q$.
  The induced bijection between simple $R$-modules and simple $S$-modules preserves the tower structure; that is, it preserves the towers, their lengths, and their (cycle or faithful) types (see \cite[Theorem 19.4]{levy-robson11}).
  
  Let $V$ be an unfaithful simple right $R$-module and let $X$ be a finitely generated projective right $R$-module.
  Then $\rho(X,V)$ is the largest $t \in \bN_0$ such that there exists an epimorphism $X \to V^{(t)}$.
  It follows that $\rho(X,V) = \rho(X \otimes_R Q, V \otimes_R Q)$.
  Since also $\udim_R(X) = \udim_S (X \otimes_R Q)$, we have $\Psi_R(X) = \Psi_S(X \otimes_R Q)$.

  Thus, since the Morita equivalence induces an isomorphism of the lattice of right $R$-submodules of $P$ and the lattice of right ideals of $S$, there exists an isomorphism of abelian groups $\operatorname{gcg}_R(P) \to \operatorname{gcg}_S(S)$ satisfying $[X] \mapsto [X \otimes_R Q]$ and this isomorphism maps $\gcgm{P}$ to $\gcgm{S}$.
  By the previous lemma, $\operatorname{gcg}_R(P) \cong \operatorname{gcg}_R(R) \cong \pcg{R}$ and $\operatorname{gcg}_S(S) \cong \pcg{S}$, and the composed isomorphism $\pcg{R} \to \pcg{S}$ maps $\pcgm{R}$ to $\pcgm{S}$.
\end{proof}

If $S$ is a Dedekind right closure of $R$, we denote by $\tau\colon K_0(R) \to K_0(S)$ the homomorphism satisfying $\tau([X])=[X\otimes_R S]$.
This homomorphism restricts to an isomorphism $\pcg{R} \to \pcg{S}$, which we also denote by $\tau$.

\begin{thm} \label{t-closure}
  Let $R$ be an HNP ring and let $S$ be a Dedekind right closure of $R$.
  Then there exists an isomorphism $\tau\colon\pcg{R} \to \pcg{S}$ satisfying $[X] \mapsto [X\otimes_R S]$.
  This isomorphism maps $\pcgm{R}$ to $\pcgm{S}$.
\end{thm}

\begin{proof}
  The isomorphism $\tau\colon\pcg R \to \pcg S$ is established in \cite[Theorem 35.19]{levy-robson11}.

  We first show $\tau(\pcgm{R}) \subset \pcgm{S}$.
  Let $g \in \pcgm{R}$.
  Then $g = [I] - [R]$ for some tower-maximal right $R$-ideal $I$.
  Thus $(R/I)=(T)$ for some tower $T$.
  Suppose that $T$ consists of pairwise distinct isomorphism classes of simple modules $(V_1)$, $\ldots\,$,~$(V_n)$ such that $(V_{i+1})$ is the unfaithful successor of $(V_i)$ for all $i \in [1,n-1]$.
  Since $S$ is a Dedekind right closure of $R$, there exists a unique $j \in [1,n]$ such that $V_j \otimes_R S$ is a simple module, and for all $i \ne j$ we have $V_i \otimes_R S = \mathbf 0$.
  Thus, $S/IS\cong (R/I) \otimes_R S$ is simple, and $IS$ is a maximal right $S$-ideal.
  Hence $\tau(g) = \tau([I]-[R]) = \tau([I]) - \tau([R]) = [IS]-[S] \in \pcgm{S}$.

  We now show $\pcgm{S} \subset \tau(\pcgm{R})$.
  Let $g \in \pcgm{S}$.
  Then there exists a maximal right $S$-ideal $J$ such that $g = [J] - [S]$.
  By \cite[Theorem 13.12(b)]{levy-robson11} there exists a simple $R$-module $V$ such that $V \otimes_R S \cong S/J$.
  Let $T$ be the $R$-tower containing $V$.
  By \cref{l-ex-uniserial}, there exists a right $R$-ideal $I$ such that $R/I$ is uniserial and $(R/I) =(T)$.
  Hence $[I]-[R] \in \pcgm R$.
  Since $S$ is a Dedekind right closure of $R$, it follows that $S/(I\otimes_R S) \cong (R/I) \otimes_R S \cong S/J$.
  Hence, by Schanuel's Lemma, $[I \otimes_R S] = [J]$ in $K_0(S)$.
  Thus $\tau([I]-[R]) = \tau([I]) - \tau([R]) = [J] - [S] = g$.
\end{proof}

\section{Transfer homomorphism}
\label{sec-transfer-hom}

In this section, we show that if $R$ is a bounded HNP ring and every stably free right $R$-ideal is free, then there exists a transfer homomorphism from $R^\bullet$ to the monoid of zero-sum sequences over the subset $\cgrpm{R}$ of the class group $\cgrp{R}$.
Moreover, the catenary degree in the fiber of this transfer homomorphism is at most $2$.
The transfer homomorphism is established in \cref{thm-transfer}; its catenary degree in the fibers in \cref{thm-catenary}.

We work in a somewhat more general setting than that of bounded HNP rings.
Thus, in this section, we assume that the (non-Artinian) HNP ring $R$ has the following two additional properties.
\begin{enumerate}[label=\textbf{(F\arabic*)},ref=\textup{(F\arabic*\textup)},leftmargin=*]
\item\label{c-ft} Every faithful tower is trivial.
\item\label{c-fo} If $V$ and $W$ are simple modules that belong to towers belonging to different classes in $\cgrp{R}$, then $\Ext^1_R(V,W) = \mathbf 0$.
\end{enumerate}
Sometimes we will also require the following, stronger, replacement for~\labelcref{c-fo}:
\begin{enumerate}[label=\textbf{(F\arabic*s)},ref=\textup{(F\arabic*s)},leftmargin=*]
\setcounter{enumi}{1}
\item\label{c-fos} If $V$ and $W$ are simple modules that belong to different towers, then $\Ext^1_R(V,W) = \mathbf 0$.
\end{enumerate}

If $R$ is bounded or a Dedekind prime ring, then it trivially has property~\labelcref{c-ft}.
Consider property~\labelcref{c-fos}.
If $W$ is unfaithful, then $\Ext^1_R(V,W) \ne \mathbf 0$ implies that $W$ is an unfaithful successor to $V$, which is unique up to isomorphism.
In particular, in this case $V$ and $W$ belong to the same tower.
Thus, the condition in property~\labelcref{c-fos} is trivially satisfied whenever $W$ is unfaithful.
In particular, every bounded HNP ring has property~\labelcref{c-fos}.
If $\cgrp{R}=\mathbf 0$, then $R$ trivially has property~\labelcref{c-fo}, but not necessarily property~\labelcref{c-fos}.

In the presence of property \labelcref{c-ft}, every unfaithful simple module has a unique successor, which is again unfaithful.
Thus, for $V$ and $W$ belonging to different towers, $\Ext^1_R(V,W) = \mathbf 0$ unless both $V$ and $W$ are faithful.
Hence, in this case, \labelcref{c-fo}, respectively \labelcref{c-fos}, is equivalent to the same property where $V$ and $W$ are both faithful.

If $R$ is a prime principal ideal ring, then it is a Dedekind prime ring with $\cgrp{R}\cong \pcg{R} =\mathbf 0$ and every stably free right $R$-ideal is free.
In summary, sufficient conditions for $R$ to have properties \labelcref{c-ft,c-fo} are:
\begin{propenumerate}
  \item $R$ is a prime principal ideal ring, or
  \item $R$ is a bounded HNP ring.
\end{propenumerate}
In the latter case $R$ even has property \labelcref{c-fos}.

We begin with some lemmas that are mostly consequences of \cref{l-int-mod}.
\begin{lemma} \label{int}
  Let $J \subset I$ be right $R$-ideals and suppose that there is a tower, say $T$, such that $(I/J) = (T) + (X)$ with $(X) \in \famon{\simple{R}}$.
  \begin{enumerate}
    \item\label{int:lower} There exists a right $R$-ideal $K$ with $J \subset K \subset I$ such that $(K/J)=(T')$ for a tower $T'$ with $\langle T' \rangle = \langle T\rangle$.
      Moreover, $T$ and $T'$ are of the same type \textup{(}faithful or cycle tower\textup{)}.
      If, in addition, $T$ is a cycle tower or \labelcref{c-fos} holds, then it is possible to choose $K$ such that $(K/J)=(T)$.
  \end{enumerate}
  \vspace{-\topsep}
  \begin{enumerate}[topsep=0pt,label=\textup{(\arabic*')}]
    \item\label{int:upper} There exists a right $R$-ideal $L$ with $J \subset L \subset I$ such that $(I/L)=(T')$ for a tower $T'$ with  $\langle T' \rangle = \langle T\rangle$.
      Moreover, $T$ and $T'$ are of the same type.
      If, in addition, $T$ is a cycle tower or \labelcref{c-fos} holds, then it is possible to choose $L$ such that $(I/L)=(T)$.
  \end{enumerate}
\end{lemma}

\begin{proof}
  We show \ref*{int:lower}; the proof of \ref*{int:upper} is analogous.
  Suppose first that $T$ is a cycle tower.
  By \subref{l-int-mod:lower}, there exists a submodule $M$ of $I/J$ such that $(M)=(T)$.
  Lifting $M$ to a right $R$-ideal $J \subset K \subset I$ under the canonical epimorphism $I \to I/J$, we have $(K/J)=(T)$, and we are done.

  Suppose $T$ is a faithful tower.
  All faithful towers are trivial by \labelcref{c-ft}.
  Thus, \cite[Corollary 41.4]{levy-robson11} implies $I/J \cong F \oplus C$, where $F$ is a module all of whose composition factors are faithful, and $C$ is a module all of whose composition factors are unfaithful.
  Let $F=F_0 \supsetneq F_1 \supsetneq F_2 \supsetneq \cdots \supsetneq F_n=\mathbf 0$ be a composition series of $F$.
  The tower $T$ consists of a single isomorphism class $(V)$ of a faithful simple module $V$.
  Let $i \in [1,n]$ be maximal with $\langle F_{i-1}/F_i \rangle = \langle V \rangle$.
  If $i=n$ then $\langle F_{n-1} \rangle = \langle V \rangle$, and we are done.
  If $i <n$, then $\Ext^1_R(F_{i-1}/F_i, F_i/F_{i+1}) = \mathbf 0$ by \labelcref{c-fo}.
  Hence $\mathbf 0 \to F_i/F_{i+1} \to F_{i-1}/F_{i+1} \to F_{i-1}/F_i \to \mathbf 0$ splits, and we may change the composition series so that $\langle F_i/F_{i+1} \rangle = \langle V \rangle$.
  Inductively, we again achieve $\langle F_{n-1} \rangle = \langle V \rangle$.

  In case \labelcref{c-fos} holds, we argue analogously, but we consider isomorphism classes of simple modules instead of classes in $\cgrp{R}$.
  In this way we achieve $(F_{n-1})=(V)$.
\end{proof}

\begin{lemma} \label{l-towerlift}
  Let $J \subset I$ be right $R$-ideals, and suppose that $(I/J) = (T_1) + \cdots + (T_n) + (X)$ with towers $T_1$, $\ldots\,$,~$T_n$ and $(X) \in \famon{\simple{R}}$.

  \begin{enumerate}[series=tl-basic]
    \item\label{l-towerlift:weak}
      There exists a chain of right $R$-ideals $I \supset I_0 \supset I_1 \supset \cdots \supset I_{n-1} \supset I_n=J$
      such that $(I_{i-1}/I_i)$ consists of a single tower and $\langle I_{i-1}/I_i \rangle = \langle T_i \rangle$ for all $i \in [1,n]$.
  \end{enumerate}
  \vspace{-\topsep}
  \begin{enumerate}[topsep=0pt,series=tl-basic-prime,label=\textup{(\arabic*')}]
    \item\label{l-towerlift:weak-upper}
      There exists a chain of right $R$-ideals $I = I_0 \supset I_1 \supset \cdots \supset I_{n-1} \supset I_n \supset J$
      such that $(I_{i-1}/I_i)$ consists of a single tower and $\langle I_{i-1}/I_i \rangle = \langle T_i \rangle$ for all $i \in [1,n]$.
  \end{enumerate}
  \begin{enumerate}[topsep=0pt,tl-basic]
    \item\label{l-towerlift:strong}
      Suppose that $T_1$, $\ldots\,$,~$T_n$ are cycle towers or \labelcref{c-fos} holds.
      Then there exists a chain of right $R$-ideals $I \supset I_0 \supset I_1 \supset \cdots \supset I_{n-1} \supset I_n=J$ such that $(I_{i-1}/I_i) =(T_i)$ for all $i \in [1,n]$.
  \end{enumerate}
  \begin{enumerate}[topsep=0pt,tl-basic-prime]
    \item\label{l-towerlift:strong-upper}
      Suppose that $T_1$, $\ldots\,$,~$T_n$ are cycle towers or \labelcref{c-fos} holds.
      Then there exists a chain of right $R$-ideals $I = I_0 \supset I_1 \supset \cdots \supset I_{n-1} \supset I_n \supset J$ such that $(I_{i-1}/I_i) =(T_i)$ for all $i \in [1,n]$.
  \end{enumerate}
  \begin{enumerate}[topsep=0pt,tl-basic]
    \item\label{l-towerlift:full}
      If $(X)=\mathbf 0$, then there exists a chain of right $R$-ideals $I = I_0 \supset I_1 \supset \cdots \supset I_{n-1} \supset I_n=J$ and a permutation $\sigma \in \fS_n$ such that $(I_{i-1}/I_i) = (T_{\sigma(i)})$ and $\langle I_{i-1}/I_i \rangle = \langle T_i \rangle$ for all $i \in [1,n]$.
  \end{enumerate}
\end{lemma}
\begin{proof}
  \ref*{l-towerlift:weak}, \ref*{l-towerlift:weak-upper}, \ref*{l-towerlift:strong}, and \ref*{l-towerlift:strong-upper} follow by iterated application of the previous lemma.

  \ref*{l-towerlift:full} By \ref*{l-towerlift:weak}, there exists a chain of right $R$-ideals $I \supset I_0 \supset I_1 \supset \cdots \supset I_{n-1} \supset I_n=J$ and towers $T_1'$, $\ldots\,$,~$T_n'$ with $(I_{i-1}/I_i)=(T_i')$ and $\langle T_i'\rangle=\langle T_i\rangle$ for all $i \in [1,n]$.
  Since
  \[
  (I/J) = (T_1) + \cdots + (T_n) = (I/I_0) + (T_1') + \cdots + (T_n') \in \famon{\towers{R}},
  \]
  it follows that $I_0=I$ and that there exists a permutation $\sigma \in \fS_n$ with $T_i'=T_{\sigma(i)}$ for all $i \in [1,n]$.
\end{proof}
In the previous lemma, if the conditions of \labelcref{l-towerlift:strong} are satisfied, the statement of \labelcref{l-towerlift:full} holds with $\sigma=\id$ as a trivial consequence of \labelcref{l-towerlift:strong}.
However, we will need to make use of \labelcref{l-towerlift:full} even in the case where faithful towers may exist and \labelcref{c-fos} may not hold.

Recall, from \cref{t-k0-iso-pcg}, that there exists an isomorphism $\phi\colon \cgrp{R} \isomto \pcg{R}$.
If $T$ is a tower and $I$ is an $R$-ideal with $(R/I) = (T)$, then $\phi\big(\langle T \rangle\big) = \stzcls{I} - \stzcls{R} = [I]-[R]\in \pcgm{R} \subset \pcg{R}$,
with $\pcgm{R}$ as in \cref{d-pcgm}.

\begin{lemma} \label{l-free}
  Suppose that every stably free right $R$-ideal is free.
  If $J \subset I$ are right $R$-ideals such that $(I/J) = (T_1) + \cdots + (T_n)$ with towers $T_1$, $\ldots\,$,~$T_n$ and $\langle T_1 \rangle + \cdots + \langle T_n \rangle = \mathbf 0$ in $\cgrp{R}$, then $I$ is free if and only if $J$ is free.
\end{lemma}

\begin{proof}
  By \subref{p-f:gen}, we have $\Psi(J)=\Psi(I)$.
  Since
  \[
  \stzcls{J} - \stzcls{I} = \phi\big(\langle R/J \rangle - \langle R/I\rangle \big) = \phi\big(\langle I / J \rangle\big) = \phi\big(\langle T_1\rangle + \cdots + \langle T_n\rangle\big) = \mathbf 0,
  \]
  we also have $\stzcls{J} = \stzcls{I}$.
  Thus $I$ and $J$ are stably isomorphic.
  Hence $I$ is stably free if and only if $J$ is stably free.
  By our assumption that all stably free right $R$-ideals are free, therefore $I$ is free if and only if $J$ is free.
\end{proof}

\begin{thm} \label{thm-transfer}
  Let $R$ be an HNP ring with properties \labelcref{c-ft,c-fo} and such that every stably free right $R$-ideal is free.
  \begin{enumerate}
  \item \label{thm-transfer:zss}
    There exists a transfer homomorphism $\overline\theta\colon R^\bullet \to \cB(\cgrpm{R})$ to the monoid of zero-sum sequences $\cB(\cgrpm{R})$, given as follows:
    For $a \in R^\bullet$ with $(R/aR) = (T_1) + \cdots + (T_n)$ for $n \in \bN_0$ and towers $T_1$, $\ldots\,$,~$T_n$,
    \[
    \overline\theta(a) = \langle T_1\rangle \cdots\langle T_n\rangle.
    \]
    \textup{(}The product is the formal product in the free abelian monoid over $\cgrpm{R}$.\textup{)}

  \item \label{thm-transfer:cp}
    If \labelcref{c-fos} holds, then $\theta\colon R^\bullet \to \cp{R}$, $a \mapsto (R/aR)$ is a transfer homomorphism.

  \item \label{thm-transfer:cp-weak}
    If $\langle V \rangle =\mathbf 0 \in \cgrp{R}$ for every faithful simple module $V$, then $\theta$ is a weak transfer homomorphism.
  \end{enumerate}
\end{thm}

\begin{proof}
  We start by constructing the involved homomorphisms.
  Let $\theta\colon (R^\bullet,\cdot) \to (\cp{R},+)$ be defined by $\theta(a) = (R/aR)$.
  Then $\theta(1) = \mathbf 0$ and $\theta(ab) = (R/abR) = (R/aR) + (aR/abR) = (R/aR) + (R/bR)$.
  Thus, $\theta$ is a homomorphism of monoids.
  Let $\pi\colon (\fagrp{\towers{R}},+) \to (\cgrp{R},+)$ be the canonical epimorphism, which satisfies $(R/I) \mapsto \langle R/I\rangle$ for all right $R$-ideals $I$ in the principal genus.

  We define a surjective homomorphism $\beta_0 \colon (\famon{\towers{R}},+) \to (\famon{\cgrpm{R}},\cdot)$ by means of $\beta_0((T)) = \langle T\rangle$ for all towers $T$.
  Thus, if $T_1$, $\ldots\,$,~$T_n$ are towers, then
  \[
  \beta_0\big((T_1)+\cdots+(T_n)\big)=\langle T_1 \rangle \cdots \langle T_n \rangle,
  \]
  where the product on the right hand side is the formal product in $\famon{\cgrpm{R}}$.
  Let $\sigma\colon(\famon{\cgrpm{R}},\cdot) \to (\cgrp{R},+)$ be the sum homomorphism, which maps $g_1\cdots g_l \mapsto g_1 + \cdots + g_l$.
  It follows that $\sigma \circ \beta_0 = \pi|_{\famon{\towers{R}}}$.
  Since
  \[
  \cB(\cgrpm{R}) = \{\, S \in \famon{\cgrpm{R}} \mid \sigma(S) = 0 \,\},
  \]
  the map $\beta_0$ restricts to a homomorphism $\beta\colon \cp{R} \to \cB(\cgrpm{R})$.
  If $S \in \cB(\cgrpm{R})$, then there exist towers $T_1$, $\ldots\,$,~$T_n$ such that $S=\langle T_1\rangle \cdots \langle T_n\rangle$ and $\langle T_1 \rangle + \cdots + \langle T_n \rangle = \mathbf 0$ in $\cgrp{R}$.
  By \subref{l-k0-cycgen:constr}, there exists a right $R$-ideal $I$ such that $(R/I) = (T_1) + \cdots + (T_n)$.
  Then $I$ is free by \cref{l-free}.
  Thus $I=aR$ with $a \in R^\bullet$, and $\beta(R/aR) = S$ by construction.
  Therefore $\beta$ is surjective.

  Finally, we set $\overline \theta = \beta \circ \theta\colon R^\bullet \to \cB(\cgrpm{R})$.
  By definition of $\cp{R}$, the map $\theta$ is surjective.
  Since $\beta$ is surjective, so is $\overline \theta$.
  If $a \in R^\bullet$ with $\overline\theta(a) = 1$, then $\theta(a) = (R/aR) = \mathbf 0$.
  Hence $aR=R$, and $a \in R^\times$.
  Thus, property~\labelcref{th:units} of a transfer homomorphism holds for $\theta$ and $\overline\theta$.
  \[
  \begin{tikzcd}
                          & (\famon{\towers{R}},+) \ar[r,"\beta_0"] \ar[rr,"\pi|_{\famon{\towers{R}}}",bend left=15]  & (\famon{\cgrpm{R}},\cdot) \ar[r,"\sigma"] & (\cgrp{R},+) \\
  (R^\bullet,\cdot) \ar[r,"\theta"] \ar[rr,"\overline\theta"',bend right=15] & (\cp{R},+) \ar[r,"\beta"] \ar[u,hook] & (\cB(\cgrpm{R}),\cdot) \ar[u,hook]
  \end{tikzcd}
  \]

  \ref*{thm-transfer:zss}
  To conclude that $\overline\theta$ is a transfer homomorphism, we still have to show:
  If $a \in R^\bullet$ and $\overline\theta(a) = BC$ with $B$,~$C \in \cB(\cgrpm{R})$, then there exist $b$,~$c \in R^\bullet$ such that $a=bc$, that $\overline\theta(b)=B$, and that $\overline\theta(c)=C$.
  We have $(R/aR) = (T_1) + \cdots + (T_n)$ with towers $T_1$, $\ldots\,$,~$T_n$.
  Now $\beta(R/aR) = \langle T_1\rangle\cdots \langle T_n\rangle$.
  Since $\famon{\cgrpm{R}}$ is a free abelian monoid, we may, after renumbering the towers, without restriction assume $B=\langle T_1\rangle\cdots \langle T_m\rangle$ and $C = \langle T_{m+1}\rangle\cdots \langle T_n\rangle$ for some $m \in [0,n]$.
  Using \subref{l-towerlift:weak-upper}, we find a right $R$-ideal $I$ such that $R \supset I \supset aR$, that $(R/I) = (T_1') + \cdots + (T_m')$ with towers $T_1'$, $\ldots\,$,~$T_m'$, and that $\langle T_i'\rangle = \langle T_i \rangle$ for all $i \in [1,m]$.
  Since $\mathbf 0 = \sigma(B) = \langle T_1\rangle + \cdots + \langle T_m\rangle$, the right $R$-ideal $I$ is free by \cref{l-free}.
  Hence $I=bR$ for some $b \in R^\bullet$ and $\overline\theta(b) = B$.
  By taking $c \in R^\bullet$ with $a=bc$ we necessarily have $\overline\theta(c)=C$, and the claim follows.

  \ref*{thm-transfer:cp}
  Suppose $a \in R^\bullet$ and $(R/aR) = (R/bR) + (R/cR)$ with $b$, $c \in R^\bullet$.
  We have to show that there exist $b'$,~$c' \in R^\bullet$ such that $a=b'c'$, that $(R/bR)=(R/b'R)$, and that $(R/cR)=(R/c'R)$.
  Let $(R/bR) = (T_1) + \cdots + (T_m)$ and $(R/cR) = (T_{m+1}) + \cdots + (T_{n})$ with $n \in \bN_0$, with $m \in [0,n]$, and with towers $T_1$, $\ldots\,$,~$T_n$.
  By \subref{l-towerlift:strong-upper}, there exists a right $R$-ideal $I$ with $aR \subset I \subset R$ such that $(R/I) = (T_1) + \dots + (T_m) = (R/bR)$.
  Then $(I/aR) = (T_{m+1}) + \dots + (T_{n}) = (R/cR)$.
  By \cref{l-free}, the right $R$-ideal $I$ is free.
  Thus $I = b'R$ with $b' \in R^\bullet$.
  Let $c' \in R^\bullet$ be such that $a=b'c'$.
  Then $(R/b'R)=(R/bR)$ and $(R/c'R)=(R/cR)$.

  \ref*{thm-transfer:cp-weak}
  If $a \in R^\bullet$ and $a=bc$ with $b$,~$c \in R^\bullet \setminus R^\times$, then $\theta(b)$,~$\theta(c) \ne \mathbf 0$, and hence $\theta(a)=\theta(b)+\theta(c)$ with $\theta(b)$, $\theta(c)$ non-units of $(\cp{R},+)$.
  Thus, if $a \in R^\bullet$ and $\theta(a)$ is an atom of $\cp{R}$, then $a$ is an atom of $R^\bullet$.
  It remains to show:
  If $\theta(a) = (R/a_1R) + \cdots + (R/a_kR)$ with $(R/a_1R)$, $\ldots\,$,~$(R/a_kR)$ atoms of $\cp{R}$ and $k \ge 1$, then there exist $u_1$, $\ldots\,$,~$u_k \in R^\bullet$ and a permutation $\sigma \in \fS_k$ such that $a=u_1\cdots u_k$ and $\theta(u_i) = (R/a_{\sigma(i)}R)$ for all $i \in [1,k]$.

  Suppose that $i \in [1,k]$ is such that $(R/a_iR)$ contains a faithful simple module $W$, that is, $(R/a_iR) = (W) + (X)$ with $(X) \in \famon{\towers{R}}$.
  By \subref{int:upper}, there exists a maximal right $R$-ideal $I$ such that $a_iR \subset I \subset R$ and $\langle R/I\rangle=\langle W\rangle$.
  Since $\langle W \rangle = \mathbf 0$ by our assumption, \cref{l-free} implies $I=bR$ with $b \in R^\bullet$.
  But $a_i$ is an atom, hence $a_iR=bR$.
  Thus, if $(R/a_iR)$ contains a faithful simple module $W$, then $(R/a_iR)=(W)$.

  Let $\Omega \subset [1,k]$ be the subset of all indices $i$ for which $R/a_iR$ is a faithful simple module.
  Let $l = k - \card{\Omega}$, and let $\tau\colon [1,l] \to [1,k]\setminus \Omega$ be a bijective map.
  For all $i \in [1,l]$ the element $(R/a_{\tau(i)}R)$ is a sum of cycle towers.
  Hence, \subref{l-towerlift:strong-upper} implies that there exists a chain of right $R$-ideals
  \[
  aR \subset  I_l \subset I_{l-1} \subset \cdots \subset I_{1} \subset I_0=R,
  \]
  such that $(I_{i-1}/I_i)=(R/a_{\tau(i)}R)$ for all $i \in [1,l]$.
  Now $(I_l/aR)=\sum_{i \in \Omega} (R/a_iR)$, and $R/a_iR$ is a faithful simple module for all $i \in \Omega$.
  By \subref{l-towerlift:full}, there exists a chain of right $R$-ideals
  \[
  aR=I_{k} \subset I_{k-1} \subset \cdots \subset I_l
  \]
  and a bijection $\pi\colon [1,k-l] \to \Omega$ such that $(I_{l+i-1}/I_{l+i})=(R/a_{\pi(i)}R)$ for all $i \in [1,k-l]$.
  Combining $\tau$ and $\pi$, we obtain a permutation $\sigma \in \fS_k$ such that $(I_{i-1}/I_{i})=(R/a_{\sigma(i)} R)$ for all $i \in [1,k]$.
  Then it follows, inductively through repeated application of \cref{l-free}, that $I_{i}$ is free and $I_i = u_1\cdots u_iR$ for some $u_1$, $\ldots\,$,~$u_k \in R^\bullet$.
  Since $I_k=aR$, we may choose $u_k$ in such a way that $a=u_1\cdots u_k$.
  With this choice, $\theta(u_i)=(R/u_iR)=(I_{i-1}/I_{i})=(R/a_{\sigma(i)}R)$, as required.
\end{proof}

\begin{remark} \label{r-main}
  \begin{enumerate}
    \item
      If $R$ is a bounded HNP ring and every stably free right $R$-ideal is free, then all assumptions of the previous theorem are satisfied (including the extra assumptions in \ref*{thm-transfer:cp} and \ref*{thm-transfer:cp-weak}).

    \item
      The isomorphism $\phi\colon \cgrp{R} \to \pcg{R}$ from \cref{t-k0-iso-pcg} maps $\cgrpm{R}$ to $\pcgm{R}$.
      Therefore, $\phi$ induces a monoid isomorphism
      \[
      \overline\phi\colon
      \begin{cases}
        \cB(\cgrpm{R}) &\to \cB(\pcgm{R}) \\
        \langle T_1 \rangle \cdots \langle T_n \rangle &\mapsto \phi\big(\langle T_1\rangle\big) \cdots \phi\big(\langle T_n\rangle\big).
      \end{cases}
      \]
      Then also $\overline\phi\circ\overline\theta$ is a transfer homomorphism.
      Working with $\overline\theta$ is of course trivially equivalent to working with $\overline\phi\circ\overline\theta$.

      In this section, we will state all results for $\overline\theta$.
      For a concrete ring, it may be more useful to think in terms of $\pcg{R}$ rather than $\cgrp{R}$, because the former group appears as a factor of $K_0(R)$ and has been studied before.

    \item
      By \cref{p-krull}, the monoid $\cp{R}$ is a commutative Krull monoid, and $\cp{R} \to \famon{\cgrpm{R}}$ is a divisor homomorphism.
      The map $\beta\colon \cp{R} \to \cB(\cgrpm{R})$ in the previous proof is therefore the usual transfer homomorphism of a commutative Krull monoid as given in \cite[Proposition 3.4.8]{ghk06}.
      (However, we do not know if $\cC$ is always the divisor class group of $\cp{R}$, since we do not know whether $\cp{R} \to\famon{\cgrpm{R}}$ is a divisor theory.)
      If $\theta$ is a transfer homomorphism, then the results about $\overline\theta$ in \cref{thm-transfer,thm-catenary} follow from the decomposition $\overline\theta=\beta\circ\theta$ and the corresponding results about $\theta$ and $\beta$.

    \item
      In the absence of \labelcref{c-fos}, the map $\theta$ need not be a transfer homomorphism:
      Let $R$ be a prime principal ideal ring having two non-isomorphic faithful simple modules $V$ and $W$ with $\Ext_R^{1}(V,W) \ne \mathbf 0$.
      Then there exists a uniserial module $U$ of length $2$ with top composition factor $V$ and bottom composition factor $W$.
      Let $I$ be a right $R$-ideal with $R/I \cong U$.
      Then $I=vwR$ with $v$,~$w \in \cA(R^\bullet)$ such that $(R/vR)=(V)$ and $(R/wR)=(W)$.
      Since $R/I$ is uniserial, $\rf{v,w}$ is the unique rigid factorization of $vw$.
      In particular, there exists no representation $vw=w'v'$ with $(R/vR)=(R/v'R)$ and $(R/wR)=(R/w'R)$.
      Hence, $\theta$ is no transfer homomorphism.
      However, $\theta$ is a weak transfer homomorphism by \subref{thm-transfer:cp-weak}.

      An explicit instance of this example is worked out in \cref{e-no-f2s-no-transfer}.

    \item If \labelcref{c-ft} does not hold, then $\overline\theta$ need not be a (weak) transfer homomorphism.
      See \cref{p-nonhf,e-no-ft}.

    \item \label{r-main:est} Let $\cO$ be the ring of algebraic integers in a number field $K$, let $A$ be a central simple $K$-algebra, and let $R$ be a classical hereditary $\cO$-order in $A$.
      Let $\cO_A^\bullet$ denote the submonoid of $\cO^\bullet$ consisting of all elements that are positive at each archimedean place of $K$ that ramifies in $A$.
      Assume further that $A$ satisfies the Eichler condition with respect to $\cO$.
      Estes, in \cite{estes91b}, showed that the reduced norm $\nr\colon R^\bullet \to \cO_A^\bullet$ is a transfer homomorphism.
      The monoid $\cO_A^\bullet$ is a regular congruence submonoid of $\cO_A$ (see \cite[Chapter 2.11]{ghk06}) and the divisor class group of $\cO_A^\bullet$ can be identified with a ray class group $\cC_A(\cO)$ of $\cO$.
      Moreover, $\cC_A(\cO) \cong \pcg{R} \cong \cgrp{R}$ and $\cgrpm{R}=\cgrp{R}$.
      This gives another way of constructing the transfer homomorphism $R^\bullet \to \cB(\cgrp{R})$ in the classical setting.
  \end{enumerate}
\end{remark}

\subsection{Catenary degrees} \label{ssec:catenary}

Next, we show that the catenary degree in the fibers is at most two for the transfer homomorphisms in \cref{thm-transfer}.
The proof of this is similar in principle to the one in \cite[Theorem 3.2.8]{ghk06}, which covers commutative Krull monoids and more general classes of commutative monoids, and the one in \cite[Proposition 7.7]{baeth-smertnig15} for arithmetical maximal orders.
However, the fact that we need to deal with non-trivial towers of simple modules increases the complexity of the proof.
Thus, we first prove a technical lemma (\cref{l-weave}), which deals with the, somewhat generic, induction that appears in all the proofs mentioned above.
In the proof of \cref{thm-catenary} itself, we are then able to concentrate on the aspects specific to our situation.

\begin{lemma} \label{l-lift-permuted}
  Let $H$ be a  monoid, let $T$ be a reduced commutative monoid, and let $\theta\colon H \to T$ be a transfer homomorphism.
  Let $a \in H$.
  If $z=\rf[\varepsilon]{u_1,\cdots,u_k}$ is a rigid factorization of $a$ and $\sigma \in \fS_k$ is a permutation, then there exists a rigid factorization $z'=\rf[\varepsilon]{u_1',\cdots,u_k'}$ of $a$ with $\theta(u_{\sigma(i)}')=\theta(u_i)$ for all $i \in [1,k]$.
  Moreover, there exists a $2$-chain in the fiber between $z$ and $z'$ \textup{(}for any distance $\sd$ on $H$\textup{)}.
\end{lemma}

\begin{proof}
  We may assume $k \ge 2$, as the claim is trivially true otherwise.
  The symmetric group $\fS_k$ is generated by transpositions of the form $(i,\,i+1)$ with $i \in [1,k-1]$.
  Therefore, it suffices to show the claim for such transpositions.
  By commutativity of $T$, we have
  \[
  \theta(u_i u_{i+1})=\theta(u_i)\theta(u_{i+1})=\theta(u_{i+1})\theta(u_i).
  \]
  Since $\theta$ is a transfer homomorphism, there exist $u_i'$,~$u_{i+1}' \in \cA(H)$ such that $u_iu_{i+1}=u_i'u_{i+1}'$ with $\theta(u_i')=\theta(u_{i+1})$ and $\theta(u_{i+1}')=\theta(u_{i})$.
  Setting
  \[
  z'=\rf[\varepsilon]{u_1,\cdots,u_{i-1},u_{i}',u_{i+1}',u_{i+2},\cdots,u_k},
  \]
  it follows from the defining properties of a distance that $\sd(z,z') \le 2$.
\end{proof}

\begin{lemma} \label{l-weave}
  Let $H$ be a monoid, let $T$ be a reduced commutative monoid, and let $\theta\colon H \to T$ be a transfer homomorphism.
  Assume that there exists a function $\delta\colon \cA(H) \times \cA(H) \to \bN_0$ having the following property:
  \begin{enumerate}
    \item[] If $a=u_1\cdots u_k$ with $k > 2$ and $u_1$, $\ldots\,$,~$u_k \in \cA(H)$, and $v \in \cA(H)$ is such that $u_1H \ne vH$, that $\theta(u_1)=\theta(v)$, and that $a \in vH$, then there exist $u_1'$, $w$,~$w' \in H$ such that $u_1w=u_1'w'$, that $a \in u_1wH$, that $\delta(u_1',v) < \delta(u_1,v)$, that $\theta(u_1')=\theta(u_1)$, and that $\theta(w)=\theta(w')=\theta(u_i)$ for some $i \in [2,k]$.
  \end{enumerate}
  Then $\sc_\sd(H,\theta) \le 2$ for any distance $\sd$ on $H$.
\end{lemma}

\begin{proof}
  Let $\sd$ be a distance on $H$.
  Let $a \in H \setminus H^\times$, and let $z=\rf{u_1,\cdots,u_k}$ and $z'=\rf{v_1,\cdots,v_k}$ be two factorizations of $a$ with $\theta(u_i)=\theta(v_{\sigma(i)})$ for a permutation $\sigma \in \fS_k$ and all $i \in [1,k]$.
  We have to show that there exists a $2$-chain in the fiber between $z$ and $z'$.
  By \cref{l-lift-permuted}, we may without restriction assume $\sigma=\id$, after replacing $z$ if necessary.
  Thus, we show the following claim for all $k \in \bN$ and $d \in \bN_0$.
  \begin{claim}[$A(k,d)$]
    If $z=\rf{u_1,\cdots,u_k}$ and $z'=\rf{v_1,\cdots,v_k}$ with $\delta(u_1,v_1)=d$, with $u_1\cdots u_k=v_1\cdots v_k$, and with $\theta(u_j)=\theta(v_j)$ for all $j \in [1,k]$, then there exists a $2$-chain in the fiber between $z$ and $z'$.
  \end{claim}

  We proceed by induction on $(k,d)$ in lexicographic order:
  Assume that $A(l,e)$ holds whenever either $l < k$, or $l=k$ and $e < d$.
  Since the claim is trivially true if $k \le 2$, also assume $k >2$.

  If $u_1H = v_1H$, there exists $\varepsilon \in H^\times$ such that $v_1\varepsilon = u_1$.
  Then
  \[
  z'=\rf{v_1\varepsilon,\varepsilon^{-1}v_2,v_3,\cdots,v_k}=\rf{u_1,\varepsilon^{-1}v_2,v_3,\cdots,v_k}.
  \]
  By induction hypothesis, there exists a $2$-chain in the fiber between $\rf{u_2,\cdots,u_k}$ and $\rf{\varepsilon^{-1}v_2,v_3,\cdots,v_k}$.
  Multiplying each factorization in this $2$-chain by $u_1$ from the left yields a $2$-chain in the fiber between $z$ and $z'$.

  Suppose now that $u_1H \ne v_1H$.
  By assumption, there exist $u_1'$, $w$,~$w' \in H$ such that $u_1w=u_1'w'$, that $a \in u_1wH$, that $\delta(u_1',v_1) < \delta(u_1,v_1)$, that $\theta(u_1')=\theta(u_1)$, and that $\theta(w)=\theta(w')=\theta(u_i)$ for some $i \in [2,k]$.
  Applying \cref{l-lift-permuted} twice, to $\rf{u_2,\cdots,u_k}$ and to $\rf{v_2,\cdots,v_k}$, we may assume $i=2$.
  Since $\theta$ is a transfer homomorphism, the elements $u_1'$, $w$, and $w'$ are atoms.
  Let $c \in H$ be such that $a=u_1wc$.
  Since $\theta(c)=\theta(u_3)\cdots\theta(u_k)$ and $\theta$ is a transfer homomorphism, there exists a factorization $y=\rf{w_3,\cdots,w_k}$ of $c$ with $\theta(w_j)=\theta(u_j)$ for all $j \in [3,k]$.
  By the induction hypothesis, there exists a $2$-chain in the fiber between $\rf{u_2,\cdots,u_k}$ and $w\rfop y$.
  Multiplying each factorization in this $2$-chain by $u_1$ from the left yields a $2$-chain in the fiber between $z$ and $\rf{u_1,w}\rfop y$.
  By the defining properties of a distance, $\sd(\rf{u_1,w}\rfop y, \rf{u_1',w'}\rfop y) \le 2$.
  Finally, since $\delta(u_1',v_1) < \delta(u_1,v_1)$, the induction hypothesis implies that there exists a $2$-chain in the fiber between $\rf{u_1',w'}\rfop y$ and $z'$.
  Concatenating all these $2$-chains, it follows that there exists a $2$-chain in the fiber between $z$ and $z'$.
\end{proof}

\begin{remark}
  The previous lemma also holds more generally in the case where $H$ is a cancellative small category.
  Furthermore, one may replace the assumption by:
  \begin{enumerate}
  \item[]
    There exists $N \in \bN_{\ge 2}$, such that, if $a=u_1\cdots u_k$ with $k > N$ and $u_1$, $\ldots\,$,~$u_k \in \cA(H)$, and $v \in \cA(H)$ is such that $u_1H \ne vH$, that $\theta(u_1)=\theta(v)$, and that $a \in vH$, then there exist $u_1'$, $w_1$, $\ldots\,$,~$w_{N-1}$, $w_1'$, $\ldots\,$,~$w_{N-1}' \in H$ such that $u_1w_1\cdots w_{N-1}=u_1'w_1'\cdots w_{N-1}'$, that $a \in u_1w_1\cdots w_{N-1}H$, that $\delta(u_1',v) < \delta(u_1,v)$, that $\theta(u_1')=\theta(u_1)$, and that $\theta(w_j) = \theta(w_j') = \theta(u_{\sigma(j)})$ for some injective map $\sigma\colon [1,N-1] \to [2,k]$.
  \end{enumerate}
  Then the conclusion becomes $\sc_\sd(H,\theta) \le N$.
\end{remark}

\begin{lemma} \label{l-tower-complete}
  Let $J \subset K_0 \subset I$ be right $R$-ideals such that $(I/J) \in \famon{\towers{R}}$ and $I/K_0$ is simple.
  Let $T$ be the tower containing $(I/K_0)$.
  Then there exists a right $R$-ideal $K$ such that $J \subset K \subset K_0$ and $(I/K) = (T)$.
\end{lemma}

\begin{proof}
  If $T$ is a faithful tower, then $T$ is trivial. Hence, the claim follows with $K=K_0$.
  Suppose now that $T$ is a cycle tower.
  By \cite[Corollary 41.4]{levy-robson11}, there exists an isomorphism $\alpha\colon K_0/J \isomto F \oplus C$ with $F$ a module all of whose composition factors are faithful and $C$ a module all of whose composition factors are unfaithful.
  Replacing $J$ by $J+\alpha^{-1}(F)$, we may without loss of generality assume $K_0/J \cong C$.
  Now $(K_0/J) = (T_1) + \cdots + (T_n) + (T) - (I/K_0)$ with cycle towers $T_1$, $\ldots\,$,~$T_n$.
  By \subref{l-towerlift:strong}, there exists a submodule $J \subset K \subset K_0$ such that $(K/J) = (T_1) + \cdots + (T_n)$.
  Then $(K_0/K) = (T) - (I/K_0)$ and hence $(I/K)=(T)$.
\end{proof}

In the following proof, we denote by $\val_V\colon \fagrp{\simple{R}} \to \bZ$, respectively by $\val_T\colon \fagrp{\towers{R}} \to \bZ$, the valuations corresponding to $(V)$, respectively $(T)$, where $V$ is a simple module and $T$ is a tower.
If $(V)$ is contained in $(T)$ and $(M) \in \fagrp{\towers{R}}$, then $\val_T(M) = \val_V(M)$.

\begin{thm} \label{thm-catenary}
  In the setting of \cref{thm-transfer}, we have $\sc_\sd(R^\bullet,\overline \theta) \le 2$ for any distance $\sd$ on $R^\bullet$.
  If, in addition, \labelcref{c-fos} holds, then also $\sc_\sd(R^\bullet,\theta) \le 2$.
\end{thm}

\begin{proof}
  We use \cref{l-weave} to show $\sc_\sd(R^\bullet,\overline\theta) \le 2$.
  For $u$,~$v \in \cA(R^\bullet)$, we define $\delta(u,v)$ to be the length of the module $uR+vR/vR$.
  We have to verify that the property required in \cref{l-weave} is satisfied for this choice of $\delta$.
  Let $a \in R^\bullet$ and $a=u_1\cdots u_k$ with $k > 2$ and $u_1$, $\ldots\,$,~$u_k \in \cA(R^\bullet)$.
  Let $v \in \cA(R^\bullet)$ be such that $a \in vR$, that $\overline\theta(v)=\overline\theta(u_1)$, and that $u_1R \ne vR$.
  We first show $vR \subsetneq u_1R+vR$.
  Assume to the contrary that $vR = u_1R + vR$.
  Then $u_1R \subset u_1R + vR = vR$.
  Since $u_1$ and $v$ are both atoms, this implies $u_1R=vR$, a contradiction.

  \begin{figure}
    \begin{tikzcd}[every arrow/.append style={dash}, sep={1cm,between origins}]
      &                     &                                                & R \ar[d]                                                              &                                       &  &    \\
      &                     &                                                & \phantom{u_1R+I=}u_1R + vR = u_1R+I \ar[lllddd] \ar[rd]               &                                       &  &    \\
      &                     &                                                &                                                                       & I \ar[rrdd] \ar[dl,end anchor=center] &  &    \\
      &                     &                                                & \phantom{.} \ar[lldd,start anchor=center] \ar[dr,start anchor=center] &                                       &  &    \\
      u_1R \ar[rd] &                     &                                                &                                                                       & \phantom{u_1'R}L=u_1'R \ar[ddll]      &  & vR \\
      & u_1R \cap I \ar[dr] &                                                &                                                                       &                                       &  &    \\
      &                     & J \ar[d]                                       &                                                                       &                                       &  &    \\
      &                     & \phantom{u_1wR=u_1'w'R=}K=u_1wR=u_1'w'R \ar[d] &                                                                       &                                       &  &    \\
      &                     & aR                                             &                                                                       &                                       &  &    \\
    \end{tikzcd}
    \caption{Construction in the proof of \cref{thm-catenary}.
    The inclusion $aR \subset vR$ is not shown.}
  \end{figure}
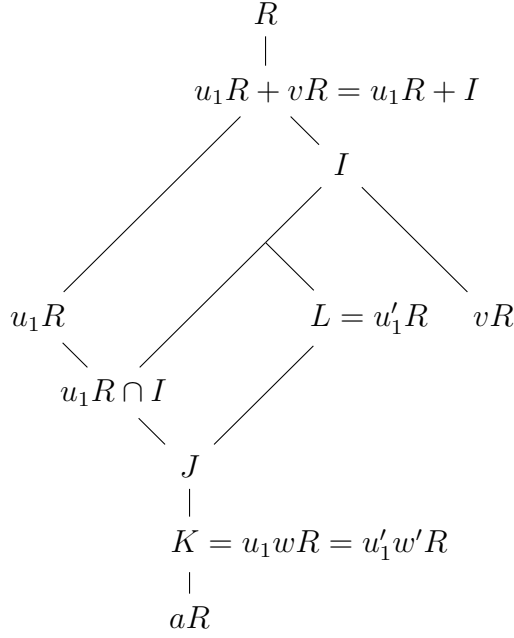

  Let $I$ be a right $R$-ideal which is maximal with respect to $vR \subset I \subsetneq u_1R + vR$.
  Since $u_1R+I = u_1R + vR$, it follows that $u_1R/(u_1R \cap I)\cong (u_1R + vR)/I$.
  In particular, $u_1R/(u_1R \cap I)$ is simple.
  Let $T$ denote the tower containing $(u_1R/(u_1R \cap I))$.
  By \cref{l-tower-complete}, there exists a right $R$-ideal $J$ with $aR \subset J \subset u_1R \cap I$ such that $(u_1R/J) = (T)$.
  Hence $(T)$ is contained in $(u_1R/aR)=(R/u_2\cdots u_kR)$.
  Since $(R/u_jR) \in \famon{\towers{R}}$ for each $j \in [2,k]$, there exists an $i \in [2,k]$ such that $(T)$ is contained in $(R/u_iR)$.
  From now on, we fix such an index $i$.
  Let $(R/u_iR) = (T) + (T_2) + \cdots + (T_m)$ with towers $T_2$, $\ldots\,$,~$T_m$.
  By \subref{l-towerlift:weak-upper}, there exists $aR \subset K \subset J$ such that $(J/K) = (T_2') + \cdots + (T_m')$ with towers $T_2'$, $\ldots\,$,~$T_m'$ satisfying $\langle T_j'\rangle = \langle T_j \rangle$ for all $j \in [2,m]$.
  Then $(u_1R/K) = (T) + (T_2') + \cdots + (T_m')$ with $\langle T \rangle + \langle T_2'\rangle + \cdots + \langle T_m' \rangle = \langle R/u_iR \rangle = \mathbf 0$.
  By \cref{l-free}, the right $R$-ideal $K$ is principal, say $K=u_1wR$ with $w \in R^\bullet$.
  Then $\overline\theta(w) = \overline\theta(u_i)$.

  Now we construct a second factorization of $u_1w$.
  First we prove the following intermediate claim.
  \begin{claim}
    There exists a right $R$-ideal $L$ with $J \subset L \subset I$ such that $(L/J)=(T')$ with a tower $T'$ and $\langle T'\rangle = \langle T \rangle$.
  \end{claim}

  \begin{proof}[Proof of Claim]
  \emph{Case 1:} Suppose $(u_1R+vR/vR)$ does not contain all of $(T)$.
  Then $T$ is necessarily non-trivial, and hence a cycle tower due to \labelcref{c-ft}.
  Let $W=u_1R+vR/I$.
  Let $V$ be a simple module of $T$ whose class does not appear in $(u_1R+vR/vR)$.
  Then, since $(R/u_1R)$, $(R/vR) \in \famon{\towers{R}}$,
  \[
  \begin{split}
    \val_{T}(R/u_1R) & = \val_{V}(R/u_1R)                                          \\
                     & \ge \val_{V}(R/(u_1R+vR)) \\
                     & = \val_{V}(R/vR) = \val_{T}(R/vR) \\
                     & = \val_{W}(R/(u_1R+vR)) + \val_{W}((u_1R+vR)/vR)            \\
                     & \ge \val_{W}(R/(u_1R+vR)) + 1.
  \end{split}
  \]
  Therefore, $\val_{W}(R/u_1R)=\val_{T}(R/u_1R) \ge \val_{W}(R/(u_1R+vR)) +1$.
  It follows that $\val_{W}((u_1R+vR)/u_1R) \ge 1$.
  In other words, $W$ appears as a composition factor of $(u_1R+vR)/u_1R$.

  Since $(u_1R+vR)/u_1R\cong I/(u_1R \cap I)$, the module $W$ also appears as a composition factor of $I/(u_1R \cap I)$.
  By construction of $J$, we have $(u_1R \cap I/ J) = (T) - (W)$, and thus $(I/J)$ contains the full tower $(T)$.
  By \subref{l-towerlift:strong}, there exists a right $R$-ideal $L$ such that $J \subset L \subset I$ and $(L/J)=(T)$.
  The claim holds with $T'=T$.

  \emph{Case 2:} Suppose $(u_1R+vR/vR)$ does contain all of $(T)$.
  Since $\overline\theta(u_1)=\overline\theta(v)$, there exists a tower $T''$ such that $(T'')$ is contained in $(u_1R+vR/u_1R)$ and $\langle T''\rangle = \langle T \rangle$.
  Hence $(T'')$ is also contained in $(I/u_1R \cap I)$ and hence in $(I/J)$.
  By \subref{l-towerlift:weak}, there exists $J \subset L \subset I$ such that $(L/J)=(T')$ for a tower $T'$ with $\langle T'\rangle = \langle T''\rangle =\langle T \rangle$.
  \end{proof}

  Now $(L/K) = (L/J)+(J/K)=(T') + (T_2') + \cdots + (T_m')$.
  Since we already know that $K$ is principal, \cref{l-free} implies that $L$ is principal, say $L=u_1'R$ with $u_1' \in R^\bullet$.
  Choosing $w' \in R^\bullet$ such that $u_1w=u_1'w'$, we have
  \[
  \overline\theta(w')=\langle T' \rangle \langle T_2'\rangle \cdots \langle T_m'\rangle = \langle T \rangle \langle T_2'\rangle \cdots \langle T_m'\rangle = \overline\theta(w)=\overline\theta(u_i).
  \]
  Then also $\overline\theta(u_1')=\overline\theta(u_1)$.
  Due to $u_1'R + vR \subset I \subsetneq u_1R + vR$, we have $\delta(u_1',v) < \delta(u_1,v)$.
  Thus, the conditions of \cref{l-weave} are satisfied and $\sc_\sd(R^\bullet,\overline\theta) \le 2$.

  If \labelcref{c-fos} holds, one shows $\sc_\sd(R^\bullet,\theta) \le 2$ analogously, using \labelcref{l-towerlift:strong} and \labelcref{l-towerlift:strong-upper} of \cref{l-towerlift} instead of \labelcref{l-towerlift:weak} and \labelcref{l-towerlift:weak-upper} to achieve $T'=T$ and $T_i'=T_i$ for $i \in [2,m]$.
\end{proof}

Under reasonable conditions, we are able to characterize when $R$ is $\sd$-factorial for $\sd$ any of the distances $\sd_{\textup{cs}}$, $\sd_{\textup{sim}}$, and $\sd^*$.
First we need a lemma.

\begin{lemma} \label{l-split}
  Let $T$ be a cycle tower and let $M$ be a module of finite length all of whose composition factors belong to $T$.
  Let $N \subset M$ be a submodule.
  If either $N$ is uniserial of length the socle-height of $M$, or $M/N$ is uniserial of length the socle-height of $M$, then the canonical sequence $\mathbf 0 \to N \to M \to M/N \to \mathbf 0$ splits.
\end{lemma}

\begin{proof}
  Let $m$ be the socle-height of $M$.
  Let $W_1$, $\ldots\,$,~$W_n$ be a set of unfaithful simple modules representing $T$ and let $J=\smash\bigcap_{i=1}^n \ann_R(W_i)$.
  Then $J^m$ annihilates $M$, and hence $M$, $N$, and $M/N$ are modules over the Artinian serial ring $\Lambda=R/J^m$.
  Moreover, $\overline J = J/J^m$ is the Jacobson radical of $\Lambda$ as a consequence of \cite[Lemma 22.8]{levy-robson11}.
  Note that $m$ is the index of nilpotence of $\overline J$, as we have $J^{m-1} \supsetneq J^m$ due to \cite[Proposition 22.9 and Lemma 22.14]{levy-robson11}.

  Suppose first that $N$ is uniserial of length $m$.
  By \cite[Lemma 50.11(iv)]{levy-robson11}, there exists a primitive idempotent $e$ of $\Lambda$ such that $N \cong e\Lambda/e\overline{J}^m=e\Lambda$.
  Since the composition length of $e\Lambda$ equals the index of nilpotence of $\overline J$, the right ideal $e\Lambda$ is an injective $\Lambda$-module by \cite[Corollary 50.18]{levy-robson11}.
  Thus, the stated short exact sequence splits over $\Lambda$ and hence also over $R$.

  Suppose now that $M/N$ is uniserial and the length of $M/N$ is equal to $m$.
  Analogous to the previous case, $M/N \cong e\Lambda$ for a primitive idempotent $e$ of $\Lambda$.
  Hence, $M/N$ is a projective $\Lambda$-module and the sequence splits.
\end{proof}

\begin{prop} \label{p-fact}
  Let $R$ be as in \cref{thm-transfer}.
  Suppose further that $\cgrpm{R}=\cgrp{R}$, and that, if $\cgrp{R}\cong\mathsf C_2$, there exist at least two distinct towers $T_1$ and $T_2$ with $\langle T_1\rangle = \langle T_2 \rangle \ne \mathbf 0$.
  Then
  \begin{enumerate}
    \item\label{p-fact:cs} $R^\bullet$ is composition series factorial if and only if $\cgrp{R}=\mathbf 0$.
      Otherwise, $\sc_{\textup{cs}}(R^\bullet) \ge 2$.
    \item\label{p-fact:sim} $R^\bullet$ is similarity factorial if and only if $R$ is a principal ideal ring.
      Otherwise, $\sc_{\textup{sim}}(R^\bullet) \ge 2$.
    \item\label{p-fact:rigid} $R^\bullet$ is rigidly factorial if and only if $R$ is a local principal ideal ring.
      Otherwise, $\sc^*(R^\bullet) \ge 2$.
  \end{enumerate}
\end{prop}

\begin{proof}
  If $\card{\cgrp{R}} > 2$, then $\cB(\cgrp{R})$ is not half-factorial.
  Hence, neither is $R^\bullet$, due to the existence of the  transfer homomorphism $\overline\theta\colon R^\bullet \to \cB(\cgrp{R})$ established in \cref{thm-transfer}.
  In the following, it therefore suffices to consider the two cases where $\cgrp{R}$ is trivial and where $\cgrp{R} \cong \sC_2$.

  \ref*{p-fact:cs}
  If $\cgrp{R} \cong \sC_2$, a construction similar to the one in \cite[Lemma 7.4]{baeth-smertnig15} shows that $R^\bullet$ is not composition series factorial:
  Let $I_1$, $J_1 \subset R$ be right $R$-ideals with $(R/I_1)=(T_1)$ and $(R/J_1)=(T_2)$.
  Let $I_2 \subset I_1$ with $(I_1/I_2)=(T_1)$ and let $J_2 \subset J_1$ with $(J_1/J_2)=(T_2)$.
  Due to $T_1 \ne T_2$, we have $I_2 + J_2 = R$.
  Then $I_2$, $J_2$, $I_1 \cap J_1$, and $I_2 \cap J_2$ are principal.
  Two maximal chains of principal right $R$-ideals from $I_2 \cap J_2$ to $R$ are given by
  \[
  I_2 \cap J_2 \subsetneq I_2 \subsetneq R \quad\text{and}\quad I_2 \cap J_2 \subsetneq I_1 \cap J_1 \subsetneq R.
  \]
  Since $(R/I_2) = (T_1) + (T_1)$, $(I_2/(I_2 \cap J_2))=(T_2) + (T_2)$, and $(R/(I_1 \cap J_1))=((I_1\cap J_1)/(I_2\cap J_2))=(T_1) + (T_2)$, these two chains correspond to factorizations $z$ and $z'$ with $\sd_{\textup{cs}}(z,z') = 2$.
  Hence $R^\bullet$ is not composition series factorial.

  If $\cgrp{R} = \mathbf 0$, then every tower-maximal right $R$-ideal is principal and the atoms of $R^\bullet$ are precisely the elements $a \in R^\bullet$ for which $aR$ is tower-maximal.
  Hence $R^\bullet$ is composition series factorial.

  If $z$,~$z'$ are two factorizations of an element $a \in R^\bullet$ with $\sd_{\textup{cs}}(z,z') > 0$, then $\sd_{\textup{cs}}(z,z') \ge 2$ due to the uniqueness of composition series.
  Hence either $\sc_{\textup{cs}}(R^\bullet)=0$, which by definition is the case if and only if $R^\bullet$ is composition series factorial, or $\sc_{\textup{cs}}(R^\bullet) \ge 2$.

  \ref*{p-fact:sim}
  If $\cgrp{R} \cong \sC_2$, then $R^\bullet$ is not composition series factorial, and hence also not similarity factorial.
  In that case, $\sc_{\textup{sim}}(R^\bullet) \ge \sc_{\textup{cs}}(R^\bullet) \ge 2$.
  Let $\cgrp{R}=\mathbf 0$.
  If $R$ is a Dedekind prime ring, then $R$ is a principal ideal ring, and hence similarity factorial.
  Suppose conversely that $R$ is similarity factorial.
  We have to show that $R$ is a Dedekind prime ring, that is, all towers are trivial.

  Suppose that $T$ is a non-trivial tower.
  Due to property~\labelcref{c-ft}, the tower $T$ is a cycle tower.
  Let $W_1$, $\ldots\,$,~$W_n$ with $n \ge 2$ be pairwise non-isomorphic simple modules that represent $T$ such that $W_i$ is an unfaithful successor of $W_{i-1}$ for all $i \in [2,n]$ and $W_1$ is an unfaithful successor of $W_n$.
  By \cref{l-ex-uniserial}, there exist right $R$-ideals $I_1$ and $I_2$ such that $R/I_1$ and $R/I_2$ are uniserial with composition factors, from top to bottom, $W_1$, $\ldots\,$,~$W_n$, respectively, $W_2$, $\ldots\,$,~$W_n$,~$W_1$.
  Necessarily $I_1 + I_2 = R$ and hence $R/(I_1 \cap I_2) \cong R/I_1 \oplus R/I_2$.
  Since $(R/I_1 \cap I_2)=(T)+(T)$ and $\langle T \rangle = \mathbf 0$, the right $R$-ideal $I_1 \cap I_2$ is principal, say $I_1 \cap I_2 = aR$ for some $a \in R^\bullet$.
  Let $J_1$ be the unique maximal right $R$-ideal between $I_1$ and $R$, and let $J_2$ be the unique right $R$-ideal which is minimal with respect to properly containing $I_2$.
  Then $R/J_1 \cong W_1$ and $R/J_2$ is uniserial with composition factors $W_2$, $\ldots\,$,~$W_n$.
  Thus $R/(J_1 \cap J_2) \cong W_1 \oplus (R/J_2)$ is not uniserial, and hence isomorphic to neither $R/I_1$ nor $R/I_2$.
  Similarly, $(J_1 \cap J_2)/(I_1 \cap I_2) \cong (J_1/I_1) \oplus W_1$ is isomorphic to neither $R/I_1$ nor $R/I_2$.
  However, $J_1 \cap J_2$ is principal since $(R/(J_1 \cap J_2)) = (T)$.
  It follows that the chains
  \[
  aR=I_1 \cap I_2 \subset I_1 \subset R \quad\text{and}\quad aR=I_1 \cap I_2 \subset J_1 \cap J_2 \subset R
  \]
  correspond to factorizations of $a$ which have similarity distance at least $2$.

  To conclude $\sc_{\textup{sim}}(R^\bullet) \ge 2$, we show that $a$ does not have any factorization whose atoms have similarity classes $R/I_1$ and $W_1 \oplus (R/J_2)$, or $R/I_2$ and $W_1 \oplus (R/J_2)$.
  Suppose that $aR \subsetneq K \subsetneq R$ with $(R/K)=(T)$.
  If one of $R/K$ and $K/aR$ is isomorphic to $R/I_1$ or $R/I_2$, then it is a direct summand of $R/aR$ by \cref{l-split}.
  Uniqueness of direct sum decomposition of modules of finite length implies that the other summand must be isomorphic to $R/I_2$ or  $R/I_1$.
  This proves the claim.

  \ref*{p-fact:rigid}
  If $R$ is a local PID and $J(R)$ is its Jacobson radical, then $J(R)$ is the unique maximal right $R$-ideal and $J(R) = aR$ with $a \in R^\bullet \setminus R^\times$.
  From this it follows that the right $R$-ideals are linearly ordered, and hence $R$ is rigidly factorial.
 
  If $R^\bullet$ is rigidly factorial, then it is also similarity factorial, and hence $R$ is a principal ideal ring by \ref*{p-fact:sim}.
  For every $a \in R^\bullet$, there is a unique chain of right $R$-ideals between $aR$ and $R$.
  Hence, the right $R$-ideals must be linearly ordered by inclusion.
  It follows that there exists a unique maximal right ideal of $R$.
  Thus, $R$ is local.

  If $z$,~$z'$ are two rigid factorizations of $a \in R^\bullet$ with $z \ne z'$, then the definition of the rigid distance together with cancellativity are easily seen to imply $\sd^*(z,z') \ge 2$.
  Hence, either $\sc^*(R^\bullet)=0$ or $\sc^*(R^\bullet) \ge 2$.
\end{proof}

\begin{cor} \label{cor-cat}
   Let $R$ be as in the previous proposition.
   For $\sd$ any of $\sd_{\textup{cs}}$, $\sd_{\textup{sim}}$, or $\sd^*$,
   \[
   \sc_\sd(R^\bullet) = \max\big\{ \sc_p\big(\cB(\cgrp{R})\big), 2 \big\},
   \]
   unless $R^\bullet$ is $\sd$-factorial.
\end{cor}

\begin{proof}
  By \cref{thm-transfer} we have $\sc_\sd(R^\bullet)  \le \max\{ \sc_p(\cB(\cgrp{R})), 2 \}$.
  Since two atoms $u$ and $v$ with $(R/uR)=(R/vR)$ are mapped to the same element under $\overline\theta$,
  \[
  \sc_\sd(R^\bullet) \ge \sc_{\textup{cs}}(R^\bullet) \ge \sc_p\big(\cB(\cgrp{R})\big).
  \]
  Finally, if $R^\bullet$ is not $\sd$-factorial, then $\sc_\sd(R^\bullet) \ge 2$ by \cref{p-fact}.
\end{proof}

\section{Bounded Dedekind prime rings}
\label{sec-dedekind}

We now restrict to the case in which $R$ is a bounded Dedekind prime ring.
In \cite{smertnig13} and \cite[\S7]{baeth-smertnig15}, it was shown that, under a sufficient condition, arithmetical maximal orders possess a transfer homomorphism to a monoid of zero-sum sequences.
This result can be applied to bounded Dedekind prime rings to yield \cref{t-dedekind-straight} below.
In this short section, we will see that the transfer homomorphism obtained in this way is the same as the one obtained in the previous section.

Let $\alpha$ denote the set of maximal orders in the quotient ring $\quo(R)$ that are equivalent to $R$.
Let $S \in \alpha$.
If $I$ is a (fractional) right $S$-ideal, then
\[
T=\cO_l(I)=\{\, q \in \quo(R) \mid qI \subset I \,\} \cong \End(I_S)
\]
is a maximal order equivalent to $S$; hence $T \in \alpha$ and $I$ is a (fractional) left $T$-ideal.
We say that $I$ is a \emph{\textup{(}fractional\textup{)} $(T,S)$-ideal} if it is a (fractional) right $S$-ideal and a (fractional) left $T$-ideal.
For every two maximal orders $S$,~$T \in \alpha$, there exists a $(T,S)$-ideal.

Let $G$ be the Brandt groupoid of left- or right-$S$-ideals, where $S \in \alpha$.
Thus $G$ is a small category with set of objects $\alpha$.
For $S$,~$T \in \alpha$, the morphisms from $S$ to $T$ are the fractional $(T,S)$-ideals, with composition given by multiplication of fractional one-sided ideals.
Every such fractional $(T,S)$-ideal $I$ is invertible (that is, an isomorphism); its inverse is
\[
I^{-1} = \{\, x \in \quo(R) \mid xI \subset S \,\} = \{\, x \in \quo(R) \mid IxI \subset I \,\} = \{\, x \in \quo(R) \mid Ix \subset T\,\}.
\]

For each $S \in \alpha$, the subgroup $G(S)$ of (two-sided) fractional $S$-ideals forms a free abelian group with basis the maximal $S$-ideals.
For $S$,~$T \in \alpha$ there exists a canonical isomorphism between $G(S)$ and $G(T)$: If $X$ is a fractional $(T,S)$-ideal, then $G(S) \to G(T)$, $I \mapsto XIX^{-1}$ is an isomorphism which is independent of $X$.
We identify all these groups by means of these canonical isomorphisms, denote the resulting group by $\bG$, and, for an $S$-ideal $I$, denote by $(I)$ its representative in $\bG$.

If $M$ is a maximal right $S$-ideal, and $\fP$ is the maximal $S$-ideal contained in $M$, we set $\eta(M) = (\fP) \in \bG$.
The map $\eta$ extends multiplicatively to a homomorphism $\eta\colon G \to \bG$, which we call the \emph{abstract norm}.
Let
\[
P_{R^\bullet} = \{\, \eta(Rq) \mid q \in \quo(R)^\times  \,\},
\]
let $C = \bG/P_{R^\bullet}$, and let
\[
C_M = \{\, [\eta(I)] \mid \text{$I$ is a maximal right $S$-ideal, $S \in \alpha$} \,\}.
\]

Straightforward application of the abstract results from \cite{smertnig13,baeth-smertnig15} to bounded Dedekind prime rings yields the following.
\begin{thm}[{\cite[Corollary 7.11]{baeth-smertnig15}}] \label{t-dedekind-straight}
  Let $R$ be a bounded Dedekind prime ring.
  Assume that a \textup{(}fractional\textup{)} right $R$-ideal $I$ is principal if and only if $\eta(I) \in P_{R^\bullet}$.
  Then there exists a transfer homomorphism $\theta\colon R^\bullet \to \cB(C_M)$.
  If $\sd$ is a distance on $R^\bullet$ that is invariant under conjugation by normalizing elements, then $\sc_\sd(R^\bullet,\theta) \le 2$.
\end{thm}

To make effective use of this theorem, it is necessary to
\begin{enumerate*}
\item express $C$ and $C_M$ in terms of more familiar algebraic invariants, and
\item understand the meaning of the condition appearing in the theorem.
\end{enumerate*}

In the special case where $R$ is a classical maximal order (over a holomorphy ring) in a central simple algebra over a global field, $C=C_M$ is a ray class group, which is isomorphic to $\cgrp{R}\cong \pcg{R}$, and the condition in the theorem can be expressed as: every stably free right $R$-ideal is free. (See \cite{smertnig13}.)
We now show that a similar characterization (with $C\cong \cgrp{R}$) holds in arbitrary bounded Dedekind prime rings.

\begin{lemma} \label{maxe}
  Let $I$ and $J$ be maximal right $R$-ideals.
  The following statements are equivalent:
  \begin{equivenumerate}
  \item \label{maxe:norm}$\eta(I) = \eta(J)$,
  \item \label{maxe:ann} $\ann_R(R/I) = \ann_R(R/J)$,
  \item \label{maxe:quoiso} $R/I \cong R/J$.
  \end{equivenumerate}
\end{lemma}

\begin{proof}
  \ref*{maxe:norm}${}\Leftrightarrow{}$\ref*{maxe:ann}: The annihilator of $R/I$ is the maximal $R$-ideal contained in $I$.

  \ref*{maxe:ann}${}\Rightarrow{}$\ref*{maxe:quoiso}:
  Let $\fP = \ann_R(R/I)$.
  Then $R/I$ and $R/J$ are simple modules over the simple Artinian ring $R/\fP$.
  Hence $R/I \cong R/J$ over $R/\fP$ and thus also over $R$.

  \ref*{maxe:quoiso}${}\Rightarrow{}$\ref*{maxe:ann}: Clear.
\end{proof}

As a generalization of the equivalence of \ref*{maxe:norm} and \ref*{maxe:quoiso} we immediately obtain the following.

\begin{lemma} \label{l-dedekind-classiso}
  There exists an isomorphism $K_0\modfl(R) \to \bG$ which maps $(V)$ to $(\ann(V))$ if $V$ is a simple module.
  This induces an isomorphism
  \[
  \cgrp{R}=K_0\modfl(R)/\qcp{R} \cong \bG/P_{R^\bullet} = C.
  \]
\end{lemma}

\begin{proof}
  We have $K_0\modfl(R)=\fagrp{\simple{R}}$, and $\bG$ is isomorphic to the free abelian group on maximal $R$-ideals.
  By \cref{maxe}, there is a bijection between the set of isomorphism classes of simple $R$-modules and the set of maximal $R$-ideals.
  Hence, there exists an isomorphism $\varphi\colon K_0\modfl(R) \to \bG$ as claimed.

  If $I$ is a right $R$-ideal, then $\varphi((R/I)) = \eta(I)$.
  In particular $\varphi((R/aR)) = \eta(aR)$ for all $a \in R^\bullet$.
  Since $\qcp{R}$ is the quotient group of $\{\, (R/aR) \mid a \in R^\bullet \,\} \subset K_0\modfl(R)$ and $P_{R^\bullet}$ is the quotient group of $\{\, aR \mid a \in R^\bullet \,\} \subset \bG$, the group $\qcp{R}$ is mapped bijectively to $P_{R^\bullet}$ by $\varphi$.
\end{proof}

Thus, we see that, in the case of bounded Dedekind prime rings, the transfer homomorphism from \cite{smertnig13,baeth-smertnig15} is the same as $\overline\theta$, constructed in \cref{thm-transfer}.
Moreover, for a (fractional) right $R$-ideal $I$ we have $\eta(I) \in P_{R^\bullet}$ if and only if $\stzcls{I}=\stzcls{R}$, that is, $I$ is stably free.

\begin{remark}
  \begin{enumerate}
  \item
    If $R$ is a commutative Dedekind domain, then the groupoid $G$ has a single object, $R$, and its endomorphisms are the fractional $R$-ideals.
    Thus, $G$ is simply the group of fractional $R$-ideals (that is, the nonzero fractional ideals of $R$).
    Since every left or right $R$-ideal is two-sided, $\eta$ is also trivial and we may identify $G = \bG$.
    The group $C=\bG/P_{R^\bullet}$ is therefore the ideal class group of $R$ as it is traditionally defined.
    \cref{l-dedekind-classiso} shows that $C$ coincides with $\cgrp{R}$.
    From \cref{t-k0-iso-pcg} we recover the well-known statement $K_0(R) \cong C \oplus \bZ$.

  \item
    Under the isomorphism $\cgrp{R} \cong \bG/P_{R^\bullet}$, the expression $\langle I / J \rangle$, used in \cref{sec-class-groups}, corresponds to $[\eta(JI^{-1})]$.
    Since $\eta$ is a homomorphism, this gives another way of justifying \cref{l-symprop} in the case of bounded Dedekind prime rings.
  \end{enumerate}
\end{remark}

\section{Some examples}
\label{sec-examples}

In this final section, we give some examples complementing the main results.
In \cref{sec-transfer-hom} we have made use of two sufficient conditions \labelcref{c-ft,c-fo}, which are always satisfied if $R$ is bounded.
Likely these conditions are not necessary for the existence of a transfer homomorphism.
However, based on the proof, it seems quite natural to impose condition \labelcref{c-fo}, as the existence of non-trivial extensions between simple modules of towers of different classes presents an obvious obstacle to the construction of a transfer homomorphism.
Condition \labelcref{c-ft} seems less natural at first.
However, in \cref{p-nonhf,e-no-ft}, we show that an HNP ring with a single non-trivial tower, which is faithful of length $2$, need not be half-factorial, even if every stably free right ideal is free and $\pcg{R}=\cgrp{R}=\mathbf 0$.
In particular, for such a ring, there does not exist a (weak) transfer homomorphism from $R^\bullet$ to $\cB(\cgrp{R})$.
This shows that \labelcref{c-ft} cannot simply be dispensed with.

If $H$ is an atomic monoid, recall that $\rho_2(H)$ is the supremum over all $k \in \bN$ such that there exist atoms $u_1$, $u_2$, $v_1$, $\ldots\,$,~$v_k \in H$ with
\[
u_1u_2 = v_1\cdots v_k.
\]
Clearly, if $\rho_2(H) > 2$, then $H$ is not half-factorial.

Let $K$ be a field of characteristic $0$, and let $A=A_1(K)=K[x][y;\frac{d}{dx}]$ be the first Weyl algebra over $K$.
That is, $A$ is a $K$-algebra generated by $x$ and $y$ subject to $xy-yx=1$.
Then $A$ is a simple Dedekind domain, in other words, a simple HNP domain all of whose towers are trivial.
It is well known that $\pcg{A}$ is trivial, but $A$ has non-free stably free right $A$-ideals.
In terms of factorizations, this exhibits itself in $A^\bullet$ not being half-factorial; the well known example
\begin{equation}\label{eq-weyl}
x^2y=(1+xy)x,
\end{equation}
with all the factors being atoms of $A^\bullet$, shows $\rho_2(A^\bullet) \ge 3$.
However, the matrix ring $M_2(A)$ also has $\pcg{M_2(A)}=\mathbf 0$ by Morita equivalence and, since $\udim(M_2(A))=2$, every stably free right $M_2(A)$-ideal is free.
Therefore $M_2(A)$ is a principal ideal ring, and hence $M_2(A)^\bullet$ is similarity factorial.

The computations for the following examples are given at the end of the section.

\begin{exm} \label{e-emb-factor}
  Embedding the example from \cref{eq-weyl} into $M_2(A)^\bullet$ by writing the elements into the first coordinate, we see that $1+xy$ factors as a product of two atoms
  \[
  \begin{bmatrix}
    1+xy & 0 \\
    0      & 1
  \end{bmatrix}
  =
  \begin{bmatrix}
    x^2 & 1+xy \\
    x  & y
  \end{bmatrix}
  \begin{bmatrix}
    -y^2  &  y \\
    xy+1  & -x
  \end{bmatrix}.
  \]
\end{exm}

\begin{exm} \label{e-no-f2s-no-transfer}
  The module $A/x(x-y)A$ is uniserial with unique composition series induced from $A \supset xA \supset x(x-y)A$.
  Hence, the element
  \[
  \begin{bmatrix} x(x-y) & 0 \\
    0                    & 1 \\
  \end{bmatrix} \in M_2(A)^\bullet
  \]
  has only the obvious rigid factorization.
  Since $A/xA \ncong A/(x-y)A$, this shows that $\theta\colon M_2(A)^\bullet \to \cP(M_2(A))$ is not a transfer homomorphism.
  However, $\overline\theta\colon M_2(A)^\bullet \to \cB(\mathbf 0) \cong (\bN_0,+)$ is a transfer homomorphism by \cref{thm-transfer}.
  (This is a degenerate trivial case; every half-factorial monoid has a transfer homomorphism to $\bN_0$ given by the length function.)
\end{exm}

\begin{prop} \label{p-nonhf}
  Let $R$ be an HNP ring such that every stably free right $R$-ideal is free.
  Suppose that $T$ and $T'$ are faithful towers, at least one of which is non-trivial, and suppose that $\langle T \rangle=\langle T'\rangle = \mathbf 0$.
  If $\Ext^1_R(\tbottom T,\ttop T')\ne \mathbf 0$ and $\Ext^1_R(\tbottom T',\ttop T) \ne \mathbf 0$, then $\rho_2(R^\bullet) \ge 3$.

  In particular, if also $\card{\pcg{R}}\le 2$, there exists no transfer homomorphism to a monoid of zero-sum sequences over a subset of $\pcg{R}$.
\end{prop}

\begin{proof}
  Let $T$ be represented by simple modules $W_1$, $\ldots\,$,~$W_m$, with $W_i$ an unfaithful successor of $W_{i-1}$ for all $i \in [2,m]$,
  and let $T'$ be similarly represented by simple modules $W_1'$, $\ldots\,$,~$W_n'$.
  Suppose without restriction $m \ge 2$.
  By our assumption $\Ext_R^1(W_m,W_1') \ne \mathbf 0$ and $\Ext_R^1(W_n',W_1) \ne \mathbf 0$.
  By \cite[Corollary 16.3]{levy-robson11}, there exist uniserial modules $U$ and $U'$ whose composition factors, from top to bottom, are $W_2$, $\ldots\,$,~$W_m$, $W_1'$, $\ldots\,$,~$W_n'$,~$W_1$, respectively, $W_1$, $\ldots\,$,~$W_m$.
  Uniserial modules are cyclic, hence $U \cong R/I$ and $U'\cong R/J$ for right $R$-ideals $I$ and $J$.
  Since the top composition factors of $U$ and $U'$ differ, we must have $I+J=R$.
  Thus, $R/(I \cap J) \cong U \oplus U'$.
  The right $R$-ideals $I$, $J$, and $I \cap J$ are principal, say $I \cap J = aR$ for some $a \in R^\bullet$.
  The chain $R \supset I \supset I \cap J$ gives a factorization of $a$ of length $2$.
  On the other hand, there exist right $R$-ideals $I \cap J \subset L \subset K \subset R$ such that $(R/K) = (T)$, $(K/L)=(T')$ and $(L/I \cap J)=(T)$.
  This yields a factorization of $a$ of length $3$.
\end{proof}

Let $R = \bI_A(xA) = K + xA$ be the idealizer of the maximal right $A$-ideal $xA$.
Then $R$ is an HNP domain with one non-trivial faithful tower of length $2$ and all other towers are trivial faithful towers.
In $M_2(R)$ again every stably free right $M_2(R)$-ideal is free, because $\udim(M_2(R))=2$, and $\pcg{M_2(R)}=\pcg{R}=\pcg{A}=\mathbf 0$.
Thus, \labelcref{c-fo} holds, while \labelcref{c-ft} is violated.
We note, but will not use, that $M_2(R)$ is itself the idealizer of the right ideal $M_2(xA)$ in $M_2(A)$.

Over $A$, there exist simple modules $U$ and $V$ such that $\Ext^1_A(U,V) \ne \mathbf 0$ and $\Ext^1_A(V,U) \ne \mathbf 0$.
For instance, this is the case for $U=A/xA$ and $V=A/(x-y)A$ by \cite[Corollary 5.8]{mcconnell-robson73}.
Over the idealizer $R=\bI(xA)$, the simple module $U$ is a uniserial module of length two, with top composition factor $W_1 \cong A/R$ and bottom composition factor $W_2 \cong R/xA$.
The module $V$ remains a simple module over $R$, with $\Ext^1_R(W_2,V) \ne \mathbf 0$ and $\Ext^1_R(V,W_1) \ne \mathbf 0$ (see \cite[\S5]{levy-robson11}).

In the ring $M_2(R)$ the situation remains the same by Morita equivalence, except that now also every stably free right $R$-ideal is free.
Thus $\rho_2(M_2(R)^\bullet) \ge 3$ by the preceding proposition.
We now give an explicit example of this.
The computation, which is based on the module-theoretic reasoning, is sketched below.

\begin{exm} \label{e-no-ft}
The ring $M_2(R)$ is not half-factorial; indeed for
\[
a=\begin{bmatrix}
  x(x-y)(x-yx) & x(x-y)(-xy+xy^2) \\
  x^2 - (1+xy)x & (1+xy)(1-x)+x^2y^2
  \end{bmatrix}
\]
we have
{\footnotesize
\begin{align*}
a&=\underbrace{\begin{bmatrix}
  x(x-y) & 0 \\
  0 & 1\\
\end{bmatrix}}_{u_1}
\underbrace{\begin{bmatrix}
  x-yx & -xy+xy^2 \\
  x^2 - (1+xy)x & (1+xy)(1-x)+x^2y^2
\end{bmatrix}}_{u_2}\\
&=\underbrace{\begin{bmatrix}
x & xy\\
x & 1+xy
\end{bmatrix}}_{w_1}
\underbrace{\begin{bmatrix}
-xy^2 + x^2y - xy - x + 1 &     -xy^3 + x^2y^2 - xy^2 - xy \\
xy - x^2 + x             & xy^2 - x^2y + xy +1
\end{bmatrix}}_{w_2}
\underbrace{\begin{bmatrix}
x & -xy \\
-x & 1+xy
\end{bmatrix}}_{w_3}
\end{align*}}%
where $u_1$, $u_2$, $w_1$, $w_2$, and $w_3$ are atoms of $M_2(R)^\bullet$.
While $w_2$ is somewhat unwieldy, we have
\[
w_1w_2 = \begin{bmatrix} x(x-y) & x(x-y)y \\
x & 1 + xy.
\end{bmatrix}
\]
Note that $u_1$ is not an atom in the bigger ring $M_2(A)$.
In $M_2(A)$ all factorizations of $a$ have length $3$.
\end{exm}

\begin{proof}[Outline of computations for \cref{e-emb-factor,e-no-ft}]
  Under the Morita correspondence, the lattice of right submodules of $R_R^2$ is isomorphic to the lattice of right ideals of $M_2(R)$.
  Under this correspondence, a module $M \subset R_R^2$ is mapped to the right ideal of $M_2(R)$ consisting of matrices all of whose columns are elements of $M$.
  In particular, if $M$ is free of uniform dimension $2$ with basis $(x_1,x_2)$, $(y_1,y_2) \in R_R^2$, the corresponding right ideal is the principal right $M_2(R)$-ideal generated by
  \[
  \begin{bmatrix} x_1 & y_1 \\ x_2 & y_2 \end{bmatrix}.
  \]
  Factorizations of an element $a \in M_2(R)^\bullet$ correspond to finite maximal chains of principal right ideals between $aM_2(R)$ and $M_2(R)$.
  Thus, if $A$ denotes the submodule of $R_R^2$ corresponding to $aM_2(R)^\bullet$, factorizations of $a$ correspond to finite maximal chains of free submodules between $A$ and $R_R^2$.
  It is this point of view that we will use to construct the desired examples.
  The main work lies in computing explicit bases for certain free modules.

  The module $(A/x(x-y)A)_A$ is uniserial by \cite[Corollary 5.10]{mcconnell-robson73}.
  The ring $R=K + xA$ is the idealizer ring of the maximal right ideal $xA$ of $A$.
  Hence, $(A/xA)_R$ is a uniserial module of length $2$, with unique nonzero proper submodule $(R/xA)_R$.

  \newcommand{\savedparindent}{\parindent}

  \begin{enumerate}[align=left,leftmargin=0pt,itemindent=0pt,labelwidth=!,itemsep=0.25em,topsep=0.25em,label=\emph{Step \arabic*:}]
  \item \emph{Find a right $R$-ideal $I_R$ with $R_R/I_R \cong (A/xA)_R$.}
    Since $(A/xA)_R$ is uniserial, any element of $(A/xA)_R$ not contained in $(R/xA)_R$ generates $(A/xA)_R$.
    Thus
    \[
    \varphi\colon R_R \to (A/xA)_R,\quad r \mapsto yr + xA
    \]
    is an epimorphism.
    If $r=\lambda + xf$ with $\lambda \in K$ and $f \in A$, then
    \[
    \varphi(\lambda + xf) = y\lambda + xyf - f + xA = y\lambda - f + xA.
    \]
    Thus $x^2A \subset \ker(\varphi)$ and $1+xy \in \ker(\varphi)$, and hence $x^2A + (1+xy)R \subset \ker(\varphi)$.
    Since $(R/x^2A)_R$ has length $3$, the module $(A/xA)_R$ has length $2$, and $1+xy \not\in x^2A$, equality holds.
    Moreover, also $\ker(\varphi)=x^2R + (1+xy)R$, since $x^2y = x(1+yx)=(1+xy)x$.
    We set
    \[
    I = \ker(\varphi) = x^2A + (1+xy)R = x^2R + (1+xy)R.
    \]

  \item\emph{Find a basis of $R_R \oplus I_R$.}
    We know that $I_R$ is stably free, and need to exhibit an explicit isomorphism $R_R^2 \cong R_R \oplus I_R$.
    For this, we first seek an isomorphism $A_A^2 \cong A_A \oplus IA_A$.

    \begin{enumerate}[align=left,leftmargin=0pt,itemindent=0pt,labelwidth=!,itemsep=0.25em,topsep=0.25em,listparindent=\savedparindent,label=\emph{Step \arabic{enumi}\alph*:}]
    \item\emph{Find a basis of $A_A \oplus IA_A$.}
      Let $\pi\colon A_A^2 \to x^2A + (1+xy)A$ be the epimorphism satisfying $\pi(f,g)=x^2f + (1+xy)g$.
      Since $x^2A \cap (1+xy)A = x^2yA = (1+xy)xA$, there is a short exact sequence
      \[
      \begin{tikzcd}
        \mathbf 0 \ar[r] & A_A \ar[r]  & A_A^2 \ar[r,"\pi"] & IA_A = x^2A + (1+xy)A \ar[r] & \mathbf 0,
      \end{tikzcd}
      \]
      where the homomorphism on the left is given by $f \mapsto (yf,-xf)$.
      Since $A_A$ is hereditary, $IA_A$ is projective and hence the sequence splits.
      If $\varepsilon$ is a right inverse of $\pi$, then $\varepsilon(x^2)=(1,0) + (y,-x)f = (1+yf,-xf)$ and $\varepsilon(1+xy)=(0,1)+(y,-x)g=(yg,1-xg)$ for some $f$,~$g \in A$.
      Since $x^2y = (1+xy)x$, we must have $(y+yfy,-xfy)=(ygx,x-xgx)$, which is equivalent to $1+fy=gx$.
      This equation is solved by $f=-x$ and $g=-y$, and one checks that indeed
      \[
      \varepsilon\colon
      \begin{cases}
        IA_A & \to A^2_A \\
        x^2  & \mapsto (1,0) + (y,-x)(-x) = (1-yx, x^2) \\
        1+xy & \mapsto (0,1) + (y,-x)(-y) = (-y^2, 1+xy)
      \end{cases}
      \]
      is a homomorphism with $\pi\circ \varepsilon = \id_{IA_A}$.
      Thus, we have an isomorphism
      \[
      \kappa^{-1}\colon A_A \oplus IA_A \isomto A_A^2, \quad (f,g) \mapsto (yf,-xf) + \varepsilon(g).
      \]
      Computing the preimages of $(1,0)$ and $(0,1)$ under this isomorphism, it follows that the pair $\kappa(1,0)=(x,x^2)$, $\kappa(0,1)=(y,1+xy)$ is a basis of $A_A \oplus I A_A$.

      The basis gives us the left factor in \cref{e-emb-factor}.
      The corresponding right factor can be found by solving a linear system.
      To finish this example, it remains to see that the two factors are atoms.
      We already know that $(A/xA)_A$ has length $2$ and that $(A/(x^2A\cap (1+xy)A))_A$ has length $3$ due to $x^2A\cap (1+xy)A = x^2yA=(1+xy)xA$.
      It follows that the length of $(A/(1+xyA))_A$ is $2$ and that $1+xyA \subsetneq IA \subsetneq A$.
      Hence, $(A/IA)_A$ and $(IA/(1+xy)A)_A$ are simple, and the computed factors are indeed atoms.

    \item\emph{Modify the basis of $A_A \oplus IA_A$ so that it is also one of $R_R \oplus I_R$.}
      Since $y \not\in R$, the computed basis of $A_A \oplus IA_A$ is not one of $R_R\oplus I_R$.
      To find a basis of $R_R\oplus I_R$ from the one for $A_A \oplus IA_A$, we follow the steps of the proof of the Descent Theorem (see \cite[Theorem 34.4]{levy-robson11}) in its most basic case.
      Consider the diagram
      \[
      \begin{tikzcd}
        A_A \oplus IA_A \ar[r,"\alpha"] & (A/xA)_A^2 \ar[r] & \mathbf 0 \\
        A_A \oplus A_A \ar[r,"\beta"] \ar[u,"\kappa","\rotatebox{90}{\(\sim\)}"'] & (A/xA)_A^2 \ar[r] \ar[u,equal] & \mathbf 0 \\
        A_A \oplus A_A \ar[r,"\gamma"] \ar[u,dashed,"\tau","\rotatebox{90}{\(\sim\)}"'] & (A/xA)_A^2 \ar[r] \ar[u,equal] & \mathbf 0,
      \end{tikzcd}
      \]
      where $\alpha(f,g) = (f+xA,g+xA)$, the homomorphism $\beta$ is chosen so that the upper part of the diagram commutes, that is, $\beta(1,0)=(0,0)$ and $\beta(0,1) = (y+xA, 1+xA)$, and $\gamma$ satisfies $\gamma(1,0)=(1+xA,0)$ and $\gamma(0,1)=(0,1+xA)$.
      Our goal is to find the indicated isomorphism $\tau$ that makes the lower diagram commute.
      Then, since $R_R \oplus R_R$ is the preimage of the socle of $(A/xA)_A^2$ under $\gamma$, and $R_R \oplus I_R$ is the preimage of the socle of $(A/xA)_A^2$ under $\alpha$, the isomorphism $\kappa \circ \tau$ restricts to an isomorphism $R_R \oplus R_R \to R_R \oplus I_R$.

      Computing $\tau$ amounts to straightening $\gamma$ as in \cite[Lemma 49.5]{levy-robson11}.
      First note that $\gamma(y,1) = \beta(0,1)$.
      Thus $(y,1)$, $(1,0)$ is a basis of $A_A^2$ for which $\gamma((y,1)A)=(A/xA)_A^2$, and $\gamma|_{(y,1)A}$ is equivalent to $\beta|_{(0,1)A}$.
      From the relation $xy-yx=1$ it is immediate that a preimage of $\gamma(1,0)=(1+xA,0)$ under $\gamma|_{(y,1)A}$ is $(-yx,-x)$.
      Define $\psi\colon (1,0)A \to (y,1)A$ by $(1,0) \mapsto (yx,x)$.
      Then $\gamma|_{(1,0)A} = \gamma|_{(y,1)A} \circ (-\psi)$.
      Thus $\gamma(1+yx,x) = \gamma((1,0) + \psi(1,0)) = 0$.
      It follows that $(y,1)$, $(1+yx,x)$ is the desired basis of $A_A^2$.

      We define $\tau$ by $\tau^{-1}(1,0) = (1+yx, x)$ and $\tau^{-1}(0,1) = (y,1)$.
      Then $\tau(1,0) = (1,-x)$ and $\tau(0,1)=(-y,1+xy)$.
      Consequently, $\kappa\circ\tau(1,0) = (x-yx, x^2-(1+xy)x)$ and $\kappa\circ\tau(0,1) = (-xy+xy^2, (1+xy)(1-x) + x^2y^2)$, and these vectors constitute a basis of $R_R \oplus I_R$.
      Subsequently, a basis for $x(x-y)R_R \oplus I_R = (R_R \oplus I_R) \cap (x(x-y)R_R \oplus R_R)$ is obtained by multiplying the first coordinate of these basis vectors by $x(x-y)$ from the left.
      Under Morita equivalence, this corresponds to the matrix $a$ in \cref{e-no-ft}.
    \end{enumerate}

    The left factor $u_1$ of $a$ is immediate from the construction.
    The cofactor $u_2$ can be found by solving a linear system.
    The representation $a=u_1u_2$ corresponds to the chain of submodules on the left hand side of the diagram in \cref{fig-no-ft}.
    It is clear that $u_2$ is an atom, since $(M_2(R)/u_2M_2(R))_{M_2(R)}$, which corresponds to $(R/I)_R \cong (A/xA)_R$, consists of a single tower.
    To see that $u_1$ is an atom, we show that $(R_R \oplus R_R)/(x(x-y)R_R \oplus R_R)$ is uniserial.

    \begin{figure}
      \begin{tikzcd}[every arrow/.append style={dash},column sep={3.2cm,between origins}]
        & R \oplus R & \\
        xA \oplus R \ar[ru,"W_2"] & & R\oplus J \ar[lu,"W_1"'] \\
        x(x-y)A \oplus R \ar[u,"V"] & X \ar[lu,"W_1"'] \ar[ru] & R \oplus I \ar[u,"W_2"'] \\
        x(x-y)R \oplus R \ar[u,"W_1"] & Y \ar[u,"V"'] \ar[lu] \ar[ruu] & \\
        & & \\
        & x(x-y)R \oplus I \ar[luu,"(R/I)_R\cong(A/xA)_R"] \ar[ruuu] \ar[uu,"W_1\oplus W_2"',very near end] & \\
      \end{tikzcd}
      \caption{Modules involved in the construction of \cref{e-no-ft}.}
      \label{fig-no-ft}
    \end{figure}
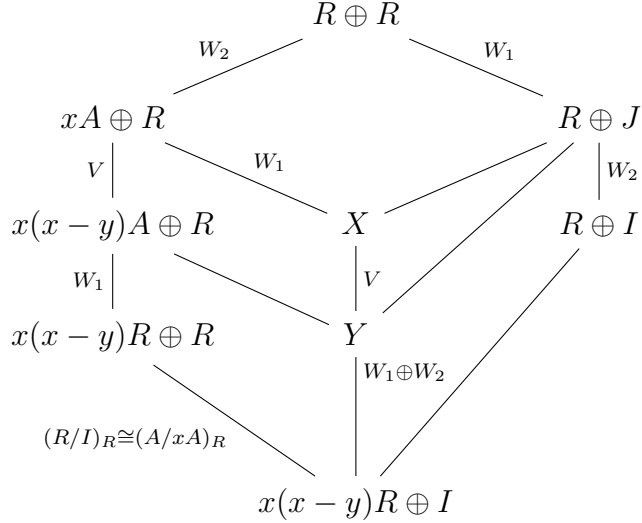

  \item\emph{The module $(R_R \oplus R_R)/(x(x-y)R_R \oplus R_R)$ is uniserial.}
    It suffices to show that $(R/x(x-y)R)_R$ is uniserial.
    By \cite[Lemma 16.1]{levy-robson11}, it suffices to show that the two length $2$ subfactors $(R/x(x-y)A)_R$ and $(A/(x-y)R)_R \cong (xA/x(x-y)R)_R$ are both uniserial, that is, non-split.
    The $A$-modules $(A/x(x-y)A)_A$ and $(A/(x-y)xA)_A$ are uniserial by \cite[Corollary 5.10]{mcconnell-robson73}.
    Suppose $(R/x(x-y)A)_R$ is split.
    Then $(R/x(x-y)A)_R \cong (W_2)_R \oplus V_R$, and hence $(R/x(x-y)A)_R \otimes {}_R A_A \cong (A/xA)_A \oplus V_A$.
    However, $(R/x(x-y)A)_R \otimes {}_R A_A \cong (A/x(x-y)A)_A$ since $R_R \otimes {}_R A_A \cong A_A$ by multiplication and this isomorphism carries $x(x-y)A_R \otimes {}_R A_A$ to $x(x-y)A_A$.
    (Recall that ${}_R A$ is finitely generated projective, hence flat.)
    This contradicts the fact that $(A/x(x-y)A)_A$ is uniserial.

    Similarly, $(A/(x-y)R)_R \otimes {}_R xA_A \cong (A/(x-y)xA)_A$ since $A_R \otimes {}_R xA_A \cong A_A$ by multiplication and this isomorphism carries the submodule $(x-y)R_R \otimes {}_R xA_A$ to $(x-y)RxA_A=(x-y)xA_A$.
    Thus, if $(A/(x-y)R)_R$ were split, then $(A/(x-y)R)_R \cong V_R \oplus (W_1)_R$ and hence $(A/(x-y)xA)_A \cong (A/(x-y)R)_R \otimes {}_R xA_A \cong V_A \oplus (A/xA)_A$.
    This contradicts the fact that $A/(x-y)xA$ is uniserial.

  \item\emph{Find submodules for the factorization of length $3$.}
    We still have to obtain the second factorization, which will have length $3$.
    The module $(R/xA)_R$ is the unique nonzero proper submodule of $(A/xA)_R$.
    Hence, the unique right $R$-ideal properly contained between $I_R$ and $R_R$, which we denote by $J_R$, can be found by taking a preimage of $1+xA$ under $\varphi\colon R_R \to (A/xA)_R, 1 \mapsto y+xA$.
    The relation $xy-yx=1$ suggests to try $\varphi(-x) = -yx+xA = 1-xy + xA = 1 + xA$.
    Thus
    \[
    J = -xR + I = xR + I = xR + (1+xy)R.
    \]
    Set
    \[
    X_R = (xA_R \oplus R_R) \cap (R_R \oplus (xR + (1+xy)R))_R = xA_R \oplus (xR + (1+xy)R)_R.
    \]
    Then $((R_R \oplus R_R)/X)_R \cong (W_1)_R \oplus (W_2)_R$, and hence $X_R$ is free.
    Similarly, set
    \[
    Y_R = x(x-y)A_R \oplus (xR + (1+xy)R)_R.
    \]
    Since $(X/Y)_R \cong V_R$, also $Y_R$ is free and the chain $R_R \oplus R_R \supset X_R \supset Y_R \supset x(x-y)R_R \oplus I_R$ cannot be refined any further with free modules.
    It remains to find bases of $X_R$ and $Y_R$.

  \item\emph{Find a basis of $A_R \oplus J_R$.}
    Since $JA=xA + (1+xy)A=A$, it is trivial to find a basis of $A_A \oplus JA_A$.
    As before, we follow the idea of the Descent Theorem to modify this basis so that it is also one of $R_R \oplus J_R$.
    We know $\rho(R_R,W_2)=1$ since $(R/xA)_R\cong W_2$.
    Hence $\rho(J_R,W_2)=2$ by \cref{l-rank}, that is, $(J/JxA)_R \cong (R/xA)_R^2$.
    Note that $JxA=xRxA + (1+xy)RxA = x^2A + (1+xy)xA$.
    Thus, there is a homomorphism
    \[
    \begin{cases}
      (R/xA)_R^2 &\to (J/JxA)_R, \\
      (1+xA,0) &\mapsto x + JxA, \\
      (0,1+xA) &\mapsto 1+xy + JxA, \\
    \end{cases}
    \]
    which is clearly surjective and hence an isomorphism.
    Its inverse induces an isomorphism $(A_R \oplus J_R)/ (A_R \oplus JxA_R) \to (R/xA)_R^2$.
    Extending to $A$-modules, we obtain a homomorphism $\beta\colon A_A^2 \to (A/xA)_A^2$ which satisfies $\beta(1,0)=(0,0)$, $\beta(0,x)=(1+xA,0)$, and $\beta(0,1+xy)=(0,1+xA)$.
    Since $1 = x(-y) + (1+xy)$, this means $\beta(0,1) = (-y+xA, 1+xA)$.

    As before, we need to find an isomorphism $\tau$ that makes the diagram
    \[
    \begin{tikzcd}
      A_A \oplus JA_A = A_A \oplus A_A \ar[r,"\beta"]  & (A/xA)_A^2 \ar[r]  & \mathbf 0 \\
      A_A \oplus A_A \ar[r,"\gamma"] \ar[u,dashed,"\tau","\rotatebox{90}{\(\sim\)}"'] & (A/xA)_A^2 \ar[r] \ar[u,equal] & \mathbf 0,
    \end{tikzcd}
    \]
    commute, where $\gamma$ satisfies $\gamma(1,0)=(1+xA,0)$ and $\gamma(0,1) = (0,1+xA)$.
    Then $\tau$ restricts to an isomorphism $R_R \oplus R_R \to A_R \oplus J_R$, since  $R_R \oplus R_R$ is the preimage of the socle of $(A/xA)_R^2$ under $\gamma$, and $R_R \oplus J_R$ is the preimage of the socle of $(A/xA)_R^2$ under $\beta$.

    Note that $\gamma(-y,1)=(-y+xA,1+xA)=\beta(0,1)$.
    Thus, the pair $(1,0)$, $(-y,1)$ is a basis of $A_A^2$ for which $\gamma((-y,1)A)=(A/xA)^2$ and $\gamma|_{(-y,1)A}$ is equivalent to $\beta|_{(0,1)A}$.
    A preimage of $\gamma(1,0)=(1+xA,0+xA)$ under $\gamma|_{(-y,1)A}$ is given by $(-y,1)x$.
    Thus $\gamma(1+yx,-x)=0$, and $\tau$ with $\tau^{-1}(1,0) = (1+yx,-x)$ and $\tau^{-1}(0,1) = (-y,1)$ makes the diagram commute.
    It follows that $\tau(1,0) = (1,x)$ and $\tau(0,1) = (y,1+xy)$, and these two vectors constitute a basis of $A_R \oplus J_R$.
  \end{enumerate}

  Since $X_R = xA_R \oplus J_R$ and $Y_R=x(x-y)A_R \oplus J_R$, bases for $X_R$ and $Y_R$ are obtained from the one of $A_R \oplus J_R$ by multiplying the first coordinate by $x$, respectively $x(x-y)$, from the left.
  Under the Morita equivalence, this gives the elements $w_1$ and $w_1w_2$.
  The cofactors $w_2$ and $w_3$ can again be computed by solving a linear system.
  Since $(R_R \oplus R_R) / X_R \cong W_1 \oplus W_2$ and $X_R/Y_R \cong V_R$ both of $w_1$ and $w_2$ are atoms.
  Since the composition series of $Y_R/(x(x-y)R_R \oplus I_R)$ consists of a single tower, also $w_3$ is an atom.
\end{proof}

\bibliographystyle{alphaabbr}
\bibliography{all}

\end{document}